\documentclass[a4paper,12pt]{book}

\usepackage{amsmath,amssymb,latexsym,amstext,amsthm}
\usepackage{dsfont}
\usepackage{graphicx,subfig,sidecap}
\usepackage{color}
\usepackage{multicol}
\usepackage{pdfsync}
\usepackage{anyfontsize}
\usepackage{wrapfig}
\usepackage{lipsum}
\usepackage{mathrsfs}
\usepackage{makeidx}   
\makeindex             
\usepackage[colorlinks=true, pdfstartview=FitV, linkcolor=ChapUC3M,citecolor=ChapUC3M, urlcolor=ChapUC3M]{hyperref}
\usepackage{slashbox} 
\definecolor{ChapUC3M}{RGB}{0,14,120}

\definecolor{gris}{gray}{.94}

\usepackage{dotlessi}

\ifpdf\hypersetup{%
pdftitle = {Coupled systems of nonlinear Schr\"odinger and
Korteweg-de Vries equations},
pdfsubject = {Trabajo fin de M\'{a}ster},
pdfkeywords = {Schr\"odinger equation. KdV equation. Coupled nonlinear systems. Variational Methods.},
pdfauthor = {Rasiel Fabelo Berm\'udez \textcopyright\ },
}
\fi
\newcommand{\be}{\begin{equation}}
\newcommand{\ee}{\end{equation}}
\newenvironment{pf}{\noindent{\it
Proof}.\enspace}{\rule{2mm}{2mm}\medskip}
\newenvironment{pfn}{\noindent{\it Proof. }}{\rule{2mm}{2mm}\medskip}
\newenvironment{pfnt}{\noindent{\it Proof of Theorem \ref{caso N=1}. }}{\rule{2mm}{2mm}\medskip}
\newenvironment{gradientesim}{\noindent{\it Proof of Theorem \ref{th:deriv}. }}{\rule{2mm}{2mm}\medskip}
\newenvironment{pfnkondra}{\noindent{\it Proof of Theorem \ref{compact emb}. }}{\rule{2mm}{2mm}\medskip}

\newcommand{\R}{\mathbb{R}}
 
  \newcommand{\RN}{\mathbb{R}^N}

\newcommand{\N}{\mathbb{N}}
\newcommand{\E}{\mathbb{E}}

\newcommand{\h}{\mathbb{H}}
\newcommand{\dyle}{\displaystyle}
\renewcommand{\a }{\alpha }
\renewcommand{\b }{\beta }

\newcommand{\D }{\Delta }
\newcommand{\e }{\varepsilon }
\newcommand{\g }{\gamma}
\newcommand{\G }{\Gamma }
\renewcommand{\l }{\lambda }
\renewcommand{\L }{\Lambda }
\newcommand{\m }{\mu }
\newcommand{\n }{\nabla }

\renewcommand{\o }{\omega }

\newcommand{\bra}{\langle} 
\newcommand{\ket}{\rangle} 

\newcommand{\cM}{{\cal M}}

\newcommand{\cN}{{\mathcal{N}}}

\newcommand{\intN}{\int_{\R^N}}
\newcommand{\intR}{\int_\R}

\newcommand{\wt}{\widetilde}

\newcommand{\bu}{{\bf u}}
\newcommand{\bv}{{\bf v}}

\newcommand{\bw}{{\bf w}}
\newcommand{\bo}{{\bf 0}}

\newcommand{\bh}{{\bf h}}
\newcommand{\bk}{{\bf k}}

\newtheorem{Theorem}{Theorem}[section]
\newtheorem{Corollary}[Theorem]{Corollary}
\newtheorem{Lemma}[Theorem]{Lemma}
\newtheorem{Proposition}[Theorem]{Proposition}
\newtheorem{Definition}[Theorem]{Definition}
\newtheorem{remark}[Theorem]{Remark}
\newtheorem{remarks}[Theorem]{Remarks}
\newtheorem{example}[Theorem]{Example}
\newtheorem{examples}[Theorem]{Examples}

\newenvironment{Example}{\begin{example}
\rm}{\rule{2mm}{2mm}\end{example}}

\usepackage{titletoc}
\usepackage{titlesec}
\usepackage{lipsum} 

\titleformat{\chapter}[display]
 {\bfseries\LARGE}
 {\filright\MakeUppercase{ \LARGE  \chaptertitlename} \LARGE\thechapter}
 {2ex}
  {\titleline[c]{\titlerule[2pt]}\vspace{1.5ex}\filleft}
  [\vspace{0ex}{\rule{\textwidth}{2pt}}]

\titleformat{\section}[runin]
{\large \normalfont\bfseries}
{\S\ \thesection.}{.5em}{}[.]

\titleformat{\subsection}[runin]
{\large \normalfont\bfseries}
{\S\ \thesubsection.}{.5em}{}[.]

\newcommand{\rnc}{\renewcommand}

\rnc{\exp}{\mathrm{e}}

\date{}

\hypersetup{
 colorlinks=black,
 citecolor=black,
 linkcolor=black,
 urlcolor=black}

\begin{document}


\thispagestyle{empty}

\vspace*{-1cm}
 {\color{ChapUC3M}\begin{center}

 \includegraphics[width=.7\textwidth]{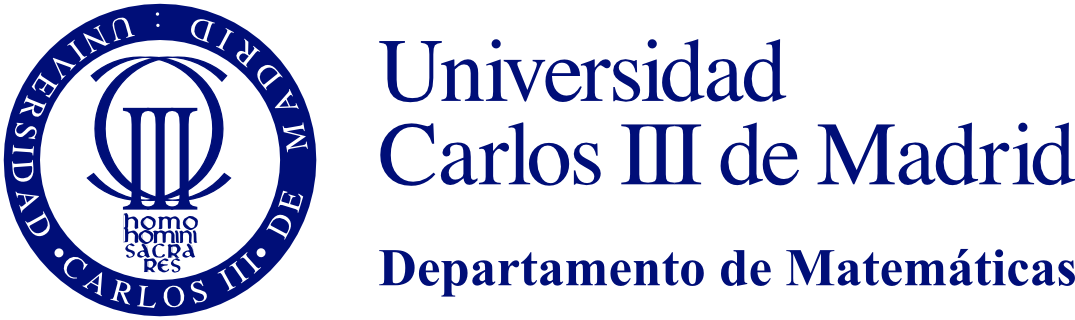}

\vspace{4cm}
\noindent {\color{ChapUC3M}{\fontsize{15}{0} \textbf{Master's Thesis} } \hfill {\rule{.7\textwidth}{3pt}}

\vspace{0.5cm}

\noindent{\fontsize{20.8}{0}\selectfont\textbf{ Coupled system of nonlinear Schr\"odinger and
Korteweg-de Vries equations}}


\vspace{0.3cm}
{\rule{\textwidth}{3pt}}}

\vspace{3cm}  {\fontsize{15}{0}  \begin{tabular}{rl}
   \textbf{Author:} & \textbf{Rasiel Fabelo Berm\'udez} \\ & \\
  \textbf{Advisors:} &  \textbf{Eduardo Colorado Heras} \\ & \textbf{Pablo \'Alvarez Caudevilla}
\end{tabular}}

\vspace{2cm}

 \noindent{\fontsize{15}{0}\selectfont\textbf{ Master in Mathematical Engineering}}

 \vspace{3cm} \hfill { \fontsize{15}{0}   \textsl{Legan\'{e}s,  September 2016}}
\end{center}}

\clearpage
\
\thispagestyle{empty}
\newpage

\thispagestyle{empty}

\vspace*{4cm}

\begin{flushright}
\emph{To my mother,\\
my father and my sister.}
\end{flushright}

\clearpage
\
\thispagestyle{empty}

\newpage
\pagenumbering{Roman} \setcounter{page}{1}


\begin{center}\vspace{5mm}{\Large\bf Abstract}\end{center}

This work is divided into two parts. First, we analyze the existence of positive bound and ground states for a second order stationary system coming from a coupled system of nonlinear Schr\"odinger--Korteweg-de Vries equations. Second, we extend these results for a higher order system of nonlinear Schr\"odinger--Korteweg-de Vries equations. Looking for  ``standing-traveling'' waves we arrive at a bi-harmonic stationary system, for which we prove the existence and multiplicity of solutions under appropriate conditions on the parameters.

\begin{center}\vspace{5mm}{\Large\bf Resumen}\end{center}

Este trabajo est\'a dividido en dos partes. Primero, se analiza la existencia de soluciones de un sistema estacionario de segundo orden que proviene de un sistema no linear tipo Schr\"odinger--Korteweg-de Vries. La otra parte del trabajo est\'a dedicada al estudio un sistema no lineal de alto orden tambi\'en de tipo Schr\"odinger--Korteweg-de Vries. Buscando soluciones en forma de onda ``estacionaria-viajera'' se obtiene un sistema biarm\'onico estacionario, para el cual se demuestra la existencia y multiplicidad  de soluciones bajo determinadas condiciones de los par\'ametros.

\tableofcontents
\chapter*{Acknowledgements}
\addcontentsline{toc}{chapter}{Acknowledgements}
\markboth{ACKNOWLEDGEMENTS}{ACKNOWLEDGEMENTS}
\

First, I would like to express my sincere gratitude to my advisors Dr. Eduardo Colorado Heras \ and \  Dr. Pablo \'Alvarez Caudevilla  \ for \ the \ continuous \  support throughout my Master studies  and related research, for their patience, motivation, and immense knowledge. Their guide helped me at all times during the research and writing of this thesis.  I could not have imagined having better advisors and mentors for my Master thesis.

I am grateful to the Department of Mathematics of UC3M, for giving me the opportunity to carry out my studies and helping me out with the scholarship. 

I would like to thank  all who have been at my side during this time, in particular  to Dr. H\'ector Pijeira Cabrera for his advice and Edith for the warm hospitality.  I would also like to give special thanks to my friends  Noel, Ruxlan, Yanay, Abel, Yanely, Ariel, Janice, Laura and Ana Luisa for their company.

Last but not least, I would like to express my deepest gratitude to my loved ones, who have supported me throughout the entire  process. I am grateful to my mother, my father, my sister and my grandparents for supporting me spiritually throughout these studies and my life in general. I will be grateful forever for your love.
\chapter*{Notations}
\addcontentsline{toc}{chapter}{Notations}
\markboth{NOTATIONS}{NOTATIONS}
\begin{tabular}{p{7.4cm}p{7cm}}
$L^p(\Omega)$ & Lebesgue space with norm $\|\cdot\|_{L^p}$. \\
$L^q_{loc}(\Omega)$&Space of functions $f:\Omega\to\R$ such that $f\in L^p(K)$ for every compact set $K\subset\Omega$.\\ 
$\mathcal{C}(\Omega)$ & Space of continuous functions in $\Omega$. \\
$\mathcal{C}^n(\Omega)$& Space of $n$ times continuously differentiable functions in $\Omega$. \\
$\mathcal{C}^\infty(\Omega)$ & Space of infinitely differentiable  functions in $\Omega$. \\
$\mathcal{C}_c(\Omega)$ & Space of functions with compact support in $\Omega$. \\
$\mathcal{C}_c^\infty(\Omega)$ & Space of infinitely differentiable functions with compact support in $\Omega$.\\
$\R^N_+=\{(x',x_N)\in \R^N\ : x_N>0\ \}$& Set with the last component $x_N>0$.\\
${\bf Q}=\{(x',x_N)\in \R^N: |x'|<1,\ |x_N|<1\}$ & Set with the last component $x_N<1$ and $x'$ in the unit ball in $\R^{N-1}$. \\
${\bf Q}_+={\bf Q}\cap \R^N_+$& Intersection between ${\bf Q}$ and $\R^N_+$.\\
${\bf Q}_0=\{(x',0)\in \R^N\ : |x'|<1\}$&Set with $x'$ in the unit ball in $\R^{N-1}$ and $x_N=0$.\\
$W^{m,p}(\Omega)$& Sobolev space with $m$ derivatives in $L^p(\Omega)$.\\
$W^{m,p}_{r}(\Omega)$& Space of the radially symmetric functions belong $W^{m,p}(\Omega)$.\\
$W^{m,p}_{rd}(\Omega)$& Space of the non-increasing radially symmetric functions belong $W^{m,p}(\Omega)$.\\
$L(X,Y)$ & Set of the linear continuous maps from $X$ into $Y$.\\
$L_k(X,Y)$& Space of $k$-linear maps from $X$ into $Y$.\\

$\mathcal{C}(\mathcal{U},Z)$& Set of continuous maps from $\mathcal{U}$ onto $Z$.\\

\end{tabular}

\begin{tabular}{p{7.4cm}p{7cm}}

$\mathcal{C}^k(\mathcal{U},Z)$& Subset of $\mathcal{C}(\mathcal{U},Z)$ of $k$ times differentiable maps $J$ such that the application $\mathcal{U}\to L_k(X,\R)$, defined as $u\mapsto d^kJ(u)$, is continuous.
\\

$\mathcal{C}^{0,\alpha}(\mathcal{U},Z)$& Set of maps $J\in \mathcal{C}(\mathcal{U},Z)$ such that
\[
\sup\limits_{\substack{u,v \in \mathcal{U}\\ u\neq v}}
\left(\frac{|J(u)-J(v)|}{\|u-v\|^\alpha}\right)<+\infty
\]
for some $\alpha\in (0,1]$. If $\alpha=1$ these maps are called Lipschitz continuous and if $\alpha<1$ these maps are nothing but the H\"older continuous maps.\\
$\mathcal{C}^{k,\alpha}(\mathcal{U},Z)$& Set of maps $J\in \mathcal{C}^k(\mathcal{U},Z)$ such that $d^kJ(u)\in\mathcal{C}^{0,\alpha}(\mathcal{U},Z)$.\\
$\Omega^\star$&Symmetrized set of $\Omega$\\
$u^\star$&Schwarz symmetrization of function $u$\\
$\hbar$ & Reduced Planck constant (Planck constant divided by $2\pi$)\\
 $\displaystyle \nabla =\left(\frac{\partial }{\partial x_1},\frac{\partial }{\partial x_2},\dots,\frac{\partial }{\partial x_N}\right)
$ & Gradient differential operator\\
 $\displaystyle\Delta=\left(\frac{\partial^2 }{\partial x_1^2}+\frac{\partial ^2}{\partial x_2^2}+\dots+\frac{\partial^2 }{\partial x_N^2}\right)$ & Laplacian differential operator\\
 $\bra\cdot,\cdot\ket_H$&Scalar product in the Hilbert space $H$.\\
$\|\cdot\|_X$& Norm in the space X.\\
$X'$ & Dual space of $X$\\
 $\hookrightarrow$& Continuous embedding.\\
$\hookrightarrow\hookrightarrow$& Compact embedding.\\
 $\rightharpoonup$& Weak Convergence.\\
 $a.e.$& Convergence almost everywhere.\\
$f|_\Omega$& Restriction of function $f$ to $\Omega$.\\
$\mathcal{F}|_\Omega$& Set of functions $f\in\mathcal{F}$ restricted to $\Omega$.\\ 
$\omega_N$ & Volume of the unit ball in $\R^N$.\\
$\omega_N'$ & Surface of the unit sphere in $\R^N$.\\

$T_pM$&Tangent space to $M$ at the point $p\in M$.\\
$\oplus$& Direct sum.\\
$d_GJ$&G\^ateaux differential of $J$.\\
$dJ$&Fr\'echet differential of $J$.\\
$d_MJ$&Constrained derivative of $J$ on a manifold $M.$\\
$\nabla_M$&Constrained gradient of $J$ on $M$.\\
$ \mathcal{H}_{N}X$& $N$-dimensional (Hausdorff) measure of the set $X$.
 
\end{tabular}


\chapter*{Introduction}
\addcontentsline{toc}{chapter}{Introduction}
\markboth{INTRODUCTION}{INTRODUCTION}

This work aims to prove the existence of solutions for two coupled systems of partial differential equations.  
\

The first system is composed by a nonlinear Schr\"odinger equation and a Korteweg-de Vries equation as follows 
\begin{equation}\label{NLS-KdV11}\tag{S1}
\boxed{\left\{\begin{array}{rcl}
if_t + f_{xx} + |f|^2f+ \b fg & = &0\\
g_t +g_{xxx} +gg_x  + \frac 12\b(|f|^2)_x  & = & 0,
\end{array}\right.}
\end{equation}
where $f=f(x,t)\in \mathbb{C}$ while $g=g(x,t)\in \mathbb{R}$, and
$\b\in \mathbb{R}$ is the real coupling coefficient. System \eqref{NLS-KdV11}
appears in phenomena of interactions between short and long
dispersive waves, arising in fluid mechanics, such as  the
interactions of capillary-gravity water waves. Indeed, $f$
represents the short wave, while $g$ stands for the long wave. See
\cite{aa,c2,c3,cl,fo} and the references therein for more details. We  look for  solitary ``traveling'' waves solutions, namely
solutions to \eqref{NLS-KdV11} of the form
\[ (f(x,t),g(x,t))=\left(e^{i\o t}
e^{i\frac c2 x}u(x-ct),v(x-ct)\right), \] with $u$ and $v$ real functions. Choosing  $\l_1=\o+\frac{c^2}{4}$,
$\l_2=c$, we get that $u,\, v$ solve the following stationary problem in dimension one 
\begin{equation}\label{NLS-KdV111}
\left\{\begin{array}{rcl}
-u'' +\l_1 u & = & u^3+\beta uv \\
-v'' +\l_2 v & = & \frac 12 v^2+\frac 12\beta u^2.
\end{array}\right.
\end{equation}
This system has been previously studied by Dias, Figueira and Oliveira in 
\cite{dfo}. Also, a generalization of \eqref{NLS-KdV111}  with general power nonlinearities, has been previously analyzed  by the same authors in \cite{dfo2} and by Albert and Bhattarai in \cite{ab}. The results obtained in the works previously mentioned were improved in several points in \cite{c2}. In this work,  we focus our attention on one of these points which deals with the existence of positive even ground and bound states of \eqref{NLS-KdV111} under the appropriate range of parameter settings. We mainly perform a detailed analysis of the recent work \cite{c2,c3}, where positive solutions of  \eqref{NLS-KdV111} are classified proving:

\begin{itemize}

\item Existence of positive even ground states of \eqref{NLS-KdV111} under the
following hypotheses:
\begin{itemize}
\item  the coupling coefficient $\b>\L>0$ for an appropriate constant $\L$; see Theorem \ref{th:1},
\item $\b>0$ and $\l_2\gg 1$; see Theorem \ref{th:ground2}.
\end{itemize}
\item Existence of positive even bound states of \eqref{NLS-KdV111} when:
\begin{itemize}
\item $0<\b\ll 1$; see Theorem \ref{th:2}, where we also give a bifurcation result,
\item $0<\b<\L$ and $\l_2\gg 1$; see Theorem \ref{th:bound2}.
\end{itemize}
\end{itemize}
We also extend these results to dimensions $N=2,3$. The coexistence of positive bound and ground states for  $0<\b<\L$
and $\l_2$ large is a great novelty due to the  difference with the more
studied systems of nonlinear Schr\"odinger  equations in the last several years; see
Remark \ref{rem:11}.

\

The second system that we study is a higher order system coming from  \eqref{NLS-KdV111} as a natural extension. More
precisely, we consider the following system
\begin{equation}\label{NLS-KdV edu111}\tag{S2}
\boxed{\left\{\begin{array}{rcl}
if_t -  f_{xxxx} + |f|^2f+ \b fg & = & 0\\
g_t-g_{xxxxx}+|g|g_x+\tfrac12\beta(|f|^2)_x& = & 0.
\end{array}\right.}
\end{equation}
Looking for ``standing-traveling" waves
solutions of the form
\[
(f(x,t),g(x,t))=\left(e^{i\l_1 t} u(x),v(x-\l_2t)\right),
\]
with $u$ and $v$ real functions,  we arrive at the fourth-order stationary system
 \be \left\lbrace
\begin{array}{ccl}\label{eq:NLS-KdV2221}
u^{(iv} +\l_1u & =& u^3+\beta uv\\
v^{(iv}+\l_2v & =& \frac{1}{2}|v|v+\frac{1}{2}\beta u^2,
\end{array}
 \right.
\ee where $w^{(iv}$ denotes the fourth derivative of $w$. This is the first time,
up to our knowledge, that the interaction of standing waves and
traveling waves is analyzed in the mathematical literature. 
Although system \eqref{NLS-KdV edu111} only make sense in dimension $N=1$, we can consider the stationary system \eqref{eq:NLS-KdV2221} in higher dimensional cases
\be
\left\lbrace
\begin{array}{ccl}\label{eq:NLS-KdV21}
\D^2u +\l_1u & =& u^3+\beta uv\\
\D^2v+\l_2v & =& \frac{1}{2}|v|v+\frac{1}{2}\beta u^2,
\end{array}
 \right.
\ee
 where $u,v\in W^{2,2}(\R^N)$, $1\le N\le 7$, $\l_j>0$ with
$j=1,2$ and $\beta>0$ is the coupling parameter.

Recently, other similar fourth-order systems studying the interaction of coupled nonlinear Schr\"odinger equations have appeared, for example in \cite{pev}, where the
coupling terms have the same homogeneity as the nonlinear terms. Note that, as far as we know, there is no previous mathematical work analyzing a higher order system with the nonlinear and coupling terms considered in \eqref{eq:NLS-KdV21}.

In system \eqref{eq:NLS-KdV21}, we first analyze the dimensional cases $2\le N\le 7$ in the radial framework by using the compactness described in Remark
\ref{re:functionl constrain y PS condition}-$(iii)$. The one dimensional case is also studied through the application of a  measure lemma due to P. L. Lions \cite{lions2} to circumvent the lack of compactness. To be more precise, we
prove that there exists a positive critical value of the coupling
parameter $\beta$, denoted by $\L'$ and defined by
 \eqref{Lambda},
such that the associated functional constrained to the corresponding Nehari manifold possesses
 a positive global minimum. We show that this positive global minimum is a critical point with energy below the energy of the semi-trivial solution under the following
 hypotheses:
 either $\beta>\Lambda'$, or  $\beta>0$ and $\lambda_2\gg 1$.
Furthermore, we find a mountain pass critical point if $\beta<\L'$ and $\lambda_2\gg 1$.

\

This work is organized as follows. In Chapter \ref{Preliminaries} we present some preliminaries necessary for the proper understanding of the results and the sake of completeness. The Schr\"odinger equations and the Korteweg-de Vries equation are introduced with a brief historical summaries of their discoveries. We recall the Sobolev spaces and their most important properties. Basic concepts of calculus of variation are included, such as the Palais-Smale compactness condition and the Mountain Pass Theorem. We also present some results about the Schwarz symmetrization that will be useful for our work, especially in the second chapter.

Chapter \ref{chap 2to orden} is devoted to the study of system \eqref{NLS-KdV11}. In Section \ref{sec:funct} we
introduce the functional framework and give some
definitions. Next,  we define the Nehari Manifold in Section
\ref{sec:key}, proving some properties of it. We establish a useful
measure lemma and show a result dealing with qualitative properties
of the semi-trivial solution. Section \ref{sec:mainR} is divided
into two subsections; the first one contains the proof of the existence of ground
states, and the second one deals with the existence of bound states. 

In Chapter \ref{chap 4to}, we perform the corresponding analysis of the fourth order system \eqref{NLS-KdV edu111}. In Section \ref{sec:2}, we
introduce the notation, establish the functional framework,
 define the Nehari manifold and study its properties.
Section \ref{sec:4} is devoted to prove the main results. It is divided into two subsections; in the first one
 we study the high-dimensional case ($2\le
N\le 7$), while the second one   deals
with the one-dimensional case.

\clearpage
\

\thispagestyle{empty}
\newpage
\pagenumbering{arabic} \setcounter{page}{1}

\chapter{Preliminaries}\label{Preliminaries}

In this chapter we present the Schr\"odinger equation and the Korteweg-de Vries equation. We  also discuss some preliminary notions that we are going to use throughout this work. 

\section{The Schr\"odinger equation}\

In quantum mechanics, the Schr\"odinger equation is a partial differential equation that describes how the quantum state of a quantum system changes with time. It was formulated in 1926  by the Austrian physicist Erwin Schr\"odinger \cite{S}. 
In classical mechanics, Newton's second law ($F=ma$) is used to mathematically predict the state of a given system at any time after a known initial condition. In quantum mechanics, the analogue of Newton's law is Schr\"odinger equation for a quantum system (usually atoms, molecules, and subatomic particles). The Schr\"odinger equation is a linear partial differential equation, describing the time-evolution of the system's ``wave function'' (also called a ``state function") \cite{Griffiths}.
Although Schr\"odinger equation is often presented as a separate postulate, some authors \cite[\textsection3]{Ballentine} show that some properties resulting from the Schr\"odinger equation may be deduced just from symmetry principles alone; for example, the commutation relations. Generally, ``derivations'' of the Schr\"odinger equation demonstrate its mathematical plausibility for describing wave-particle duality but, to date, there are no universally accepted derivations of the Schr\"odinger equation from appropriate axioms.
In the Copenhagen interpretation\footnote{The Copenhagen interpretation is an expression of the meaning of quantum mechanics that was largely devised in the years 1925 to 1927 by Niels Bohr and Werner Heisenberg. It remains one of the most commonly taught interpretations of quantum mechanics. 
According to the Copenhagen interpretation, physical systems generally do not have definite properties prior to being measured, and quantum mechanics can only predict the probabilities that measurements will produce certain results.} of quantum mechanics, the wave function is the most complete description that can be given of a physical system. The Schr\"odinger equation describes not only molecular, atomic, and subatomic systems, but also macroscopic systems, possibly even the whole universe. This equation, in its most general form, is consistent with both classical mechanics and special relativity, but the original formulation by Schr\"odinger himself was non-relativistic.

The Schr\"odinger equation takes the form
\begin{equation}\label{eq:Schrodinger_R}
i\hbar\,f_t=\widehat{H} f,
\end{equation}
where $i$ is the imaginary unit, $\hbar$ is the reduced Planck constant, $f=f(\mathbf{x},t)$ is a complex wave function on $\R^N\times \R$, $f_t$ denote $\partial f/ \partial t$ and $\widehat{H}$ is a Hamiltonian operator which characterizes the total energy of a given wave function and takes different forms depending on the physical situation. The best known example of this kind of equation is the non-relativistic Schr\"odinger equation for a single particle moving in an external field
\begin{equation}\label{eq:Schrodinger_famous}
i\hbar\, f_t=\left[-\frac{\hbar}{2\mu}\Delta+V\right] f,
\end{equation}
where $\widehat{H}$ was taken as the total energy equals kinetic energy plus potential energy, $\Delta$ is the Laplacian differential operator and $\mu$ is the reduced mass.
Rescaling \eqref{eq:Schrodinger_famous} by
\begin{equation}
f'(\mathbf{x}',t)= f(\mathbf{x},t),\qquad \mathbf{x}'=\sqrt{2\mu}\mathbf{x},
\end{equation} 
and taking a nonlinear variation of the form $$V=-\hbar J(|f'|^2),$$ for a  given smooth complex function $J$, we obtain (omitting primes) the so called nonlinear Schr\"odinger equation
\begin{equation}\label{eq:NLSE1}
if_t+\Delta f + J(|f|^2)f=0.
\end{equation}
The nonlinear Schr\"odinger equation is a classical field equation whose principal applications are related to the propagation of light in nonlinear optical fibers and planar waveguides \cite{Malomed}, and Bose-Einstein condensates\footnote{A Bose-Einstein condensate  is a state of matter of a dilute gas of bosons cooled to temperatures very close to absolute zero.} confined to highly anisotropic cigar-shaped traps, in the mean-field regime \cite{Pitaevskii}. Additionally, the equation appears in the studies of small-amplitude gravity waves\footnote{In fluid dynamics, gravity waves are waves generated in a fluid medium or at the interface between two media when the force of gravity or buoyancy tries to restore equilibrium. An example of such an interface is that between the atmosphere and the ocean, which gives rise to wind waves.} on the surface of deep inviscid (zero-viscosity) water \cite{Malomed}, the Langmuir waves\footnote{The Langmuir waves are rapid oscillations of the electron density in conducting media such as plasmas or metals. } in the plasma \cite{Malomed}, the propagation of plane-diffracted wave beams in the focusing regions of the ionosphere \cite{Gurevich}, the propagation of Davydov's alpha-helix solitons\footnote{Davydov soliton is a quantum quasiparticle representing an excitation propagating along the protein alpha-helix self-trapped amide I. It is a solution of the Davydov Hamiltonian. It is named for the Soviet and Ukrainian physicist Alexander Davydov.}, which are responsible for energy transport along molecular chains \cite{Balakrishnan}, and many others. 
More generally, the nonlinear Schr\"odinger equation  appears as one of the universal equations that describe the evolution of slowly varying packets of quasi-monochromatic waves in weakly nonlinear media that have dispersion \cite{Malomed}. 

 In particular, the one-dimensional nonlinear Schr\"odinger equation is an example of an integrable model. In quantum mechanics, the one-dimensional nonlinear Schr\"odinger equation is a special case of the classical nonlinear Schr\"odinger field\footnote{In quantum mechanics and quantum field theory, a Schr\"odinger (nonlinear Schr\"odinger) field is a quantum field which obeys the  Schr\"odinger (nonlinear Schr\"odinger) equation.}, which in turn is a classical limit of a quantum Schr\"odinger field. Both the quantum and the classical one-dimensional nonlinear Schr\"odinger equation are integrable. In more than one dimension, the equation is not integrable. It allows us a collapse and wave turbulence \cite{Falkovich}.

 Another important equation related to the Schr\"odinger equation is the so called fractional Schr\"odinger equation, which is a fundamental equation of fractional quantum mechanics. It was introduced by Nick Laskin in 1999 (see \cite{Las2000,Las2002}) as a result of extending the Feynman path integral, from the Brownian-like to L\'evy-like quantum mechanical paths. The term fractional Schr\"odinger equation was coined by Nick Laskin who made a generalization of standard quantum mechanics called fractional quantum mechanics. The fractional Schr\"odinger equation is obtained  form \eqref{eq:Schrodinger_R} replacing $\widehat{H}$ by a  fractional Hamiltonian operator $\widehat{H}_\alpha$ of the from 
\begin{equation}
\widehat{H}_\alpha=D_\alpha(-\hbar^2\Delta)^{\alpha/2}+V,
\end{equation}  
where $D_\alpha$ is a scale constant, $(-\hbar^2\Delta)^{\alpha/2}$ is the quantum Riesz fractional derivative\footnote{ The Riesz fractional derivative was originally introduced in \cite{Riesz}.}. The most common case in the literature of this equation is the three-dimensional case, in which the 3D quantum  Riesz fractional derivative is given by
\[
(-\hbar^2\Delta)^{\alpha/2}f(\mathbf{x},t)=\frac{1}{(2\pi\hbar)^3}\int_{\R^3}e^{i\frac{\mathbf{p}\mathbf{x}}{\hbar}}|\mathbf{p}|^\alpha\widehat{f}(\mathbf{p},t)\,d\mathbf{p},
\]
where
\[
\widehat{f}(\mathbf{p},t)=\int_{\R^3}e^{-i\frac{\mathbf{p}\mathbf{x}}{\hbar}} f(\mathbf{x},t)\,d\mathbf{x},
\]
is the  three-dimensional Fourier transform of $f$. The index $1 < \alpha \leq≤ 2$ in the above expression  is the so called  L\`evy index. Thus, the fractional Schr\"odinger equation includes a space derivative of fractional order $\alpha$ instead of the second order space derivative in the standard Schr\"odinger equation. At $\alpha= 2$, the fractional Schr\"odinger equation becomes the standard Schr\"odinger equation.
There are many applications of the fractional Schr\"odinger equation such as the  	fractional Bohr atom, the  fractional quantum oscillator, the 	fractional quantum mechanics in solid state systems, and others; see \cite{Xiaoyi} for more applications.

\section{The Korteweg-de Vries equation}\

The Korteweg-de Vries  equation  is a universal mathematical model for the description of weakly nonlinear long wave propagation in dispersive media. This equation is given by
\begin{equation}\label{eq:KDVE}
g_t+\alpha gg_x+\beta g_{xxx}=0, 
\end{equation}
where $g(x,t)$ is a real function of the one-dimensional space coordinate $x$ and time $t$. The coefficients $\alpha$ and $\beta$
are determined by the medium properties and can be either constants or functions.

 An incomplete list of physical applications of the Korteweg-de Vries equations includes shallow-water gravity waves \cite{Hereman}, ion-acoustic waves\footnote{In plasma physics, an ion-acoustic wave is one type of longitudinal oscillation of the ions and electrons in a plasma, much like acoustic waves travelling in neutral gas.} in collisionless plasma \cite{Nakamura,Wakil}, waves in bubbly fluids \cite{Khismatullin}, waves in the ocean and many others \cite{Crighton}. This broad range of applicability is explained by the fact that the Korteweg-de Vries equation describes a combined effect of the lowest-order, quadratic, nonlinearity (term $gg_x$) and the simplest long-wave dispersion (term $g_{xxx}$). One can find derivations of the Korteweg-de Vries equation for different physical contexts in the books by Dodd et al \cite{Dodd}, Drazin and Johnson \cite{Drazin},  Newell \cite{Newell}, and others. 
 
  Although the Korteweg-de Vries equations with constant coefficient was originally derived in the second half of the 19th century, its real significance as a fundamental mathematical model for the generation and propagation of long nonlinear waves of small amplitude has been understood only after the seminal works of Zabusky and Kruskal (1965)\cite{Zabusky}, Gardner, Greene, Kruskal and Miura (1967)\cite{Gardner} and Lax (1968)\cite{Lax}. These authors showed that the Korteweg-de Vries equation (unlike a ``general'' nonlinear dispersive equation) can be solved exactly for a broad class of initial or boundary conditions and, importantly, the solutions often contain a combination of localized wave states, which preserve their ``identity'' in the interactions with each other; pretty much as classical particles do. In the longtime asymptotic solutions, such localized states represent solitary waves, which are waves that maintain their shape while propagating at a constant speed. Such solitary wave solutions of the Korteweg-de Vries equation have been called {\it solitons}\footnote{The soliton phenomenon was first described in 1834 by John Scott Russell who observed a solitary wave in the Union Canal in Scotland. He reproduced the phenomenon in a wave tank and named it the {\it Wave of Translation}. Solitons are caused by a cancellation of nonlinear and dispersive effects in the medium. These are the solutions of a widespread class of weakly nonlinear dispersive partial differential equations describing physical systems.} by Zabusky and Kruskal  in 1965 \cite{Zabusky} owing to their unusual particle-like behaviour in the interactions with other solitary waves and nonlinear radiation. However, the solitons were already well known due to the original works of Russel (1845) \cite{Russell}, Boussinesq (1972) \cite{Boussinesq}, Rayleigh (1876), and Korteweg and de Vries (1895) \cite{Korteweg}. 
  

When  $\alpha$ y $\beta$ are constant, the Korteweg-de Vries equation  can be rescaled in order to eliminate the constants. Setting 
\begin{equation}
g'(x',t')=\alpha g(x,t),\qquad x'=\frac{x}{\sqrt{\beta}},\qquad t'=\frac{t}{\sqrt{\beta}},
\end{equation}
equation \eqref{eq:KDVE} takes the form (omitting primes),
\begin{equation}
g_t+gg_x+g_{xxx}=0.
\end{equation}
We shall look for a solution of the above equation in the form of a traveling wave, i.e., $g(x,t)=v(\theta)$, where $\theta=x-ct$ is the travelling phase and $c$ is the phase velocity. Moreover, we assume that $v,v'$ and $v''$ vanish at infinity where, now, $v'$ and $v''$ represent the first and second derivative respectively of $v$. Now, the Korteweg-de Vries equation reduces to an ordinary differential equation for the function $v$ as follows

\begin{figure}[t]
\begin{center}
    \includegraphics[height=6cm]{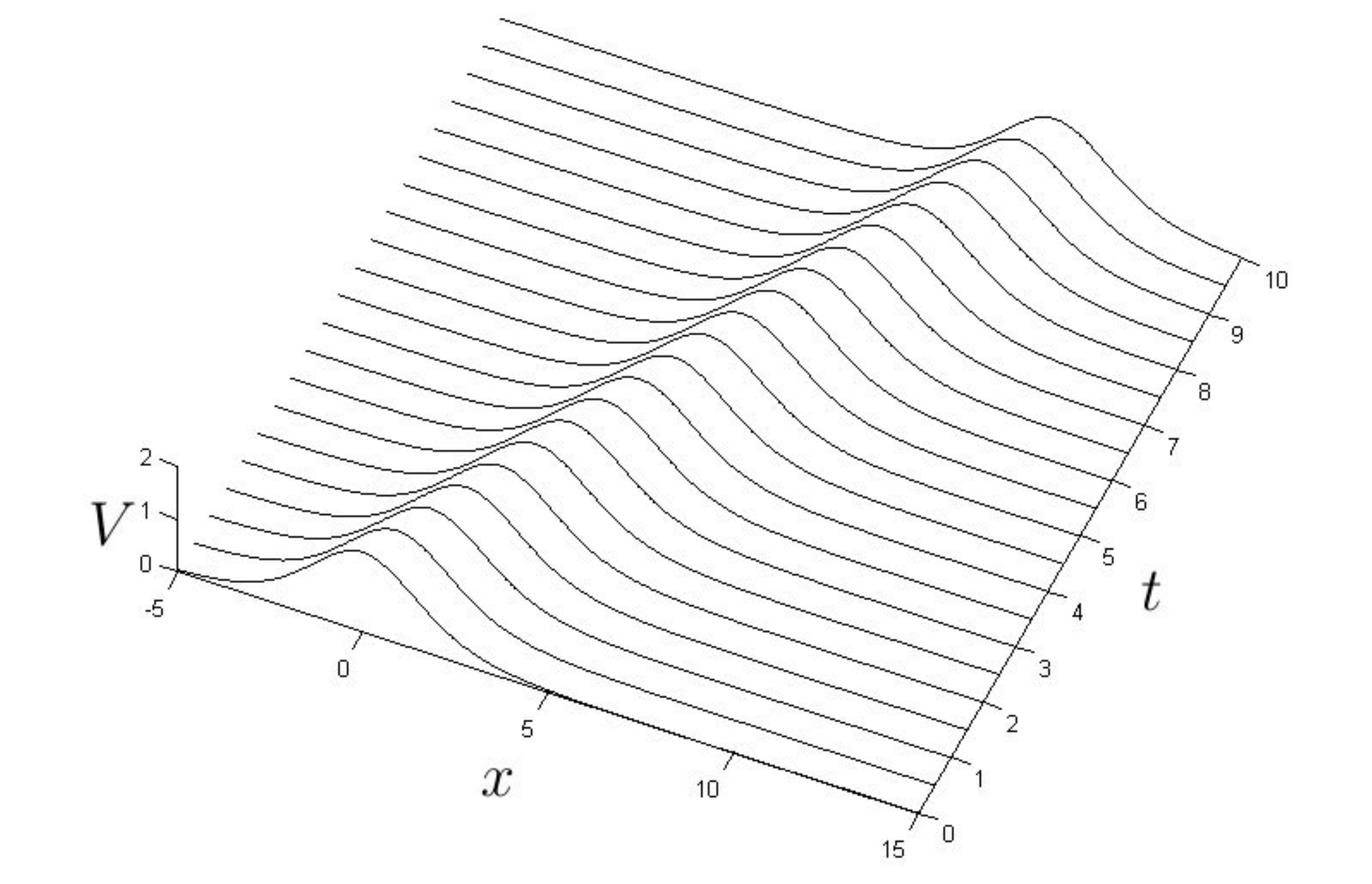}
\end{center}
   \caption{Soliton}
 \label{Soliton}
\end{figure}\begin{equation}\label{cambio kdv}
-cv'+vv'+v'''=0,
\end{equation}
which integrated gives us
\begin{equation}
-cv+\frac{v^2}{2}+v''+c_1=0.
\end{equation}
Taking into account that $v$ vanishes   at infinity it follows that $c_1=0$, and performing the change of variable $v(\theta)=2cV(\sqrt{c}\theta)$ we get
\begin{equation}
V''-V+V^2=0.
\end{equation}
We can see in \cite{kw} that the above equation has a unique positive even solution given by
\begin{equation}\label{v semitrivial}
V(\theta)=\frac{3}{2\cosh^2\left(\frac{\theta}{2}\right)}.
\end{equation}
Taking into account that $\theta=x-ct$, the expression \eqref{v semitrivial} describes a right-moving soliton (see Figure \ref{Soliton}).

\section{Sobolev spaces}
\

In this section we will present the Sobolev spaces and some of their properties.

 Let $\Omega\subset\R^N$ be an open set.
\begin{Definition}
 For $1\leq p\leq \infty$, the Sobolev space $W^{1,p}(\Omega)$ is defined by
\[
W^{1,p}(\Omega)=\left\{
u\in L^p(\Omega) \left|
\begin{array}{l}
\exists\ g_1,g_2,\dots,g_N\in L^p(\Omega)\qquad\text{such that}\\ 
\displaystyle\int_\Omega u\frac{\partial \varphi}{\partial x_i}=-\displaystyle\int_\Omega g_i\varphi\quad\forall\varphi\in \mathcal{C}_c^\infty(\Omega),\quad\forall i=1,2,\dots,N
\end{array} 
\right.
\right\}.
\]
\end{Definition}
For $u\in W^{1,p}(\Omega)$ we define the weak partial derivatives and gradient as follows
\[
\frac{\partial u}{\partial x_i}=g_i,\qquad
 \nabla u=\left(\frac{\partial u}{\partial x_1},\frac{\partial u}{\partial x_2},\dots,\frac{\partial u}{\partial x_N}\right).
\]
The space $W^{1,p}(\Omega)$ is equipped with the norm
\begin{equation}\label{norma sobolev 1}
\|u\|_{W^{1,p}}=\|u\|_{L^p}+\sum\limits_{i=1}^N\left\|\frac{\partial u}{\partial x_i}\right\|_{L^p},\footnote{ When there is no confusion we shall often write $W^{m,p}$ and $L^p$ instead of $W^{m,p}(\Omega)$ and $L^p(\Omega)$ respectively.}
\end{equation}
which, for $1\leq p< \infty$, is equivalent to the norm
\begin{equation}\label{norma sobolev 2}
\|u\|_{W^{1,p}}=\left(\|u\|_{L^p}^p+\|\nabla u\|_{L^p}^p\right)^{\frac{1}{p}},
\end{equation}
where 
\[
\|\nabla u\|_{L^p}=\left(\sum\limits_{i=1}^N\left\|\frac{\partial u}{\partial x_i}\right\|_{L^p}^p\right)^\frac{1}{p}.
\]

Now we will introduce the Sobolev spaces with higher orders of regularity.
\begin{Definition}
Let $m\geq 2$ be an integer and $1\leq p\leq \infty$, then, we define by induction the set $W^{m,p}(\Omega)$
\[
W^{m,p}(\Omega)=\left\{
u\in W^{m-1,p}(\Omega)
\left|\ 
\frac{\partial u}{\partial x_i}\in W^{m-1,p}(\Omega),\quad \forall i=1,2,\dots,N
\right.
\right\}.
\]
\end{Definition}
We use the standard multi-index notation $\alpha=(\alpha_1,\alpha_2,\dots,\alpha_N)$, with $\alpha_i\geq 0$ integers, to denote the weak partial derivatives as follows,
\[
D^\alpha u =\frac{\partial^{|\alpha|}u}{\partial x_1^{\alpha_1}\partial x_2^{\alpha_2},\cdots, \partial x_N^{\alpha_N}},\quad\text{where}\quad |\alpha|=\sum\limits_{i=1}^N \alpha_i\leq m.
\]
The space $W^{m,p}(\Omega)$ is a Banach space equipped with the norm
\begin{equation}\label{norma sobolev m1}
\|u\|_{W^{m,p}}=\sum\limits_{0\leq|\alpha|\leq m}\left\|D^\alpha u\right\|_{L^p}.
\end{equation}
which, for $1\leq p< \infty$, is equivalent to the norm
\begin{equation}\label{norma sobolev m2}
\|u\|_{W^{m,p}}=\left(\sum\limits_{0\leq|\alpha|\leq m}\left\|D^\alpha u\right\|^p_{L^p}\right)^\frac{1}{p}.
\end{equation}
Moreover, $H^m(\Omega)=W^{m,2}(\Omega)$ equipped with the scalar product 
\be
\bra u,v\ket_{W^{m,2}}=\sum\limits_{0\leq|\alpha|\leq m} \bra D^\alpha u,D^\alpha v\ket_{L^2},
\ee 
is a Hilbert space.  We also have the following result.

\begin{Theorem}\label{reflex}{\normalfont \cite[Theorem 3.6]{Adams}}
The space $W^{m,p}(\Omega)$ is separable if $1\leq p< \infty$, and reflexive if $1<p<\infty$.
\end{Theorem}

Another important result is the Theorem of global approximation by smooth functions, or also known by the Meyers-Serrin Theorem; see \cite{meyers}. 

 \begin{Theorem}[Global approximation by smooth functions]\label{Global approximation by smooth functions}{\normalfont \cite[\textsection5.3.2]{Evans}}
 Assume $\Omega$ is bounded, and suppose as well that $u\in W^{m,p}(\Omega)$ for some $1\leq p<\infty$. Then there exist functions $u_n\in \mathcal{C}^\infty(\Omega)\cup W^{m,p}(\Omega)$ such that $u_n\to u$ in $W^{m,p}(\Omega)$.
 \end{Theorem}

\subsection{Continuous and compact embedding}

\begin{Definition}
Let $X$ and $Y$ be two Banach spaces, with norms $\|\cdot\|_X$ and $\|\cdot\|_Y$ respectively. We say that $X$ is continuously embedded in $Y$, and we denote it by
\[
X\hookrightarrow Y,
\]
if $X\subset Y$ and the inclusion map is continuous, i.e., if there exists a constant $C\geq 0$ such that
\[
\|x\|_Y\leq C\|x\|_X,\qquad\forall x\in X.
\]
\end{Definition}
\begin{Definition}
Let $X$ and $Y$ be two Banach spaces, with norms $\|\cdot\|_X$ and $\|\cdot\|_Y$ respectively. We say that $X$ is compactly embedded in $Y$, and we denote it by
\[
X\hookrightarrow \hookrightarrow Y,
\]
if $X\hookrightarrow Y$ and the embedding of $X$ into $Y$ is a compact operator, i.e., if  every bounded sequence in the norm $\|\cdot\|_X$ has a convergent subsequence in the norm $\|\cdot\|_Y$.
\end{Definition}


\begin{Definition}\label{p aster} Given integers $N,m\geq 1$, and a real number  $1\leq p<\infty$, we define the critical exponent $p^*$ as follows
\[
p^*=\left\{
\begin{array}{cc}
\displaystyle\frac{Np}{N-mp}&\quad\text{if}\ \ mp<N\vspace{0.3cm}\\
+\infty&\quad\text{if}\ \ mp\geq N
.\end{array}
\right.
\]
\end{Definition}

The next theorem gives us the continuous embedding of Sobolev spaces. Its proof can be found in detail in \cite[Theorem 4.12]{Adams}.

\begin{Theorem}\label{cont emb}
Let $m\geq 1$ be an integer and $1\leq p<\infty$. Then,
\[\begin{array}{llr}
W^{m,p}(\Omega)\hookrightarrow L^q(\Omega),\quad & 
\displaystyle\forall q\in [p,p^*],& \quad\displaystyle\text{if}\ \ mp<N,\\ 
W^{m,p}(\Omega) \hookrightarrow L^q(\Omega),\quad & 
\displaystyle\forall q\in [p,\infty), &\quad\displaystyle\text{if}\ \ mp=N,\\ 
W^{m,p}(\Omega) \hookrightarrow \mathcal{C}(\overline{\Omega}),\quad &  &\quad\displaystyle\text{if}\ \ mp>N.
\end{array}\]
\end{Theorem}

\begin{remark}
The value of $p^*$ defined in \ref{p aster} is obtained by a scaling argument. For example, let us take $m=1$ and $\Omega=\R^N$. If we suppose that there exists a constant $C$ and $1\leq q\leq \infty$ such that 
\begin{equation}\label{sobolev ineq}
\|u\|_{L^q}\leq C\|\nabla u\|_{L^p},\qquad\forall u\in \mathcal{C}_c^\infty(\R^N),
\end{equation}
then, in particular, it is true for $u_\lambda(x)=u(\lambda x)$ for all $\l>0$. Note that on one hand we have
\[
\|u_\l\|_{L^q}^q=\int_{\RN}|u(\l x)|^q\,dx=\int_{\RN}|u( y)|^q\l^{-N}\,dy=\l^{-N}\|u\|_{L^q}^q,
\]
and, on the other hand 
\begin{align*}
\|\nabla u_\l\|_{L^p}^p&=\sum\limits_{i=1}^N\left\|\frac{\partial }{\partial x_i}u_\l(x)\right\|_{L^p}^p=\sum\limits_{i=1}^N\int_{\RN}\left|\frac{\partial }{\partial x_i}u(\l x)\right|^p\,dx\\
&=\sum\limits_{i=1}^N\int_{\RN}\left|\l\frac{\partial }{\partial y_i}u(y)\right|^p\l^{-N}\,dy=\l^{p-N}\|\nabla u\|_{L^p}^p.
\end{align*}
Thus, substituting $u_\l$ in \eqref{sobolev ineq} we obtain
\begin{align*}
\|u_\l\|_{L^q}&\leq C\|u_\l\|_{L^p}\\
\|u\|_{L^q}&\leq C \l^{\left(1+\frac{N}{q}-\frac{N}{p}\right)}\|\nabla u\|_{L^p},
\end{align*}
and, taking into account that this holds for all $\lambda>0$, then
\[
1+\frac{N}{q}-\frac{N}{p}=0,
\]
or equivalently
\[q=\frac{Np}{N-p}=p^*.\]
The above expression coincides with Definition \ref{p aster} for $m=1$.
\end{remark}

Now we will introduce the notion of set of class $\mathcal{C}^1$ in order to present later the Rellich-Kondrachov Theorem.
\begin{Definition}
We say that an open set $\Omega\subset\RN$ with boundary $\G$, is of class $\mathcal{C}^m$, if for every $z\in\G$ there exists a neighbourhood $\mathcal{U}$ of $z$ in $\RN$ and a bijective map $h:{\bf Q}\to \mathcal{U}$ such that
\[
h\in \mathcal{C}^m(\overline{\bf Q}),\quad 
h^{-1}\in \mathcal{C}^m(\overline{\mathcal{U}}),\quad
h({\bf Q}_+)=\mathcal{U}\cap\{\bf Q\}\quad\text{and\ \ \  h}({\bf  Q}_0)=\mathcal{U}\cap \G.
\]  
\end{Definition}
The following theorem is a very important result about compact embedding in Sobolev spaces and it is obtained as a part of the Rellich-Kondrachov Theorem (see \cite[Theorem 6.3 ]{Adams} for further information). 
\begin{Theorem}[Rellich-Kondrachov]\label{compact emb}
Suppose that $\Omega\subset\R^N$ is a bounded set of class $\mathcal{C}^m$ and  $1\leq p<\infty$. Then we have the following compact embeddings
\[\begin{array}{llr}
W^{m,p}(\Omega)\hookrightarrow\hookrightarrow L^q(\Omega),\quad & 
\displaystyle\forall q\in [1,p^*),& \quad\displaystyle\text{if}\ \ mp<N,\\ 
W^{m,p}(\Omega)\hookrightarrow \hookrightarrow L^q(\Omega),\quad & 
\displaystyle\forall q\in [p,\infty), &\quad\displaystyle\text{if}\ \ mp=N,\\ 
W^{m,p}(\Omega)\hookrightarrow \hookrightarrow \mathcal{C}(\overline{\Omega}),\quad &  &\quad\displaystyle\text{if}\ \ mp>N.
\end{array}\]
\end{Theorem}
We will only show the proof for the cases $mp< N$ and $mp= N$ of the above theorem, because these are the most important cases we will use in our work. The proof of the case $mp> N$ can be found in \cite[\textsection6.5]{Adams}. Before starting with the proof we will introduce some necessary results such as the following extension theorem  and the Riesz-Fr\'echet-Kolmogorov Theorem.

\begin{Theorem}\label{extension}{\normalfont \cite[Theorem 5.22]{Adams}} Suppose that $\Omega\subset\R^N$ is of class $\mathcal{C}^m$ with bounded boundary (or $\Omega=\R^N_+$). Then there exists a linear extension operator 
\begin{align*}
P:W^{m,p}(\Omega)&\to W^{m,p}(\R^N),\qquad(1\leq p\leq \infty)
\end{align*}
such that for all $u\in W^{m,p}(\Omega)$,
\begin{align}
Pu|_\Omega&=u,\\
\|Pu\|_{L^p(\R^N)}&\leq C\|u\|_{L^p(\Omega)},\\
\|Pu\|_{W^{m,p}(\R^N)}&\leq C\|u\|_{W^{m,p}(\Omega)}\label{iii extension},
\end{align}
where $C$ is a constant that depends only on $\Omega$.
\end{Theorem}

\begin{Theorem}[Riesz-Fr\'echet-Kolmogorov]\label{rfc} {\normalfont \cite[Theorem 4.26]{Brezis}} Let $\mathcal{F}$ be a bounded set in $L^p(\R^N)$ with $1\leq p<\infty$, and let
\begin{align*}
\tau_h:L^p(\R^N)&\to L^p(\R^N),\\
u&\mapsto \tau_hu,
\end{align*}
 be the shift map such that $\tau_hu(x)=u(x+h)$ with $x,h\in\R^N.$  Assume that 
 \begin{equation}\label{limh}
 \lim\limits_{|h|\to 0} \|\tau _h u-u\|_{L^p(\R^N)}=0,\qquad\text{uniformly in } u\in\mathcal{F}. 
 \end{equation}
Then, the closure of $\mathcal{F}|_\Omega$ in $L^p(\Omega)$ is compact for any measurable set $\Omega\subset\R^N$ with finite measure.
\end{Theorem}
We will also need to use the following proposition, which we will include without proof.
\begin{Proposition}\label{desig gradiente}{\normalfont \cite[Proposition  9.3]{Brezis}}
Let $u\in W^{1,p}(\R^N)$ with $1
\leq p\leq\infty$. Then 
\[
\|\tau _h u-u\|_{L^p(\R^N)}\leq |h|\|\nabla u\|_{L^p(\R^N)}
.\] 
\end{Proposition}

\

\begin{pfnkondra} Let $\mathcal{H}$ be the unit ball in $W^{m,p}(\Omega)$ and  $mp< N$. Let $P$ be the extension operator of Theorem \ref{extension}. Set $\mathcal{F}=P(\mathcal{H})$, so that $\mathcal{H}=\mathcal{F}|_\Omega$.  In order to show that $\mathcal{H}$ has compact closure in $L^q(\Omega)$ for $q\in [1,p^*)$ we invoke Theorem \ref{rfc}. Since $\Omega$ is bounded, we may always assume that $p\leq q$.  Clearly, $\mathcal{F}$ is bounded in $W^{m,p}(\R^N)$ by \eqref{iii extension} and thus it is also
bounded in $L^r(R^N)$ with $r\in [p,p^*]$ thanks to the continuous embedding seen in Theorem \ref{cont emb}. Now we need to check that
\[\lim\limits_{|h|\to 0} \|\tau _h u-u\|_{L^q(\R^N)}=0,\qquad\text{uniformly in } u\in\mathcal{F}. \]
By Proposition \ref{desig gradiente}, we have 
\[
\|\tau _h u-u\|_{L^p(\R^N)}\leq |h|\|\nabla u\|_{L^p(\R^N)}\leq |h|\|u\|_{W^{m,p}(\R^N)} ,\qquad\forall v\in \mathcal{F}.
\]
Since $p\leq q<p^*$, we may write 
\[
\frac1q=\frac{\alpha}{p}+\frac{1-\alpha}{p^*}\qquad\text{for some }\quad \alpha\in (0,1].
\]
Thanks to the interpolation inequality (see \cite[pp 93]{Brezis}), we have
\begin{align*}
\|\tau _h u-u\|_{L^q(\R^N)}&\leq\|\tau _h u-u\|_{L^p(\R^N)}^\alpha\|\tau _h u-u\|_{L^{p^*}(\R^N)}^{1-\alpha}\\
&\leq|h|^\alpha\|u\|_{W^{m,p}(\R^N)}^\alpha\left(2\|u\|_{L^{p^*}(\R^N)}\right)^{1-\alpha}\leq C|h|^\alpha,
\end{align*}
where $C$ is independent of $u$ since $\mathcal{F}$ is bounded in $W^{m,p}(\R^N)$ and in $L^{p^*}(\R^N)$. Therefore, the desired conclusion  is obtained by Theorem \ref{rfc}. 

The case $mp=N$  reduces to the same analysis substituting $p^*$ by a large enough number $l$. This is possible thanks to the continuous embedding $W^{m,p}(\RN)\hookrightarrow L^r(\R^{N})$ for all $r\in [p,\infty)$.

\end{pfnkondra}

 Notice that in Theorem \ref{compact emb} the region $\Omega$ must be bounded. We also have another result that holds for $\Omega=\RN$ but it is only true for a particular subspace of $W^{m,p}(\RN)$, as we will see below.
\begin{Theorem}\label{radial compact embe} {\normalfont\cite[Theorem II.1.]{Lions-JFA82}} Suppose $m\geq 1$,  $1\leq p<\infty$ and $N\geq 2$. Let $W^{m,p}_r(\RN)$ be the subspace of the radially symmetric functions belong $W^{m,p}(\RN)$. Then 
\begin{equation}\label{compact n>2}
W_r^{m,p}(\RN)\hookrightarrow\hookrightarrow L^q(\RN),\qquad  
\text{with}\quad \ p<q<p^*.
\end{equation}
\end{Theorem}
In the one-dimensional case (N=1), we do not have the compact embedding \eqref{compact n>2} but it holds if we work in the subspace of the non-increasing radially symmetric functions of $W^{m,p}(\R)$, i.e., we have the following result.

\begin{Theorem}\label{caso N=1} Suppose $m\geq 1$ and  $1\leq p<\infty$. Let $W^{m,p}_{rd}(\R)$ be the subspace of the non-increasing radially symmetric functions belong $W^{m,p}(\R)$. Then 
$$W^{m,p}_{rd}(\R)\hookrightarrow\hookrightarrow L^{q}(\R),\qquad \text{with}\quad \ p<q<\infty.
$$
\end{Theorem}
 A similar result to Theorem \ref{caso N=1} is proved in \cite{lions} and the author proposes the idea of the proof for the one-dimensional case.  Now we present some results in order to prove the above theorem according to these ideas.
\begin{Theorem}\label{teorema1}{\normalfont  \cite[Theorem A.I.]{lions}}
Let $P,Q:\R\to\R$ be two continuous functions satisfying 
\begin{equation}\label{c1}
\dfrac{P(s)}{Q(s)}\to0,\ \ \  \text{as}\ |s|\to\infty,
\end{equation}\label{c2}
and $u_n$ be a sequence of measurable functions from  $\R^N$ to $\R$ such that
\begin{equation}\label{c2}
q=\sup_n \int\limits_{\R^N}|Q(u_n(x))| dx<\infty,
\end{equation}
and
\begin{equation}\label{c3}
P(u_n)\xrightarrow[\ n\ ]{} v\ \ \ \text{in}            \ \ \R^N.
\end{equation}
Then, for any bounded Borel set $B$, one has
\[
\int\limits_{B}|P(u_n(x))-v(x)|dx\xrightarrow[\ n\ ]{}0.
\]
If one further assume that
\begin{equation}\label{c4}
\dfrac{P(s)}{Q(s)}\to0,\ \ \  \text{as}\ s\to0,
\end{equation}
and
\begin{equation}\label{c5}
u_n(x)\to0,\ \ \  \text{as}\ |x|\to\infty, \  \text{uniformly  with respect to } n,
\end{equation}
then, $P(u_n)$ converges to $v$ in $L^1(\R^N)$ as $n\to \infty$.

\end{Theorem}


\begin{Lemma}\label{lema para 2}

If  $u\in L^p(\R^N)$ with $1 \leq  p < \infty$ is a non-increasing radially symmetric function, then 
\[
|u(x)|\leq |x|^{-N/p}\left(\frac{N}{\omega_N'} \right)^{1/p} \|u\|_{L^p},\ \ \ \forall\ x\neq0. 
\] 
where $\omega_N'$ is the surface of the unit sphere in $\R^N$.

\end{Lemma}

\begin{pfn}
Setting $r=|x|$, we have
\[
 \omega_N'|u(r)|^p\frac{r^N}{N} \le \omega_N'\int_0^r|u(s)|^p s^{N-1}ds\le\|u\|^p_{L^p},
\] 
which concludes the proof.
\end{pfn}
\begin{Theorem}[Brezis-Lieb]\label{Brezis-Lieb}{\normalfont \cite[Theorem 1]{brezis-lieb}} Suppose $u_n\to u$ a.e. and $\|u_n\|_{L^p} \le M<\infty$ for all $n\in\N$ and for some $p\in[1,\infty)$. Then 
\begin{equation}
\| u_n \|_{L^p}^p-
\|  u_n-u\|_{L^p}^p
\xrightarrow[\ n\ ]{}
\|  u\|_{L^p}^p.
\end{equation}
\end{Theorem}
Knowing the previous results, we can prove Theorem \ref{caso N=1}.

\

\begin{pfnt} The continuous embedding
$$
W_{rd}^{m,p}(\R)\hookrightarrow L^q(\R),\qquad\text{with}\quad p\leq q<\infty,$$ is obtained from Theorem \ref{cont emb}, since $W_{rd}^{m,p}(\R)$ is continuously embedded in $W^{m,p}(\R)$. Now we are going to prove that the embedding is compact.
Thanks to Theorem \ref{reflex} we have the reflexivity of $W^{m,p}(\R)$, then, all bounded subsets of $W^{m,p}(\RN)$ have  weakly compact closure. Thus, from a bounded sequence $u_n\in W_{rd}^{m,p}(\R) $  we can extract a weakly convergent subsequence in $W^{m,p}(\R)$, i.e., there exists $u\in W^{m,p}(\R)$ such that the relabelled subsequence $u_n\rightharpoonup u$.

If we denote by $I_k$ the interval $(-k,k)$, the Rellich-Kondrachov Theorem gives us in particular
\begin{equation}\label{compacta en intervalo}
W^{m,p}(I_k)\hookrightarrow\hookrightarrow L^q(I_k),\qquad\text{with}\quad p\leq q<\infty,\quad \forall k\in \N.
\end{equation}
We will denote by $u_n|_{I_k}$ the restriction of $u_n$ on $I_k$. Notice
that
$$\|u_n|_{I_k}\|_{W^{m,p}(I_k)}\leq\|u_n\|_{W^{m,p}(\R)}<M,\qquad\forall n,k\in \N,$$
for some constant $M$ since $u_n$ is bounded. Fixing $k=1$ and using \eqref{compacta en intervalo}, we can extract a subsequence $u_n^1$ of $u_n$ such that
$$
u_n^1|_{I_1}\xrightarrow[\ n\ ]{} u|_{I_1}\quad\text{in}\ L^q(I_1),\qquad\text{with}\quad 1\leq q<\infty.
$$
Moreover, from a strongly convergent sequence we can extract a subsequence which converges almost everywhere (see \cite[Theorem 4.9]{Brezis}). Thus, we can  assume that 
$$
u_n^1|_{I_1}\xrightarrow[\ n\ ]{} u|_{I_1}\ \ a.e.\ \text{in}\ I_1.$$
Repeating the same procedure we can construct a sequence of subsequences in the form $$\left\{ u^1_n\right\}_{n\in\N}\supset\cdots\supset\left\{u^k_n\right\}_{n\in\N}\supset\left\{u^{k+1}_n\right\}_{n\in\N}\supset\cdots,$$
 such that 
 $$
u_n^k|_{I_k}\xrightarrow[\ n\ ]{} u|_{I_k}\ \ a.e.\ \text{in}\ I_k\qquad \forall k\in \N.$$
Then, we obtain that the subsequence  
\begin{equation}\label{casi donde quiera}
u_n^n\xrightarrow[\ n\ ]{} u\ \ a.e.\ \text{in}\ \R,
\end{equation}
and $u_n^n$ also verifies the weak convergence. We will rewrite $u_n^n$ as $u_n$ for simplicity.

In order to prove strong convergence of $u_n$ we will use  Theorem \ref{teorema1} choosing $P$ and $Q$ as follows
\[
P(s)=|s|^q,\qquad Q(s)=|s|^p+|s|^{q+1},
\]
and we are going to check its hypothesis. Provided $p<q<\infty$ , it is clear that, when $s\to 0$ or $s\to \infty$, we have 
\begin{equation}
\dfrac{P(s)}{Q(s)}=\dfrac{|s|^q}{|s|^p+|s|^{q+1}}\to 0,
\end{equation}
thus, the hypothesis (\ref{c1}) and (\ref{c4}) hold. Using the continuous embedding we obtain that 
\begin{align*}
\int_{\R}|Q(u_n(x))| dx = & \int_{\R}\left(|u_n|^p+|u_n|^{q+1} \right) dx =\|u_n\|_{L^p}^p + \|u_n\|_{L^{q+1}}^{q+1} \\ 
 \le& \|u_{n}\|_{W^{m,p}}^p+C\|u_{n}\|_{W^{m,p}}^p \leq (C+1)M^p,
\end{align*}
hence, hypothesis (\ref{c2}) holds too. Regarding hypothesis (\ref{c3}), it is clear that 
$$
P(u_n)\xrightarrow[\ n\ ]{} P(u)\ \ a.e. \ \text{in}            \ \ \R.
$$
by the continuity of $P$ and \eqref{casi donde quiera}. In order to check the last hypothesis (\ref{c2}), we will use Lemma \ref{lema para 2}, then 
\[
|u_n(x)|\leq \frac{\|u_n\|_{L^p}}{(2|x|)^{\frac1p}} \leq\frac{M}{(2|x|)^{\frac1p}} \ \ \ \forall\ x\neq0.
\]
Now, applying Theorem \ref{teorema1}, we have
\[
\| |u_n|^q-|u|^q \|_{L^1}\xrightarrow[\ n\ ]{} 0,
\]
thus, it follows that 
\begin{equation}\label{covergencia de normas}\|u_n\|_{L^q}^q\xrightarrow[\ n\ ]{}\|u\|_{L^q}^q.
\end{equation}
To conclude the proof, we will use the Brezis-Lieb Theorem. Notice that $\|u_n\|_{L^q}$ is bounded due to the continuous embedding. This fact, together with \eqref{covergencia de normas}, gives us, through Theorem \ref{Brezis-Lieb}, that

\[
\|  u_n-u\|_{L^q}^q\xrightarrow[\ n\ ]{} 0,
\]
which means that $u_n$ converges strongly in $L^q(\R)$ for $p<q<\infty $. Therefore, the embedding is compact. 
\end{pfnt}


\section{Calculus of variations}\
  
 In this section, we will show some basic elements of the calculus of variations that we will use for the variational formulation problems in PDEs.

\subsection{G\^ateaux and Fr\'echet differential}\label{sec:dif}
 Let $X$ be a Banach spaces and let $J: \mathcal{U}\subset X\to \R$ be a map where $\mathcal{U}$ is an open non-empty subset of $X$. Let us denote by $L(X,\R)$ the set of the linear continuous maps from $X$ into $\R$.
 
\begin{Definition} We say that $J$ is G\^ateaux differentiable at $u\in \mathcal{U}$ if there exists $A\in L(X, \R)$ such that 
\[
\lim\limits_{\epsilon\to 0}\frac{J(u+\epsilon h)-J(u)}{\epsilon}=A(h).\]
for all $h\in X$. The map $A$ is uniquely determined and is  called the G\^ateaux differential of $J$ at $u$. We denote it  by $d_GJ(u)$ where $d_GJ(u)[h]=A(h)$.
\end{Definition}
The G\^ateaux differential can be interpreted as a generalization of the usual directional derivative of  differential calculus of several variables.   
\begin{Definition}
We say that the map $J$ is Fr\'echet differentiable at $u\in \mathcal{U}$, if there exists a map $A\in L(X,\R)$ such that
\begin{equation}\label{frechet}
J(u+h)-J(u)-A(h)=o(\|h\|),\footnote{
The symbol $o(x)$ is one of the Landau symbols. It is used to symbolically express the asymptotic behaviour. Given two functions $f(x)$ and $g(x)$, it is said that $f=o(g)$ as $x\to a$, if $\lim_{x\to a}f(x)/g(x)=0$.
}\qquad\text{as}\quad h\to 0.
\end{equation}
In this case, $A$ is the Fr\'echet differential of $J$ at the point $u$, and it is denoted by $dJ(u)$, where  $dJ(u)[h]=A(h)$. 
\end{Definition}
If $J$ is Fr\'echet differentiable at every point $u\in X$ then J is said to be differentiable on $X$.

\begin{Theorem}
If the Fr\'echet differential exist, it is unique.
\end{Theorem}
\begin{pfn}
Suppose that there exist two maps $A,B\in L(X,\R)$ that satisfy \eqref{frechet}. Then, shooing  $h$ such that $\|h\|=1$ we obtain
\[
A(th)-B(th)=o(t),\qquad\text{as}\quad t\to 0,
\]
or equivalently
\[
0=\lim\limits_{t\to0}\frac{A(th)-B(th)}{t}=A(h)-B(h).
\]
Therefore, $A$ and $B$ are two linear functionals  matching the unit sphere and, hence, equal. As a consequence, the Fr\'echet differential is unique.
\end{pfn}

The next proposition establishes the relation between Fr\'echet and G\^ateaux differentiability.

\begin{Proposition}
If $J$ is Fr\'echet differentiable at the point $u\in \mathcal{U}$, then, $J$ is also G\^ateaux differentiable at the point $u$ and $dJ(u)[h]=d_GJ(u)[h]$.
\end{Proposition}
\begin{pfn}
Let $h$ be a vector in $X$ such that $\|h\|=1$. Since $J$ is Fr\'echet differentiable 
\[
\lim\limits_{t\to 0} \frac{J(u+th)-J(u)-dJ(u)[th]}{\|th\|}=0.
\]
Now, using the linearity of $dJ(u)$ and multiplying by a bounded quantity 
\[
\lim\limits_{t\to 0} \frac{|t|}{t}\frac{J(u+th)-J(u)-tdJ(u)[h]}{|t|}=0,
\]
and 
\[
\lim\limits_{t\to 0} \frac{J(u+th)-J(u)}{t}-dJ(u)[h]=0,
\]
hence
\[
dJ(u)[h]=\lim\limits_{t\to 0} \frac{J(u+th)-J(u)}{t}=d_GJ(u)[h].
\]
Therefore, taking into account that $dJ(u)$ and $ d_GJ(u) $ are linear continuous maps that coincide at the unit sphere, we conclude that $dJ(u)[h]=d_GJ(u)[h]$ for all $h\in X$.
\end{pfn}

\begin{remark}
The converse of the above proposition is not true in general, but it holds if, for example, we have the existence and continuity of $d_GJ$ in a neighbourhood of $u$, i.e., if for some neighbourhood $\mathcal{V}\subset X$, the map $\mathcal{V}\to X'$  given by $v\mapsto  d_GJ(u)$ is well defined and continuous; see \cite[Theorem 1.9]{a-prodi}.
\end{remark}

Let $\mathbb{X}=X\times X$ and $J:\mathbb{X}\to \R$. We also  consider the maps
 \[J_u: v\mapsto J(u,v)\qquad\text{and}\qquad J_v: u\mapsto J(u,v).\] 
The partial derivative of $J$ with respect to $u$ (with respect to $v$), at the point $(u,v)\in \mathbb{X}$ is defined by
 \[\partial_uJ(u,v)=dJ_v(u)\qquad(\partial_vJ(u,v)=dJ_u(v)).\]
  where $\partial_uJ(u,v),\partial_vJ(u,v)\in L(X,\R)$. If $J$ is differentiable at the point $(u,v)$ there exists a map $dJ(u,v)\in L(\mathbb{X}, \R)$ such that
\begin{equation}\label{differential uv}
J(u+h_1,v+h_2)-J(u,v)-dJ(u,v)[h_1,h_2]=o\left(\|(h_1,h_2)\|_{\mathbb{X}}\right)
\end{equation}
where $\|\cdot\|_{\mathbb{X}}$ denotes a norm in the product space, for example
$$
\|h_1,h_2\|_{\mathbb{X}}=\max\{\|h_1\|,\|h_2\|\},\qquad\text{with}\quad (h_1,h_2)\in \mathbb{X}.
$$
From \eqref{differential uv}, we obtain
\[
J(u+h_1,v)=J(u,v)+dJ(u,v)[h_1,0]+o\left(\|h_1\|\right),
\] 
that we can rewrite as
\[
J_v(u+h_1)=J_v(u)+dJ(u,v)[h_1,0]+o\left(\|h_1\|\right).
\] 
Thus, $J_v$ is differentiable at the point $u$ and $$dJ(u,v)[h_1,0]=dJ_v(u)[h_1]=\partial_uJ(u,v)[h_1].$$
Analogously, $J_u$ is differentiable at the point $v$ and $$dJ(u,v)[0,h_2]=dJ_u(v)[h_2]=\partial_vJ(u,v)[h_2].$$ Now, using the linearity of the map $dJ(u,v)$, we have
\begin{align}
dJ(u,v)[h_1,h_2]&=dJ(u,v)[(h_1,0)+(0,h_2)]\nonumber\\
&=dJ(u,v)[h_1,0]+dJ(u,v)[0,h_2]\label{partial diff}\\
&=\partial_uJ(u,v)[h_1]+\partial_vJ(u,v)[h_2]\nonumber.
\end{align} 
Furthermore, the following result holds; see \cite{a-prodi}.

\begin{Proposition}
If $J$ possesses the partial derivative with respect to $u$ and $v$ in a neighbourhood $\mathcal{V}$ of $(u,v)$ and the maps $u\mapsto\partial_uJ$ and $v\mapsto\partial_vJ$ are continuous in $\mathcal{V}$, then J is differentiable at $(u,v)$ and
\begin{equation}
dJ(u,v)[h_1,h_2]=\partial_uJ(u,v)[h_1]+\partial_vJ(u,v)[h_2].
\end{equation}
\end{Proposition}  

Let $J:\mathcal{U}\to \R$ be a differentiable map on $\mathcal{U}\subset X$ such that the map $X\to L(X,\R)$, of the form $v\mapsto dJ(v)$, is differentiable at $u\in \mathcal{U}$. Then, the derivative of such a map at $u$ is denoted as the second derivative $d^2J(u)\in L(X,L(X,\R))$. From the canonical isomorphism between $L(X,L(X,\R))$ and $L_2(X,\R)$, the space of the bilinear maps from $X$ to $\R$, we can consider $d^2J(u)\in L_2(X,\R)$. By induction on $k$, we can define the $k-$th derivative $d^kJ(u)$ belonging to $L_k(X,\R)$, the space of $k$-linear maps from $X$ into $\R$. If $J$ is $k$ times differentiable at every point of $\mathcal{U}$, we say that $J$ is $k$ times differentiable on $\mathcal{U}$.

%
%
%
%

\subsection{Critical points and extremes of functionals}
By a functional we mean a correspondence which assigns a definite real or complex number to each function belonging to some class $X$. In this work we will considerate real functional that  take values in some Banach space $X$ of functions, which could be for example a Sobolev space. In general, one could consider functionals defined on open subsets of $X$. But, for the sake of simplicity, in the sequel we will always deal with functionals defined on all of $X$, unless explicitly remarked.  The differential of a functional $J:X\to\R$ is defined as we saw in subsection \ref{sec:dif}.

\begin{Definition}
A critical point of the functional $J:X\to\R$ is a point $z\in X$ such that $J$ is differentiable at $z$ and $dJ(z)=0$. 
\end{Definition}

According to the previous definition, a critical point $z$ satisfies
\[
dJ(z)[h]=0,\qquad\forall h\in X.
\] 
In the applications, critical points turn out to be weak solutions of differential equations. Roughly, we look for solutions of boundary value problems consisting of a differential equation together with some boundary conditions. These equations will have a variational structure: they be the Euler-Lagrange equation of a functional J on a suitable space of functions $X$, chosen depending on the boundary conditions. The critical points of $J$ on $X$ give rise to solutions of these boundary value problems.

If we consider a Hilbert space $E$ and $J\in \mathcal{C}^1(E,\R)$, then, taking into account the Riesz Theorem, for all $u\in E$ there exists a unique element in $J'(u)\in E$ such that
\begin{equation}\label{Riesz}
\bra J'(u),h\ket=dJ(u)[h],\qquad\forall h\in E. 
\end{equation}
The element $J'(u)$, sometimes also denoted by $\nabla J(u)$, is called the gradient of $J$ at $u$. With this notation, a critical point of $J$ is a solution of the equation $J'(u)=0$. The second derivative, which is a symmetric bilinear map, can be also represented as the operator $J''(u):E\to E$, $h\mapsto J''(u)h$ such that
\[
\bra J''(u)h,k\ket=d^2J(u)[h][k],\qquad\forall\ h,k\in E.
\]
\begin{Definition}
We say that a point $z\in X$ is a local minimum (maximum) of the functional $J:X\to \R$ if there exists a neighbourhood $\mathcal{V}$ of $z$ such that 
\[
J(z)\leq J(u)\quad \left( J(z)\geq J(u)\right),\qquad \forall u\in \mathcal{V}\setminus\{z\}.
\] 
If the above inequality is strict, we say that $z$ is a strict local minimum (maximum) of $J$. If this inequality holds for every $u\in X\setminus\{z\}$, $z$ is said to be a global minimum (maximum) of the functional.
\end{Definition}

\begin{Proposition}\label{punto estacionario} If $z\in X$ is a local minimum (maximum) of a functional $J:X\to \R$, and $J$ is differentiable  at $z$, then $z$ is a critical point of $J$.
\end{Proposition}
\begin{pfn}
If $z$ is a minimum, for a fixed $h\in X$, there exists $\delta>0$ such that,  
\[
J(z)\leq J(u+th),\qquad\forall\ |t|<\delta.
\]
Now, taking into account that $J$ is differentiable at $z$, the differential evaluated in the direction $h$ coincides with the limits
\[
0\leq\lim\limits_{t\to 0^+}\frac{J(u+th)-J(u)}{t}=dJ(z)[h]=\lim\limits_{t\to 0^-}\frac{J(u+th)-J(u)}{t}\leq 0.
\]
Therefore $dJ(z)=0$.
\end{pfn}

Next, we state some results dealing with the existence of minima or maxima for coercive and weakly lower semi-continuous  functionals.
\begin{Definition}
A functional $J$ is called coercive if
\[
\lim\limits_{\|u\|_X\to +\infty}J(u)=+\infty.
\]
\end{Definition} 
\begin{Definition}
A functional $J$ is said to be weakly lower semi-continuous if, for every sequence $u_n \in X$ such that $u_n\rightharpoonup u$, the following holds
\[
J(u)\leq \liminf\limits_n J(u_n).
\]
\end{Definition}
\begin{Lemma}\label{lema acotacion}
Let $X$ be a reflexive Banach space and let $J:X\to\R$ be a coercive and weakly lower semi-continuous. Then, $J$ is bounded from below on $X$, i.e., there exists $a\in \R$ such that $J(u)\geq a$ for all $u\in X$.
\end{Lemma}
\begin{pfn}
We suppose by contradiction that there exists a sequence $u_n\in X$ such that $J(u_n)\to -\infty$. Since $J$ is coercive, it follows that $u_n$ is bounded. Thus, by the reflexivity of $X$, there exists an element $u\in X$ and a weakly convergent subsequence (relabelling) such that $u_n\rightharpoonup u$. Now, since $J$ is weakly lower semi-continuous, we infer that
\[
J(u)\leq \liminf J(u_n)=-\infty,
\]
and it is a contradiction. Therefore $J$ is bounded from below. 
\end{pfn}
\begin{Theorem}
Let $X$ be a reflexive Banach space and let $J:X\to\R$ be  coercive and weakly lower semi-continuous. Then, $J$ has a global minimum, i.e., there exists $z\in X$ such that 
\[J(z)=\inf\limits_{u\in X}J(u).\]
Moreover, if $J$ is differentiable at $z$, then $dJ(z)=0$.
\end{Theorem}

\begin{pfn}
From the Lemma \ref{lema acotacion} it follows that 
\[
m=\inf\limits_{u\in X}J(u),
\]
is finite. If we take a minimizing sequence, namely $u_n\in X$ such that $J(u_n)\to m$, we have again by the coercivity of $J$ that $u_n$ is a bounded sequence and $u_n\rightharpoonup u$ for some $u\in X$. Using that $J$ is  weakly lower semi-continuous we obtain
\[
J(u)\leq \liminf J(u_n)=m,
\]
and the last inequality cannot be strict because $m$ is the infimum of $J$ on $X$. Therefore $J$ achieves its global minimum at $z$: $J(z)=m$. Using Proposition \ref{punto estacionario}, we can conclude the proof.
\end{pfn}
\begin{remark}
Since $z$ is a maximum for $J$ if and only if it is a minimum for $-J$, a similar result holds for the existence of maxima, provided $-J$ is coercive and weakly lower semi-continuous
\end{remark}

\subsection{Differentiable manifolds}
In this subsection, we recall some aspects about  differentiable manifolds. 
\begin{Definition}
Let $X$ by a Hilbert space and $I$ a set of indices. A topological space $M$ is a $\mathcal{C}^k$ Hilbert manifold modelled on $X$, if there exists an open covering $\{\mathcal{U}_i\}_{i\in I}$ of $M$ and a family $\psi_i:\mathcal{U}_i\to X$ of mappings such that the following conditions hold
\begin{itemize}
\item $\mathcal{V}_i=\psi_i(\mathcal{U}_i)$ is open in $X$ and $\psi_i$ is an homeomorphism from $\mathcal{U}_i$ onto $\mathcal{V}_i$;
\item $\psi_j\circ\psi_i^{-1}:\psi_i(\mathcal{U}_i\cap\mathcal{U}_j)\to\psi_i(\mathcal{U}_i\cap\mathcal{U}_j)$ is of class $\mathcal{C}^k$.
\end{itemize}
\end{Definition} 
Each pair $(\mathcal{U}_i,\psi_i)$ in the preceding definition is called a local chart, the maps  $\psi_j\circ\psi_i^{-1}$ are the changes of charts and the pair $(\mathcal{V}_i,\psi_i^{-1})$ is called a local parametrization of $M$.  We have been assuming that $X$ is a Banach space and in this case $M$ is said to be a Banach manifold modelled on X. Moreover, in more general situations, each $\psi_i$ could map $\mathcal{U}_i$ in different Hilbert spaces $X_i$. However, on any connected component of $M$, each $X_i$ can be identified through isomorphism with a single Hilbert space $X$ and we will still say that $M$ is modelled on $X$. For the applications that we will use in this work is suffices to consider the specific case in which $M$ is a subset of a Hilbert space $E$ and is modelled on a Herbert subspace $E_1\subset E $. In particular we will limit ourselves to the case where the manifold is defined in the form 
\begin{align}\label{eq:manifold}\tag{M}
\begin{split} M=G^{-1}(0),&
\\
\text{where}\quad G\in \mathcal{C}^1(E,\R),\quad \text{such that} &\quad G'(u)\neq 0, \quad  \forall u\in M.
\end{split}
\end{align}



\begin{Definition} 
Let $M$ be a manifold as \eqref{eq:manifold}. We define the tangent space to $M$ at the point $p\in M$ by
\[
T_pM=\{h\in E\ :\bra G'(p), h\ket=0\}.
\]
\end{Definition}
Note that $T_pM$ is the orthogonal of $G'(p)$ in $E$, thus $T_pM$ is a closed subspace of $E$ and hence this is also a Hilbert space with the same scalar product of $E$. 

\

Given a functional $J:E\to \R$, the constrained derivative of $J$ on a manifold $M\subset E$ at the point $p\in M$, is the restriction to $T_pM$ of the linear map $dJ(p)\in L(E,\R)$, i.e., if we denote this constrained derivative as $d_MJ$, we have that $d_MJ(p)\in L(T_pM,\R)$ and 
\[
d_MJ(p)[h]=dJ(p)[h],\qquad \forall h\in T_pM.
\]
Using again the Riesz theorem we obtain that there exists a unique element in $T_pM$, which we denote by $\nabla _MJ(p)$, such that 
\begin{equation}\label{proyec}
\bra\nabla_MJ(p),h\ket=d_MJ(p)[h],\qquad\forall h\in T_pM.
\end{equation}
The element $\nabla_MJ(p)$ is named the constrained gradient of $J$ on $M$. Moreover, from \eqref{proyec} we have
\[
\bra\nabla_MJ(p),h\ket=\bra J'(p),h\ket,\qquad\forall h\in T_pM,
\]
hence, $\nabla_MJ(p)$ is nothing but the projection of $J'(p)$ on $T_pM$, which can be written as 
\[
\nabla_MJ(p)=J'(p)-\lambda_p G'(p),
\]
where
\[
\lambda_p=\frac{\bra J'(p),G'(p)\ket}{\|G'(p)\|^2}.
\]
Note that $\lambda$ is well defined on $M$ due to the condition $G'(u)\neq0$ for all $u\in M$, thus the above expression for the constrained derivative makes sense. 

\begin{Definition}
A manifold $M\in E$ is said to be of codimension one if it is modelled on a subspace $X$ of codimension one in $E$, i.e., X satisfies
\[
E=X\oplus\bra w\ket,
\]
for some $w\in E$.
\end{Definition}

In particular we have the following result.
\begin{Theorem}\label{TH:cod}
Let $M$ be a manifold of the form \eqref{eq:manifold}. Then $M$ has codimension one.
\end{Theorem}
\begin{proof}
We consider for each point $p\in M$ the map $\psi_p:E\to E$ defined as
\[
\psi_p(u)=u-p-\bra G'(p),u-p\ket w_p+G(u)w_p,\qquad\text{with}\quad
w_p=\frac{G'(p)}{\|G'(p)\|^2}.\]
Note that $\bra G'(p),w_p\ket=1$ and it follows that 
\begin{align*}
\bra G'(p),\psi_p (u)\ket=&\bra G'(p),u-p-\bra G'(p),u-p\ket w_p+G(u)w_p\ket\\
=&\bra G'(p),u-p\ket-\bra G'(p),\bra G'(p),u-p\ket w_p\ket+\bra G'(p),G(u)w_p\ket\\
=&\bra G'(p),u-p\ket-\bra G'(p),u-p\ket\bra G'(p), w_p\ket+G(u)\bra G'(p),w_p\ket\\
=&G(u).
\end{align*}
Thus, we have that $\psi_p (u)\in T_pX$ if and only if $u\in M$ and, hence, the restriction of $\psi_p$ to $M$ mapping $M$ onto $T_pM$. Moreover, $\psi_p(p)=0$, $\psi_p$ is of class $\mathcal{C}^1$ and $d\psi_p(p)=Id\in L(E,E)$. Using the Inverse Function Theorem  (see \cite{a-prodi}) we obtain that $\psi_p$ is locally invertible at $p$, furthermore,   $\psi_p$ induces  a diffeomorphism  between a neighbourhood $\wt{\mathcal{U}}$ of $p$ and a neighbourhood $\wt{\mathcal{V}}$ of $0$. Now, if we define the map $\varphi_p$ as the restriction of $\psi_p^{-1}$ to $\mathcal{V}=\wt{\mathcal{V}}\cap T_pM$, it follows that $M$ is a $\mathcal{C}^1$ manifold with local parametrization given by $(\mathcal{V},\varphi_p)$ at the point $p$. On the other hand we know that the tangent space $T_pM$ for all $p\in M$ are isomorph to some Hilbert space $X$ with codimension one, since
\[
E=T_pM\oplus\bra G'(p)\ket.
\]
Therefore, we have proved that $M$ is a Hilbert manifold modelled on a subspace $X$ of codimension one in $E$. Hence, $M$ has codimension one.
\end{proof}

\begin{Definition}
Let $J:E\to \R$ be a differentiable functional an let $M$ be a smooth Hilbert manifold.  We say that $z\in M$ is a constrained critical point of $J$ on $M$ if 
$$
d_MJ(z)=0.
$$
\end{Definition}
Moreover, a constrained critical point $z$ satisfies the equation $\nabla_MJ(z)=0$. Furthermore $J'(z)$ is orthogonal to the tangent space $T_zM$ and, if $M$ has the form \eqref{eq:manifold}, there exists a constant $\lambda$ such that
\[
J'(z)=\lambda G'(z).
\]
\begin{Definition}\label{Def:contstrained max}
We say that a point $z\in E$ is a local constrained  minimum (maximum) of the functional $J\in C(E,\R)$ on a smooth manifold $M$, if there exists a neighbourhood $\mathcal{V}$ of $z$ such that 
\[
J(z)\leq J(u)\quad \left( J(z)\geq J(u)\right),\qquad \forall u\in (\mathcal{V}\cap M)\setminus\{z\},
\] 
If the above inequality is strict we say that $z$ is a strict local constrained minimum (maximum) of $J$. In case that this inequality holds for every $u\in X\cap M\setminus\{z\}$, $z$ is called a global constrained minimum (maximum) of the functional on $M$.
\end{Definition}
Similarly to Proposition \ref{punto estacionario} we have the following necessary condition to constrained extremes.
\begin{Proposition}\label{punto estacionario constrained} If $z\in X$ is a local constrained minimum (maximum) of a functional $J:E\to \R$ on a smooth manifold $M$, and $J$ is differentiable  at $z$, then $z$ is a constrained critical point of $J$ on $M$. 
\end{Proposition}
\begin{pfn}
Let $(\mathcal{V},\varphi_z)$ be the local parametrization of $M$ at the point $z$ that was used previously in the proof of Theorem \ref{TH:cod}. Recall that
$$
\varphi_z=\psi_z^{-1}|_\mathcal{V}:\mathcal{V}\subset T_zM\to M,
$$
such that $\varphi_z(0)=z$, and
\begin{equation}\label{eq:ident}
d\varphi_z(0)=(d\psi_z(z))^{-1}|_{T_zM}=Id,
\end{equation}
(see \cite[\textsection6.3]{AM}). From the Definition \ref{Def:contstrained max} we deduce that, $z$ is a local constrained minimum (maximum) of $J$ in $M$ if and only if $0$ is a local minimum (maximum) of $J\circ\varphi_z$ on $\mathcal{V}$. Now, applying the Proposition \ref{punto estacionario}, we have that $0$ is a critical point of the functional $J\circ\varphi_z$, i.e.,
\[
d(J\circ\varphi_z)(0)[h]=dJ(z)[d\varphi_z(0)[h]]=0,\qquad\forall h\in T_zM.
\]
Therefore, since $d\varphi_z(0)=Id$, we have
\[
dJ(z)[h]=d_MJ(z)[h]=0,\qquad\forall h\in T_zM,
\]
and we can conclude that $z$ is a constrained critical point of $J$ on $M$.
\end{pfn}

\subsection{Natural constraints}
Frequently, in the variational formulation of a partial differential equation we have that the associated functional is not bounded. A useful technique used to avoid this issue is to find a manifold $M$ such that the constrained functional is bounded and $M$ contains all the critical points.

\begin{Definition}[Natural constraint]
 Let $E$ be a Hilbert space. A manifold $M$ is called a natural constraint for $J$, if $J\in \mathcal{C}^1(E,\R)$ satisfies that every constrained critical point of $J$ on $M$ is indeed a critical point of $J$, namely
 \[
\nabla _MJ(u)=0,\quad u\in M\qquad\Longleftrightarrow\qquad J'(u)=0. 
 \]
\end{Definition}
One of the most used manifolds as natural constraint is the called Nehari Manifold which was introduced by Zeev Nehari in 1960-1961 (see \cite{Nehari,Nehari2}) and defined as follows
\begin{equation}\label{Nehari def}
M=\left\{u\in E\setminus\{0\} \ |\ \bra J'(u),u\ket=0 \right\},
\end{equation}
for some functional $J\in \mathcal{C}^1(E,\R)$.
\begin{Theorem}\label{TH:natural constrain}
Let $J\in \mathcal{C}^2(E,\R)$ for some Hilbert space $E$ and let $M$ be a non-empty manifold defined as \eqref{Nehari def}. If we assume the following conditions:
\begin{equation}\label{condition1}
\exists\  r>0\quad
\text{such that}\quad B_r\cap M=\emptyset,
\end{equation} 
and
\begin{equation}\label{condition2}
\bra J''(u)u,u\ket\neq 0,\quad\forall u\in M,
\end{equation} 
then, $M$ is a natural constraint for $J$.
\end{Theorem}
\begin{pfn}
First, we  show that all constrained critical points of $J$ on $M$ are indeed critical points of $J$. In order to continue with the same notation, we set
\[
G(u)=\bra J'(u),u \ket,
\]
and we have that $G\in \mathcal{C}^1(E,\R)$, thus $M=G^{-1}(0)\setminus\{0\}$. Moreover, for $u\in M$ we can see that
\begin{equation}\label{deriv inner}
\bra G'(u),u\ket=\bra J''(u)u,u\ket+\bra J'(u),u\ket=\bra J''(u)u,u\ket\neq 0,
\end{equation}
and, hence, $G'(u)\neq 0$ for all $u\in M$. This fact and \eqref{condition1} imply that $M$ is a close $\mathcal{C}^1$ manifold of codimension one via Theorem \ref{TH:cod}. If we suppose that $z\in M$ is a constrained critical point of $J$, then
\begin{equation}\label{croyeccion cero}
\nabla_MJ(z)=J'(z)-\lambda_z G'(z)=0.
\end{equation}
Now, considering the following scalar product
\[
\l_z\bra G'(z),z\ket=\bra\l_z G'(z),z\ket=\bra J'(z),z\ket=G(z)=0
\]
we obtain that $\l_z=0$ since $\bra G'(u),u\ket\neq 0$. Therefore $J'(z)=0$ and hence $z$ is a critical point of $J$. Conversely, if we suppose that $J'(z)=0$, then $G(z)=\bra0, z\ket= 0$, thus $z\in M$. Moreover
\[
\lambda_z=\frac{\bra J'(z),G'(z)\ket}{\|G'(z)\|^2}=0,
\]
and, hence, \eqref{croyeccion cero} holds and $z$ is a constrained critical point.
\end{pfn}

\subsection{The Palais-Smale compactness condition}
The existence of constrained critical points is closely related with some compactness condition. In this subsection, we discuss the Palais-Smale condition which we use in the following chapters.

\begin{Definition}\label{Def:ps}
Let $E$ be a Hilbert space and $J\in \mathcal{C}^1(E,\R)$. We say that a sequence $u_n\in E$ is a Palais-Smale sequence if  it satisfies:
\begin{itemize}
\item $J(u_n)$ is bounded in $E$,
\item $J'(u_n)\to 0$ in $E'$.
\end{itemize}
\end{Definition}

\begin{Definition}[Palais-Smale condition]
We say that a functional $J$ satisfies the Palais-Smale condition on $E$, if every Palais-Smale sequence has a convergent subsequence in $E$.
\end{Definition}

Note that, if the functional $J\in \mathcal{C}^1(E.\R)$ satisfies the Palais-Smale condition, then, for all Palais-Smale sequence $u_n$ in $E$ such that $J(u_n)\to c$, there exists an element $u\in E$ and a convergent subsequence (relabelling) $u_n$  such that $u_n\to u$ in $E$. Therefore, by continuity, we have that $J(u)=c$ and $J'(u)=0$. In other words, $u$ is a critical point of $J$ on $E$ and $c$ is said to be a critical level. 


The following principle is an useful tool to obtain a Palais-Smale sequence.

\begin{Theorem}[Ekeland's Variational Principle]\label{principio de eke}{\normalfont \cite{eke}}
Let $J\in \mathcal{C}^{1}(E,\R)$, and let $M$ be a manifold of the form \eqref{eq:manifold}. If $J$ is bounded from below on $M$, then, for every $\epsilon>0$, there exists some point $u_\epsilon\in M$ such that
\[
J(u_\epsilon)\leq \inf\limits_{u\in M} J(u)+\epsilon^2 \qquad\text{and}\qquad\|\nabla_MJ(u_\epsilon)\|\leq\epsilon.\]
\end{Theorem}

Note that, through the Ekeland's Variational Principle, we obtain a Palais-Smale sequence $u_n$ such that 
\begin{equation}\label{m}
J(u_n)\to m=\inf\limits_{u\in M} J(u)>-\infty.
\end{equation}

The next theorem gives us the existence of constrained extremes.
\begin{Theorem}{\normalfont \cite[Theorem 7.12]{AM}}
Let $J\in \mathcal{C}^{1,1}(E,\R)$, and let $M$ be a manifold of the form \eqref{eq:manifold}. If $J$ is bounded from below on $M$ and there exists a Palais-Smale sequence in $M$ satisfying  \eqref{m},
then the infimum $m$ is achieved at some point $z\in M$ and $\nabla_M J(z)=0$. 
\end{Theorem}  

\begin{remark}
In the above theorem the condition $f\in \mathcal{C}^{1,1}(E,\R)$ can be weakened to $f\in \mathcal{C}^{1}(E,\R)$ if $M$ is a $\mathcal{C}^{1,1}$ Hilbert  or Banach Manifold, see \cite[Remarks 7.13, 10.11]{AM} for further information. 
\end{remark}

\subsection{The Mountain Pass Theorem}

In this subsection, we see one of the most useful results to prove the existence of critical points different from minima or maxima. This result is known as the {\it Mountain Pass }Theorem  and it is of particular importance for functionals that are not bounded either from below, or from above.

Let $E$ be a Hilbert space and we consider a functional $J$ with the following geometric features:
\begin{itemize}

 \item[(MP-1)] $J\in \mathcal{C}^1(E,\R)$ with $J(0)=0$ and there exist $r, \rho>0$ such that $J(u)\geq\rho$ for all $u\in S_r$, where
\[S_r=\{u\in E\  :\ \|u\|=r\};\]

\item[(MP-2)] there exists $e\in E$ with $\|e\|>r$ such that $J(e)\leq \rho$.
\end{itemize}
Notice that J might be unbounded from below. We only require that it is bounded from below on $S_r$.

We denote by $\G$ the set of all continuous  paths on $E$ joining $u=0$ and $u=e$ as follows
\begin{equation}
\G=\{\g\in \mathcal{C}([0,1],E)\ |\ \g(0)=0,\ \g(1)=e\}.
\end{equation}
We can see that $\G$ is a non-empty set because the path $\g(t)=te$ belongs to $\Gamma$. We set
\begin{equation}\label{eq:cMP}
c=\inf\limits_{\g\in\G}\max\limits_{t\in[0,1]}J(\g(t)).
\end{equation}
Note that (MP-1) implies 
\[
\max\limits_{t\in[0,1]}J(\g(t))\geq \rho,\qquad\forall\g\in\G,
\]
since all of these paths cross $S_r$, therefore $c\geq\rho>0$. 
On the other hand, we cannot ensure in general that the infimum  in \eqref{eq:cMP} is attained in $\G$, there are examples even in finite dimensional cases that show this fact. To avoid this problem we will use the Palais-Smale compactness condition. 
\begin{figure}[t]
 \centering
  \subfloat[ ]{
   \label{montaña}
     \includegraphics[height=6cm]{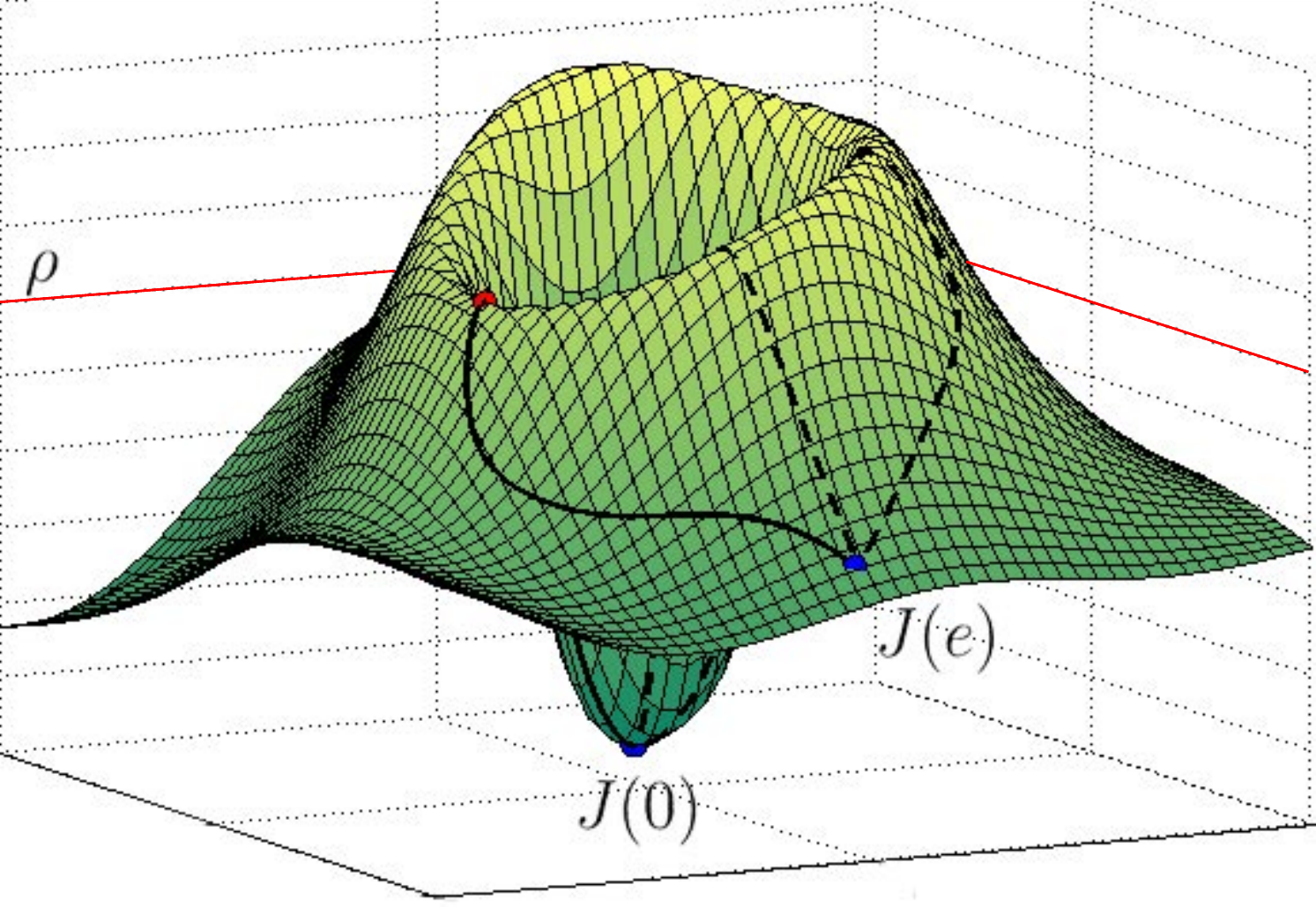}}
  \subfloat[]{
   \label{niveles}
    \includegraphics[height=5.7cm]{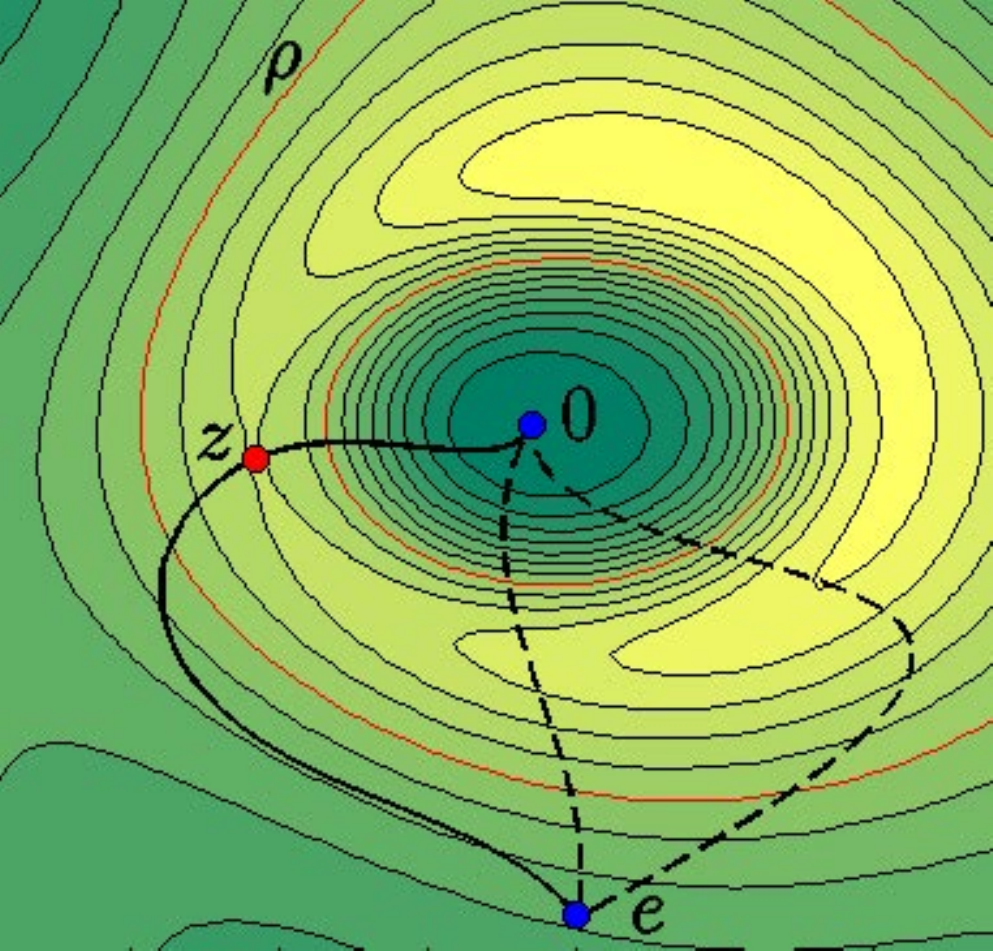}}
 \caption{ Figure (a) is a geometric notion of the Mountain Pass Theorem in $\R^3$ and Figure (b) shows the level isolines of the functional represented in Figure (a). The paths $\g\in\G$ are represented by the black curves (solid and dashed) and $z$ is the Mountain Pass critical point.}
 \label{figura montaña}
\end{figure}
\begin{Theorem}[Mountain Pass]
Let $J$ be a functional that satisfies (MP-1) and (MP-2). Let $c$ be defined as \eqref{eq:cMP} and suppose that the Palais-Smale  condition holds.
Then $c$ is a critical level for $J$. Precisely, there exists $z\in E\setminus\{0,e\}$ such that $J(z)=c$ and $J'(z)=0$. 
\end{Theorem}

The point $z$ in the above Theorem is said to be a Mountain Pass critical point and $c$ a Mountain Pass critical level. Figure \ref{figura montaña} gives us a geometric notion of the Mountain Pass Theorem. See \cite{ar} for further details and proof.

\section{The Schwarz symmetrization}\label{Schwarz}\

In this section we will define the Schwarz symmetrization as well as some of its basic properties. The contents of this section can be found at \cite{Bandle}. Let $\Omega$ be a bounded subset in $\RN$ and we denote by $|\Omega|$ the volume (Lebesgue measure) of $\Omega$ which it is clearly finite.

\begin{Definition}
The symmetrized set of $\Omega$, denoted by $\Omega^\star$, is the ball $$\Omega^\star=\{x\in\RN: |x|< r\},$$ such that $|\Omega^\star|=|\Omega|$. If $\Omega$ is compact, we set $$\Omega^\star=\{x\in\RN: |x|\leq r\}.$$
\end{Definition}
From the above definitions it follows clearly that, if $|\Omega_1|\leq |\Omega_2|$, then $\Omega_1^\star\subset\Omega_2^\star$. In particular it is true when $\Omega_1\subset\Omega_2$.

\begin{Definition}\label{def:schw}
Let $u:\Omega\subset \R^N\to \R$. Then we define the Schwarz symmetrization of $u$ in $\Omega$ as the map $u^\star:\Omega^\star\to \R$ such that
\begin{equation}\label{schwarz symm}
u^\star(x)=\sup\{\kappa:x\in(\Omega_u(\kappa))^\star\},
\end{equation}
where
\begin{equation}\label{defsim}
\Omega_u(\kappa)=\{x\in \Omega: u(x)\geq\kappa\}.
\end{equation}
\end{Definition}
Notice that if $x,y\in\Omega^\star$ and $|x|=|y|$, then
\begin{equation}
u^\star(x)=\sup\{\kappa:x\in(\Omega_u(\kappa))^\star\}=\sup\{\kappa:y\in(\Omega_u(\kappa))^\star\}=u^\star(y),
\end{equation}
thus, $u^\star$ is a radially symmetric function in $\Omega^\star$ and for this reason the term ``symmetrization'' is used. Another observation to take into account is derived from \eqref{defsim}, we have that, if $\kappa_1\geq\kappa_2$, then $\Omega_u(\kappa_1)\subset\Omega_u(\kappa_2)$ and hence $(\Omega_u(\kappa_1))^\star\subset(\Omega_u(\kappa_2))^\star$.
\begin{Proposition}
The Schwarz symmetrization $u^\star$ is a radially non-increasing function.
\end{Proposition}
\begin{pfn}
We suppose that $x,y\in\Omega^\star$ and $|x|< |y|$. Then, for all $\kappa\in \R$ such that $y\in (\Omega_u(\kappa))^\star$, from the symmetries $(\Omega_u(\kappa))^\star$, it follows that $x\in (\Omega_u(\kappa))^\star$, therefore
\begin{equation}
u^\star(x)=\sup\{\kappa:x\in(\Omega_u(\kappa))^\star\}\geq\sup\{\kappa:y\in(\Omega_u(\kappa))^\star\}=u^\star(y).
\end{equation}
\end{pfn}

The set defined by
$$
\Omega^\star_{u}(\kappa)=\{x\in\Omega^\star:u^\star(x)\geq\kappa\}, 
$$
is clearly a ball, due to the radial symmetries of $u^\star$, and $\Omega^\star_{u}(\kappa)$ has the same volume than $\Omega_u(\kappa)$ as we will prove below. First,  we will show that $\Omega_{u}^\star(\kappa)=(\Omega_u(\kappa))^\star$. If $x\in \Omega_{u}^\star(\kappa)$, then $\kappa\leq u^\star(x)=\sup\{\kappa':x\in(\Omega_u(\kappa'))^\star\}$. Thus, there exists $\kappa_1\geq\kappa$ such that $x\in(\Omega_u(\kappa_1))^\star\subset(\Omega_u(\kappa))^\star$. Conversely, if $x\in (\Omega_u(\kappa))^\star$, from \eqref{schwarz symm} we have that $u^\star (x)\geq\kappa$ and hence $x\in \Omega_{u}^\star(\kappa)$. Therefore $\Omega_{u}^\star(\kappa)=(\Omega_u(\kappa))^\star$ and finally we obtain
\begin{equation}\label{quantity}
|\Omega_u(\kappa)|=|(\Omega_u(\kappa))^\star|=|\Omega_{u}^\star(\kappa)|.
\end{equation}
Taking into account the above equalities we say that the functions $u$ and $u^\star$ are equimeasurable. In the sequel, if no confusion arises, we will denote the quantity \eqref{quantity} only by $a(\kappa)$. The fact of $u$ and $u^\star$ being equimeasurable implies the following Lemma.
\begin{Lemma}\label{th:sim1}{\normalfont \cite[pp. 49]{Bandle}}
Let $\psi(t)$ be a continuous real function, then
\begin{equation}
\int_\Omega\psi(u(x))\,dx =\int_{\Omega^\star}\psi(u^\star(x))\,dx.
\end{equation} 
\end{Lemma}
\begin{Corollary}\label{corolario sim}
Taking $\psi(t)=t^p$ with $1\leq p< \infty$ in Lemma \ref{th:sim1}, we obtain
\begin{equation}\label{normas sim}
\|u\|_{L^p(\Omega)}=\|u^\star\|_{L^p(\Omega^\star)}.
\end{equation}
\end{Corollary}

If $u$ is continuous then the function $a(\kappa)$ is strictly decreasing and has discontinuities only for those values of $\kappa$ for which the set $\{x\in\Omega:u(x)=\kappa\}$ has a non-vanishing volume. Let $\kappa(a)$ be the inverse function of $a(\kappa)$, and in the points $\kappa_0$ such that $a(\kappa_0^+)=a_1< a_2=a(\kappa_0^-)$ we complete the definition of $\kappa(a)$ by setting $\kappa(a)=\kappa_0$ for all $a\in[a_1,a_2]$. 
\begin{Lemma}\label{proposym}
Let $x\in\Omega^\star$. Then
\[
u^\star(x)=\kappa(\omega_N|x|^N)
\]
where $\omega_N$ is the volume of the unit ball in $\R^N$.
\end{Lemma} 
\begin{pfn}
Let us take the closed ball $\overline{B}_{|x|}\in\RN$ with  radius $|x|$ and we set
\[
a_x:=|\overline{B}_{|x|}|=\omega_N|x|^N\qquad\text{and}\qquad\kappa_x:=\kappa(a_x).
\]
Recalling that
\[
a_x=a(\kappa_x)=|(\Omega_u(\kappa_x))^\star|,
\]
and taking into account the spherical form of the set $(\Omega_u(\kappa_x))^\star$, we obtain that $(\Omega_u(\kappa_x))^\star=\overline{B}_{|x|}$ and obviously $x\in(\Omega_u(\kappa_x))^\star$. Hence
\begin{equation}
u^\star(x)=\sup\{\kappa:x\in(\Omega_u(\kappa))^\star\}\geq\kappa_x.
\end{equation}
If we assume that the above inequality is strict, then there exists $\m_y>\kappa_x$ such that $x\in(\Omega_u(\kappa_y))^\star$. Moreover, since $a(\kappa)$ is strictly decreasing we have $a(\kappa_y)<a(\kappa_x)$, thus 
$$x\in \overline{B}_{|y|}\subset \overline{B}_{|x|},
\qquad \text{with}\qquad |y|<|x|,$$
and this is a contradiction. Therefore 
$$u^\star(x)=\kappa_x=\kappa(a_x)=\kappa(\omega_N|x|^N).\qedhere$$
\end{pfn}
\begin{Lemma}\label{lema desig}
Let $\psi$ be a non-decreasing continuous real function and let $S\subset \Omega$ be an arbitrary region of volume $V$. Then, 
\begin{equation}\label{eq:desig S}
\int_S\psi(u(x))\,dx\leq\int_{S^*}\psi(u^\star(x))\,dx=\int_0^{V}\psi(\kappa(a))\,da.
\end{equation}
\end{Lemma}
It is important to note that the Schwarz symmetrization in the above theorem is taken in $\Omega$ not in $S$.

\begin{pfn}
We first consider the case where $\kappa(a)$ is non-constant in a neighbourhood of $V$, namely the set $\{x\in\Omega:u(x)=\kappa(V)\}$ has a vanishing volume. Then 
$$\big|\Omega_u(\kappa(V))\big|=V=\big|S\big|,$$
and consequently 
\begin{equation}\label{volumendif}
\big|S\,\setminus \,\Omega_u(\kappa(V))\big|=\big|\Omega_u(\kappa(V))\,\setminus \,S\big|.
\end{equation}
Moreover, from \eqref{defsim} we have in particular 
\begin{equation}\label{desigsym}
u(x)<\kappa(V)\qquad\forall\, x\in S\,\setminus\,\Omega_u(\kappa(V)),
\end{equation}
thus, using \eqref{volumendif} and \eqref{desigsym} and the fact that $\psi$ is a non-decreasing function, we obtain
\begin{equation}
\int_{S\,\setminus \,\Omega_u(\kappa(V))}\psi(u(x))\,dx\leq\int_{\Omega_u(\kappa(V))\, \setminus  \,S}\psi(u(x))\,dx,
\end{equation}
hence
\begin{equation}\label{defigsym}
\int_S\psi(u(x))\,dx\leq\int_{\Omega_u(\kappa(V))}\psi(u(x))\,dx.
\end{equation}
We also have $\big(\Omega_u(\kappa(V))\big)^\star=S^\star$, thus applying the Lemmas \ref{th:sim1} and  \ref{proposym}, it follows that
\begin{equation}\label{lema y prop}
\int_{\Omega_u(\kappa(V))}\psi(u(x))\,dx=\int_{S^\star}\psi(u^\star(x))\,dx=\int_0^{V}\psi(\kappa(a))\,da.
\end{equation}
Therefore the assertion is obtained from \eqref{defigsym} and \eqref{lema y prop}. If $\kappa(a)$ is constant in a maximal interval $(a_1,a_2)$ which contains $V$, then 
$$\big|\Omega_u(\kappa(V))\big|=a_2>V.$$
 In this case we replace $\Omega_u(\kappa(V))$ by a region $\Omega'$ of volume $V$ such that
$$
\Omega_u(\kappa(a_1^-))\subset\Omega'\subset\Omega_u(\kappa(a_2^+)),
$$
and the procedure is the same. 
\end{pfn}

Next we will show an example where the inequality of Lemma \ref{lema desig} takes place.  
\begin{Example}
Let $\Omega=(0,1)$, $S=(0,\frac12)$ and 
$$u(x)=\left\{
\begin{array}{cc}
1 &\text{if }\ x\in \left(\frac12, 1\right) \\ 
0 &\text{if }\ x\in \left(0, \frac12\right]. 
\end{array} 
\right.$$
We can see that the symmetrized sets in this case are $\Omega^\star=(-\frac12,\frac12)$, and $S^\star=(-\frac14,\frac14)$. We also have that 
$$u^\star(x)=\left\{
\begin{array}{cl}
1 &\text{if }\ x\in \left(-\frac14, \frac14\right) \\ 
0 &\text{if }\ x\in \left(-\frac12, -\frac14\right]\cap \left[\frac14,\frac12\right). 
\end{array} 
\right.$$
Taking $\psi$ as the identity function we obtain that  
\[
\int_S\psi(u(x))\,dx=\int\limits_0^{1/2} u(x)\, dx=0<1= \int\limits_{-1/4}^{1/4}u^\star(x)\, dx=
\int_{S^*}\psi(u^\star(x))\,dx,
\]
which verify the inequality of \eqref{eq:desig S}.
\end{Example}

We will introduce the following integral identity which we will use later.
\begin{Lemma}\label{lema sin demo}{\normalfont \cite[Lemma 2.3]{Bandle}}
Let $u$ and $v$ be real-valued functions defined in $\Omega$ with $u$ integrable over $\Omega$ and $v$ measurable over $\Omega$, satisfying the bound condition $-\infty<a\leq v(x)\leq b<+\infty$. Then
\begin{equation}
\int_\Omega uv \,dx=a\int_\Omega u\,dx+\int
_{a}^{b}\left(\int_{\Omega_v(\kappa)}u\,dx\right)d\kappa,
\end{equation}
or equivalently
\begin{equation}
\int_\Omega uv \,dx=b\int_\Omega u\,dx-\int
_{a}^{b}\left(\int_{\Omega\,\setminus\,\Omega_v(\kappa)}u\,dx\right)d\kappa.
\end{equation}
\end{Lemma}
\begin{Theorem}\label{lemma mult}
Let $u,v$ be continuous functions in $\Omega$ and let $v$ satisfy the bound condition of the Lemma \ref{lema sin demo}. Then
\begin{equation}\label{desig2}
\int_\Omega uv\,dx\leq\int_{\Omega^\star}u^\star v^\star\,dx.
\end{equation}
\end{Theorem}
\begin{pfn}
Thanks to Lemma \ref{lema sin demo} we have 
\[
\int_\Omega uv \,dx=a\int_\Omega u\,dx+\int
_{a}^{b}\left(\int_{\Omega_v(\kappa)}u\,dx\right)d\kappa,
\]
and
\[
\int_{\Omega^\star} u^\star v^\star \,dx=a\int_{\Omega^\star} u^\star\,dx+\int
_{a}^{b}\left(\int_{\Omega_{v}^\star(\kappa)}u^\star\,dx\right)d\kappa.
\]
Now, using the Lemma \ref{th:sim1}, we obtain
\begin{equation}
a\int_\Omega u\,dx=a\int_{\Omega^\star} u^\star\,dx,
\end{equation}
and, taking into account the Lemma \ref{lema desig}, it follows that
\begin{equation}
\int_{\Omega_v(\kappa)}u\,dx\leq \int_{\Omega_{v}^\star(\kappa)}u^\star\,dx.
\end{equation}
Therefore inequality \eqref{desig2} holds.
\end{pfn}

\begin{Lemma}\label{Lipschitz}{\normalfont \cite[Lemma 2.1]{Bandle}}
If $u$ is a non-negative Lipschitz continuous map that vanishes on $\partial \Omega$, then $u^\star$ is also a Lipschitz continuous map with the same Lipschitz constant.
\end{Lemma}

\begin{remark}
It is important to note that all results in this section are presented assuming that $\Omega$ is a bounded set. The concept of symmetrized set cannot be extended to unbounded sets in general. Intuitively it is clear that $(\R^N)^\star=\R^N$ but we cannot define the symmetrized set for $\Omega\neq \R^N$  when $|\Omega|=\infty$. However, we can extend the Schwarz symmetrization for a function $u:\R^N\to \R_+$ that vanishes at infinity. Note that here $\Omega_u(\kappa)$ is bounded if $\kappa >0$ and $\Omega_u(\kappa)=\R^N$ if $\kappa \leq0$, in both cases the symmetrized set is well defined and the Schwarz symmetrization can be obtained by \eqref{schwarz symm}. All results presented in this section can be extended naturally to functions $u:\R^N\to \R_+$ that vanish at infinity.
\end{remark}

\begin{Theorem}\label{th:deriv}{\normalfont \cite[Lemma 1]{talenti}}
Let $u\in \mathcal{C}^\infty(\R^N)$ a non-negative real valued function that vanishes at infinity. Then 
\begin{equation}
\int_{\R^N}|\nabla u|^p\,dx\geq \int_{\R^N}|\nabla u^\star|^p\,dx.
\end{equation}
\end{Theorem}

The above result is known as the P\'olya-Szeg\"o inequality, because it was first used in 1945 by G. P\'olya and G. Szeg\"o to prove that the capacity of a condenser diminishes or remains unchanged by applying the process of Schwarz symmetrization (see \cite{polya}).

For the proof of Theorem \ref{th:deriv} we need some results from the theory of functions of several real variables. Setting $\Omega=\R^N$, we need a formula connecting the integral of $|\nabla u|$ with the ($N-1$)-dimensional measure of the boundaries $\partial \Omega_u(\kappa)$ or the ($N-1$)-dimensional measure of the set $\{x\in\Omega:u(x)=\kappa\}$. This formulas are due to Federer \cite{federer} which in our case take the form 
\begin{equation}\label{Federer 1}
\int_\Omega|\nabla u|\,dx=\int\limits_0^{+\infty}\mathcal{H}_{N-1}\{x\in\Omega:u(x)=\kappa\}\,d\kappa,
\end{equation}
where $ \mathcal{H}_{N-1}$ stands for ($N-1$)-dimensional (Hausdorff) measure. A more general version of \eqref{Federer 1} (see \cite{federer}) is 
\begin{equation}\label{Federer 2}
\int_\Omega f(x)|\nabla u|\,dx=\int\limits_0^{+\infty}\,d\kappa\int\limits_{u(x)=\kappa}f(x)\mathcal{H}_{N-1}(dx),
\end{equation}
where $f$ is a real valued integrable function. We point out that the above formulas are valid provided $u$ is a Lipschitz continuous map. In our case this is not an issue since we are supposing that $u$ is a smooth function vanishing  at infinity, hence it is Lipschitz continuous.

Before starting with the proof let us state explicitly some properties of the level sets we have
used. The set $\partial \Omega_u(\kappa)$ is a subset of $\{x\in\Omega:u(x)=\kappa\}$ because of the continuity of $u$. Moreover, the set 
$$\{x\in\Omega:u(x)=\kappa\}\setminus \partial \Omega_u(\kappa), $$ 
only contains critical points of $u$. Hence, if $\{x\in\Omega:u(x)=\kappa\}$ does not contain critical points of $u$, then 
\begin{equation}\label{casi todo mu}
\Omega_u(\kappa)=\{x\in\Omega:u(x)=\kappa\}.
\end{equation} 
Note that if $u\in\mathcal{C}^\infty(\Omega)$, the set of all levels $\kappa$ for which $\{x\in\Omega:u(x)=\kappa\}$  contains critical points of $u$ has one-dimensional measure zero, via Sard's theorem (see \cite{Sard}). Therefore \eqref{casi todo mu} is valid at almost every $\kappa$.

\

\begin{gradientesim}
First, we will prove that the following inequality holds at almost every $\kappa$
\begin{equation}\label{18...}
\int\limits_{u(x)=\kappa}|\nabla u|^{p-1}\mathcal{H}_{N-1}(dx)\geq \left[-a'(\kappa)\right]^{1-p}\left[\mathcal{H}_{N-1}\{x\in\Omega:u(x)=\kappa\}\right]^p.
\end{equation}
where $a$ is defined by $a(\kappa)=|\Omega_u(\kappa)|$ in \eqref{quantity}.
If $p=1$ the above expression is clearly an equality. If $p>1$, we obtain by the H\"older inequality that
\[
\int\limits_{\kappa\leq u(x)<\kappa+h}\frac{|\nabla u|}{h}\,dx\leq\left[\int\limits_{\kappa\leq u(x)<\kappa+h}\frac{|\nabla u|^p}{h}\,dx\right]^\frac1p\left[-\frac{a(\kappa+h)-a(\kappa)}{h}\right]^{1-\frac{1}{p}},
\]
hence, making $h\to 0$ we have 
\begin{equation}\label{desig 121212}
-\frac{d}{d\kappa}
\int_{\Omega_u(\kappa)}|\nabla u|\,dx\leq\left[-\frac{d}{d\kappa}
\int_{\Omega_u(\kappa)}|\nabla u|^p\,dx\right]^\frac1p\left[a'(\kappa)\right]^{1-\frac{1}{p}}.
\end{equation}
On the other hand, from \eqref{Federer 1} we obtain for almost every $\kappa$ that
\[
\int_{\Omega_u(\kappa)}|\nabla u|\,dx=\int\limits_\kappa^{+\infty}\mathcal{H}_{N-1}\{x\in\Omega:u(x)=t\}\,dt,
\]
hence 
\begin{equation}\label{fom feder1}
-\frac{d}{d\kappa}
\int_{\Omega_u(\kappa)}|\nabla u|\,dx=\mathcal{H}_{N-1}\{x\in\Omega:u(x)=\kappa\}.
\end{equation}
Analogously  from \eqref{Federer 2} we have
\begin{equation}\label{a la p}
\int_{\Omega_u(\kappa)} |\nabla u|^p\,dx=\int\limits_\kappa^{+\infty}\,dt\int\limits_{u(x)=t}|\nabla u|^{p-1}\mathcal{H}_{N-1}(dx),
\end{equation}
then
\begin{equation}\label{fom feder2}
-\frac{d}{d\kappa}
\int_{\Omega_u(\kappa)}|\nabla u|^p\,dx=\int\limits_{u(x)=\kappa} |\nabla u|^{p-1} \mathcal{H}_{N-1}(dx).
\end{equation}
Thus,  substituting \eqref{fom feder1} and \eqref{fom feder2} in  \eqref{desig 121212} we obtain \eqref{18...} for almost every $\kappa$.

From \eqref{a la p} it also follows that
\begin{equation}\label{completo1}
\int_\Omega |\nabla u|^p\,dx=\int\limits_0^{+\infty}\,d\kappa\int\limits_{u(x)=\kappa}|\nabla u|^{p-1}\mathcal{H}_{N-1}(dx),
\end{equation}
and now we will use the isoperimetric inequality (see \cite[\textsection3.2.43]{federerlibro}). Recall  that in our case this inequality can be written for almost every $\kappa$ as
\begin{equation}\label{isoperimetric}
N \omega_N^{\frac{1}{N}}[a(\kappa)]^{1-\frac1N}\leq \mathcal{H}_{N-1}\{x\in\Omega:u(x)=\kappa\}.
\end{equation}
Therefore, using the inequalities \eqref{isoperimetric} and \eqref{18...} in \eqref{completo1} we obtain the estimate
\begin{equation}\label{estimate}
\int_\Omega |\nabla u|^p\,dx\geq N^p\omega_N^{\frac{p}{N}}
\int\limits_0^{+\infty}[a(\kappa)]^{p\left(1-\frac{1}{N}\right)}|a'(\kappa)|^{1-p}\,d\kappa.
\end{equation}
Notice that inequality \eqref{estimate} becomes an equality if $u$ is radially symmetric. Indeed, the equality holds in \eqref{isoperimetric} if the level set $\Omega_u(\kappa)$ is a ball, and the equality holds in \eqref{18...} if $|\nabla u|$ is constant on the level surface $\{x\in\Omega:u(x)=\kappa\}$. Note  that, in particular, the Schwarz symmetrization $u^\star$ is a Lipschitz continuous function by  Lemma \ref{Lipschitz} and it satisfies the equality in \eqref{estimate}, i.e., 
\begin{equation}\label{desigualdar con u aster}
\int_{\Omega^\star} |\nabla u^\star|^p\,dx= N^p\omega_N^{\frac{p}{N}}
\int\limits_0^{+\infty}[a(\kappa)]^{p\left(1-\frac{1}{N}\right)}|a'(\kappa)|^{1-p}\,d\kappa.
\end{equation}
Therefore, the desired conclusion is obtained  from \eqref{estimate} and \eqref{desigualdar con u aster}. 
\end{gradientesim}

\begin{remark}
All results showed in this section about the Schwarz symmetrization can be extended to functions in the Sobolev space $W^{m,p}(\Omega)$. It is possible thanks to the theorem of global approximation by smooth functions (see Theorem \ref{Global approximation by smooth functions}).
\end{remark}

\chapter{A system of nonlinear Schr\"odinger--Korteweg-de Vries equations}\label{chap 2to orden}
\markboth{SYSTEM OF NLS--KdV EQUATIONS}{SYSTEM OF NLS--KdV EQUATIONS}

In this chapter we will study in detail the results obtained in \cite{c3}, we will discuss the arguments used and complete some proofs for better understanding of procedures. This chapter deals with a system of coupled nonlinear Schr\"odinger--Korteweg-de Vries equations given by
\begin{equation}\label{NLS-KdV}\tag{S1}
\left\{\begin{array}{rcl}
if_t + f_{xx} + |f|^2f+ \b fg & = &0\\
g_t  +gg_x  + g_{xxx}+\frac 12\b(|f|^2)_x  & = & 0,
\end{array}\right.
\end{equation}
where $f=f(x,t)\in \mathbb{C}$ while $g=g(x,t)\in \mathbb{R}$, and $\b\in \mathbb{R}$ is the  coupling coefficient.  We  look for  solitary traveling waves solutions of \eqref{NLS-KdV} of the form
\be\label{eq:sol-t-v-solutions} 
(f(x,t),g(x,t))=\left(e^{i\o t}e^{i\frac c2 x}u(x-ct),v(x-ct)\right),
\ee 
with $u,v$ real functions and $c,w$ real positive constants. Note that
\begin{equation}
\begin{array}{rcl}
f_t(x,t)&=&e^{i\o t}e^{i\frac c2 x}(iwu(x-ct)-cu'(x-ct)),\\&\\
f_{xx}(x,t)&=&e^{i\o t}e^{i\frac c2 x}\left(-\frac{c^2}{4}u(x-ct)+icu'(c-ct)+u''(x-ct))\right),\\&\\
|f(x,t)|^2f(x,t)&=&e^{i\o t}e^{i\frac c2 x}(u(x-ct))^3,\\&\\
\beta f(x,t)g(x,t)&=&e^{i\o t}e^{i\frac c2 x}\beta u(x-ct)v(x-ct),\\&\\
\frac{1}{2}\beta(|f(x,t)|^2)_x&=&\beta u(x-ct)u'(x-ct),
\end{array} 
\end{equation}
therefore, the first equation of \eqref{NLS-KdV} is equivalent to solving the following ordinary differential equation
\begin{equation}\label{desglose 1}
-\left(w+\frac{c^2}{4}\right) u+u''+u^3+\beta uv=0.
\end{equation}
On the other hand, using \eqref{cambio kdv}, we have that the second equation takes the form
\[
-cv'+vv'+v'''+\beta uu'=0.
\]
Integrating the above equations, under the assumption that $u$ and $v$ vanish at infinity, we obtain
\begin{equation}\label{desglose 2}
-cv+\frac{1}{2} v^2+v''+\frac{1}{2}\beta u^2=0.
\end{equation}
Choosing  $\l_1=\o+\frac{c^2}{4}$ and 
$\l_2=c$, we get from \eqref{desglose 1} and \eqref{desglose 2} that $u,\, v$ solve the following system
\begin{equation}\label{NLS-KdV2}
\left\{\begin{array}{rcl}
-u'' +\l_1 u & = & u^3+\beta uv \\
-v'' +\l_2 v & = & \frac 12 v^2+\frac 12\beta u^2.
\end{array}\right.
\end{equation}

 We focus our attention on the existence of positive even ground and bound states of \eqref{NLS-KdV2} under an appropriate range of parameter settings. We also analyze the extension of system \eqref{NLS-KdV2} to the dimensional cases $N=2,3$.

\section{Functional setting and notation}\label{sec:funct}º

Let $E$ denotes the Sobolev space $W^{1,2}(\R)$ and, taking into account that $\l_1,\l_2>0$, we can check easily that
$$
\|u\|_j=\left(\intR ({u'}^2+\l_j u^2)\,
dx\right)^{\frac 12},\qquad j=1,2,
$$
are norms in $E$ which come from the inner product 
$$
\bra u,v\ket_j=\intR ( u'v'+ \l_j uv)\, dx, \qquad
j=1,2.
$$
\begin{Proposition}
Norms $\|\cdot\|_{W^{1,2}}$ and $\|\cdot\|_j$, for $j=1,2$ are equivalent in $E$. 
\end{Proposition}
\begin{pfn}
Fixing $j$, if we suppose that $\l_j\geq 1$, for $u\in E$ we have
\[
\|u'\|_{L^2}^2+\|u\|_{L^2}^2\leq\|u'\|_{L^2}^2+\l_j\|u\|_{L^2}^2\leq\l_j\left(\|u'\|_{L^2}^2+\|u\|_{L^2}^2\right),
\]
then
\[
\|u\|_{W^{1,2}}\leq \|u\|_j^2\leq\l_j\|u\|_{W^{1,2}}.
\]
Similarly, if $\l_j\leq 1$ we obtain
\[
\l_j\|u\|_{W^{1,2}}\leq \|u\|_j^2\leq\|u\|_{W^{1,2}}.
\]
Therefore, the norms are equivalent.
\end{pfn}

Let us define the product Sobolev space $\E=E\times E$. The elements in $\E$ will be denoted by $\bu =(u,v)$, and $\bo=(0,0)$.
Recall that the product space $\E$ is also a Hilbert space with the inner product
\begin{equation}\label{inner product}
\bra\bu_1,\bu_2\ket=\bra u_1,u_2\ket_1+\bra v_1,v_2 \ket_2,
\end{equation}
which induces the norm
$$
\|\bu\|=\sqrt{\|u\|_1^2+\|v\|_2^2}.
$$
Let $\bu=(u,v)\in \E$, the notation $\bu\geq \bo$, respectively $\bu>\bo$, means that $u,v\geq 0$, respectively $u,v>0$. Let $H$ be the subspace $W^{1,2}_r(\R)$ of radially symmetric functions in $E$, and $\h=H\times H$.

We define the following functionals which are respectively associated with the equations in system \eqref{NLS-KdV2} without coupling,
$$
I_1(u)=\frac 12 \|u\|_1^2 -\frac 14\, \intR u^4dx,\qquad I_2(v)=\frac 12 \|v\|_2^2 -\frac 16\, \intR v^3dx,\qquad u,\, v\in E,
$$
and the associated functional for the system \eqref{NLS-KdV2} can be written as follows,
\begin{equation}\label{J}
J (\bu)= I_1(u)+I_2(v)- \frac 12\b \intR u^2v\,dx,\qquad \bu\in \E.
\end{equation}
Also we can write 
\begin{equation}\label{J con Q}
J (\bu)=\frac 12\|\bu \|^2 -G_\b (\bu),\quad \bu\in \E,
\end{equation}
where
$$
Q_\b (\bu)=\frac 14\, \intR u^4dx+\frac 16\, \intR v^3dx+ \frac 12\b \intR u^2v\,dx, \qquad \bu\in \E.
$$
Notice that $I_1,I_2$ and $J$ are differentiable on $\E$ and their differentials at $\bu=(u,v)\in \E$ are given by
\begin{align}
\displaystyle DI_1(u)[h_1]&=\intR ({u'h_1'}+\l_1 uh_1)\,dx-\intR u^3h_1dx\label{di1},\\
\displaystyle  DI_2(v)[h_2]&=\displaystyle  \intR ({v'h_2'}+\l_2 vh_2)\,dx-\frac 12\intR v^2h_2dx,\label{di2}
\end{align}
and
\begin{align}\label{Dphi}
\begin{split}
dJ(\bu)[\bh]=&\partial_uJ(\bu)[h_1]+\partial_vJ(\bu)[h_2]\\ 
=&DI_1(u)[h_1]-\b \intR uvh_1\,dx
+DI_2(v)[h_2]-\frac 12\b \intR u^2h_2\,dx.
\end{split}
\end{align}
We set 
\begin{equation}\label{J1,J2}
P_1(u)=DI_1(u)[u],\qquad P_2(v)=DI_2(v)[v],
\end{equation}
and
\begin{align}\label{Psi}
\begin{split}
G(\bu)=&dJ(\bu)[\bu]=P_1(u)+P_2(v)-\frac 32\b \intR u^2v\,dx\\
=&\|\bu\|^2-\intR u^4dx-\frac 12\intR v^3dx-\frac 32\b \intR u^2v\,dx.
\end{split}
\end{align}
\begin{Definition}
We say that $\bu\in \E$ is a non-trivial {\it bound state} of \eqref{NLS-KdV2} if $\bu$ is a non-trivial  critical point of $J$. A bound state $\wt{\bu}$ is called ground state if its energy is minimal among all the non-trivial bound states, namely
\begin{equation}\label{eq:gr}
J(\wt{\bu})=\min\{J(\bu): \bu\in \E\setminus\{\bo\},\; J'(\bu)=0\}.
\end{equation}
\end{Definition}
\section{Nehari manifold and key results}\label{sec:key}\

We will work mainly  in  $\h$ thus, using \eqref{Nehari def} and \eqref{Riesz}, we will take the the Nehari manifold as follows
,
\begin{equation}\label{Nehari}
\cN =\{ \bu\in \h\setminus\{\bo\}: G (\bu)=0\}.
\end{equation}

\begin{Proposition}
The Nehari manifold $\cN$ is a natural constraint for the functional $J$.
\end{Proposition}
\begin{pfn}
We will prove the Proposition through Theorem \ref{TH:natural constrain}.
For all $\bu,\bh\in \E$ we have,
\begin{equation}\label{eq:DG}
dG(\bu)[\bh]= 2\bra\bu,\bh\ket-4\intR u^3h_1dx-\frac 32\intR v^2h_2dx-3\b \intR uvh_1\,dx-\frac 32\b \intR u^2h_2\,dx,
\end{equation}
but in particular, if $\bu=\bh$ and $\bu\in\cN$, we combine the above expression with the fact $G(\bu)=0$ and we obtain 
\begin{equation}\label{eq:gamma}
dG(\bu)[\bu]=dG(\bu)[\bu]-3G(\bu)= - \|\bu \|^2-\intR u^4\,dx<0,\quad\forall\, \bu\in \cN.
\end{equation}
Now, the above inequality jointly with \eqref{deriv inner} and \eqref{Riesz}, it follows that $\cN$ is a smooth manifold locally near any point $\bu\not= \bo$ with $G(\bu)=0$. The second derivatives have the form
\begin{align}
\displaystyle d^2I_1(u)[h_1][k_1]&=\intR ({h_1'k_1'}+\l_j h_1k_1)\,dx-3\intR u^2h_1k_1\,dx,\label{ddi1}\\
\displaystyle  
d^2I_2(v)[h_2][k_2]&=\displaystyle  \intR ({h_2'k_2'}+\l_j h_2k_2)\,dx-\intR vh_2k_2\,dx,\label{ddi2}
\end{align}
and
\begin{align}\label{DDphi}
\begin{split}
d^2J(\bu)[\bh][\bk]=&d^2I_1(u)[h_1][k_1]+d^2I_2(v)[h_2][k_2]\\
&-\b \intR vh_1k_1\,dx-\b \intR uh_2k_1\,dx-\b \intR uh_1k_2\,dx.
\end{split}
\end{align}
Evaluating at $\bu=\bo$ we obtain
\[
d^2J(\bo)[\bh]^2=d^2I_1(0)[h_1]^2+d^2I_2(0)[h_2]^2=\|h_1\|^2+\|h_2\|^2=\|\bh\|^2.
\]
Thus, $d^2J(\bo)$ is positive definite, so we infer that $\bo$ is a strict minimum for $J$. As a consequence,  $\bo$ is an isolated point of the set $G^{-1}(0)$, proving that, on the one hand $\cN$ is a smooth complete manifold of codimension $1$, and on the other hand there exists a constant $\rho>0$ so that
\be\label{eq:bound}
\|\bu\|^2>\rho,\qquad\forall\,\bu\in \cN.
\ee
Therefore, since \eqref{eq:gamma} and \eqref{eq:bound}, we can conclude that $\cN$ is a natural constraint for $J$ thanks to the theorem \ref{TH:natural constrain}.
\end{pfn}

\begin{remarks}\label{rem:obs1}

\begin{itemize}
\hspace{1cm}
\item[(i)] It is relevant to point out that working on the Nehari manifold we can combine te expresions \eqref{J con Q} with the fact $G(\bu)=0$ and we get that the functional $J$ restricted to $\cN$ takes the form \be\label{eq:restriction0}
J(\bu)=J(\bu)-\frac{1}{3}G(\bu)= \frac 16\|\bu\|^2+\frac{1}{12}\intR u^4dx,\qquad \forall\, \bu\in\cN,
\ee and substituting   \eqref{eq:bound} into \eqref{eq:restriction0} we
have
\begin{equation}\label{eq:restriction}
J(\bu)\ge   \frac 16\|\bu\|^2>\frac 16 \rho,\qquad \forall\, \bu\in\cN.
\end{equation}
Therefore, \eqref{eq:restriction} shows that the functional $J$ is bounded from below  on $\cN$, so one can try to minimize the restricted functional 
\begin{equation}\label{F}
 F=J|_{\cN},
\end{equation} on the Nehari manifold.
\item[(ii)] Notice that the full Nehari manifold
\begin{equation}\label{fullNehari}
\cM =\{ \bu\in \E\setminus\{\bo\}: G (\bu)=0\},
\end{equation}
defined over $\E$ instead $\h$, satisfies the same proprieties than $\cN$, namely equations  \eqref{eq:gamma}, \eqref{eq:restriction0} \eqref{eq:restriction} holds for $\bu\in\cM$ and it is also a natural constraint for $J$. Analogously we can define the functional 
\begin{equation}\label{Ftecho}
 \overline{F}=J|_{\cM},
\end{equation} and try to minimize it on the full Nehari manifold.

\item[(iii)] With respect to the Palais-Smale  condition, we recall that in the one dimensional case,
one cannot expect a compact embedding of $E$ into $L^q(\R)$ for $2<
q<\infty$. Indeed, working on $H$ (the radial or even case) is not
true too; see \cite[Remarque I.1]{Lions-JFA82}. However, we will
show that for a Palais-Smale sequence we can find a subsequence for which the
weak limit is a solution. This fact jointly with some properties of
the Schwarz symmetrization will permit us to prove the existence of
positive even ground states in Theorem \ref{th:1}. With some extra work one could also consider the non-negative radially decreasing functions, where one has the required compactness thanks to Berestycki and Lions \cite{lions}.
\end{itemize}
\end{remarks}
Due to the lack of compactness mentioned above in Remark \ref{rem:obs1}-$(iii)$, we state a measure theory result given in \cite{lions2}  that we will use in the proof of Theorem \ref{th:1}.
\begin{Lemma}\label{lem:measure}
If $2<q<\infty$, there exists a constant $C>0$ so that
\begin{equation}\label{eq:measure}
\intR |u|^q \, dx\le C\left( \sup_{z\in\R}\int_{|x-z|<1}|u(x)|^2dx\right)^{\frac{q-2}{2}}
\| u\|^2_{E},\quad \forall\: u\in E.
\end{equation}
\end{Lemma}
\

Taking into account the form of the second equations of  \eqref{NLS-KdV2} we  note that the system only admit semi-trivial solutions coming from the equations $-v''+\l_2v=\frac 12 v^2$, namely the possible semi-trivial solutions has the form $(0,V_2)$ where $V_2$ is the solution of the uncoupled second equation $-v''+\l_2v=\frac 12 v^2$. Recall that $V_2$ is obtained as follows
\be\label{eq:segunda}
V_2(x)=2\l_2\,V(\sqrt{\l_2}\,x),
\ee 
where $V$ is the unique positive even solution of equation $V''-V+V^2=0$ given by \eqref{v semitrivial}; see \cite{kw}. Hence $\bv_2 = (0,V_2)$ is a particular solution of  \eqref{NLS-KdV2} for any $\b\in\R$, and
moreover, it is the unique non-negative semi-trivial solution of
\eqref{NLS-KdV2}. We also define the following Nehari manifold  corresponding to the second equation
\begin{equation}\label{N2}
\cN_2 =\left\{v\in H\setminus\{0\} : P_2(v)=0\right\}=\left\{v\in H\setminus\{0\} : \|v\|_2^2 -\frac 12\intR v^3dx=0\right\},
\end{equation}
and define the tangents spaces
\[
T_{\bv_2}\cN=\left\lbrace 
\bh\in \h:dG(\bv_2)[\bh]=0
\right\rbrace\quad\text{and}\quad T_{V_2}\cN_2=\left\lbrace 
h\in H:dP_2(V_2)[h]=0
\right\rbrace.
\]
\begin{Lemma}\label{lemma tangent}
Let $\bh=(h_1,h_2)\in \E$. Then
\begin{equation}\label{eq:tang1}
\bh\in T_{\bv_2} \cN  \Longleftrightarrow h_2\in T_{V_2} \cN_2.
\end{equation}
\end{Lemma}
\begin{pfn}
We have
\[
\begin{array}{ccl}
\bh\in T_{\bv_2} \cN & \Longleftrightarrow & dG(\bv_2)[\bh]=0\\
&\Longleftrightarrow& 2\bra V_2,h_2\ket_2-\frac{3}{2}\int_{\R}V_2^2h_2=0\\
&\Longleftrightarrow& dP_2(V_2)[h_2]=0\\
& \Longleftrightarrow & h_2\in T_{V_2} \cN_2.\\
\end{array} 
\]\end{pfn}

Now we are going to see how is the geometry of the functional $J$ around the point $\bv_2$ depending of the parameter $\beta$.

\begin{Proposition}\label{lem:gl3}
There exists $\L>0$ such that
\begin{itemize}
\item[(i)]  if $\b< \L$, then $\bv_2$ is a strict local minimum  of $J$ constrained on $\cN$,
\item[(ii)] for any   $\b>\L$, then $\bv_2$ is a  saddle point of $J$ constrained  on $\cN$. Moreover,  
\begin{equation}\label{eq:infimo B>L1}
\dyle\inf_{\cN}J<J(\bv_2).
\end{equation}
\end{itemize}
\end{Proposition}
\begin{pf} \hspace{1cm}
\begin{itemize}
\item[(i)]
 We define
\be\label{eq:Lambda} \L=\inf_{\varphi\in H\setminus\{
0\}}\frac{\|\varphi\|_1^2}{\intR V_2\varphi^2dx}. \ee One has that
for $\bh\in  T_{\bv_2}\cN$,
\be\label{eq:Phi-segunda}
d^2_{\cN}J(\bv_2)[\bh]^2 =\|h_1\|_1^2 +d^2_{\cN_2}I_2(V_2)[h_2]^2-\b\intR V_2
h_1^2dx. \ee Let us take $\bh=(h_1,h_2)\in T_{\bv_2}\cN$, by
\eqref{eq:tang1} $h_2\in T_{V_2} \cN_2$, then using that $V_2$ is
the minimum of $I_2$ on $\cN_2$, there exists a  constant $c>0$ so
that
\be\label{eq:minimo-pos}
d^2_{\cN_2}I_2 (V_2)[h_2]^2\ge c\|h_2\|_2^2.
\ee
Due to \eqref{eq:Lambda} we obtain that,
$$\int_{\R}V_2h_1^2\leq \|h_1\|^2_1/\L,\qquad\forall\,h_1\in H.$$
Thus, substituting both previous inequalities in \eqref{eq:Phi-segunda} we arrive at
\be\label{eq:desig seg der}
d^2_{\cN}J(\bv_2)[\bh]^2\geq\left(1-\frac{\beta}{\L} \right) \|h_1\|_1^2+c\|h_2\|_2^2,\qquad \forall\ \bh\in T_{\bv_2}\cN.
\ee
Moreover, since $\beta<\L$, then $1-\beta/\L>0$ and $d^2_{\cN}J(\bv_2)$ is positive definite. Therefore, $\bv_2$ is a strict local minimum of $J$ on $\cN$.

\item[(ii)]
Since $\beta >\L$, there exists $\wt{h}\in H\setminus \{0\}$ such that
$$
\L< \frac{\|\wt{h}\|_1^2}{\int_{\R} V_2\wt{h}^2dx}<\b.
$$  
Using the equivalence \eqref{eq:tang1}, we have $\bh_1=(\wt{h},0)\in T_{\bv_2}\cN$ and
$$
d^2_{\cN}J(\bv_2)[\bh_1]^2 =\|\wt{h}\|_1^2 -\b_{\R^N} V_2 \wt{h}^2dx<0.$$
On the other hand, taking $h_2\in T_{V_2}\cN_2$ not equal to zero, then $\bh_2=(0,h_2)\in T_{\bv_2}\cN$ and
$$d^2_{\cN}J(\bv_2)[\bh_2]^2=I_2'' (V_2)[h_2]^2\ge c\|h_2\|_2^2>0.
$$
Consequently, this is sufficient to conclude that $\bv_2$ is a saddle point of $J$ on $\cN$ and obviously inequality \eqref{eq:infimo B>L1} holds.
\end{itemize}
 \end{pf}
\section{Existence results}\label{sec:mainR}\
\subsection{Existence of ground states}\label{sec:ground}

Concerning the existence of ground state solutions of
\eqref{NLS-KdV2}, the first result is the following.
\begin{Theorem}\label{th:1}
Suppose that $\b>\L$, then system \eqref{NLS-KdV2} has a positive
even ground state $\wt{\bu}=(\wt{u},\wt{v})$.
\end{Theorem}
\begin{pf}
To prove this theorem we will consider the full Nehari manifold $\cM$ (see Remark \ref{rem:obs1}-$(ii)$) and we divide the proof into two steps. In the first step, we prove that $\inf_{\cM}J$ is achieved at some positive function $\wt{\bu}\in\E$, while in the second step, we show that $\wt{\bu}$ can be taken even.

{\it Step 1.}
By the Ekeland's variational principle (see Theorem \ref{principio de eke}), there exists a minimizing Palais-Smale sequence $\bu_n$ in $\cM$, i.e.,
\be\label{eq:PS1}
J (\bu_n)\to c=\inf_{\cM}J,
\ee
\be\label{eq:PS2}
\n_{\cM}J(\bu_n)\to 0.
\ee
By \eqref{eq:restriction0} and \eqref{eq:PS1}, easily one finds that $\{\bu_n\}$ is a bounded sequence on $\E$, and relabeling, we can assume that
\begin{align*}
\bu_n\rightharpoonup \bu &\quad\text{in}\,\,\E,\\
\bu_n\to \bu &\quad\text{in}\,\,\mathbb{L}^q_{loc}(\R),\\
\bu_n\to \bu& \quad a.e.\ \  \text{in}\, \,\R,
\end{align*}
where 
$$\mathbb{L}^q_{loc}(\R)=L^q_{loc}(\R)\times  L^q_{loc}(\R),\qquad1\le q<\infty.$$ Moreover, we know that
\begin{equation}\label{four}
\nabla_{\cM} J(\bu_n)=J'(\bu_n)-\l_n G'(\bu_n),
\end{equation}
and, thanks to the Cauchy-Schwarz inequality, \eqref{eq:PS2} and \eqref{eq:bound} we obtain
\[
\bra\nabla_{\cM} J(\bu_n),\bu_n\ket \leq \|\nabla_{\cM} J(\bu_n)\|\|\bu_n\|\to 0\qquad\text{as}\quad n\to\infty.\]
We also have, $\bra J'(\bu_n),\bu_n\ket=G(\bu_n)= 0$ because $\bu_n\in {\cM}$, thus, 
\[
\bra\nabla_{\cM} J(\bu_n),\bu_n\ket =-\l_n \bra G'(\bu_n),\bu_n\ket\to 0,\qquad\text{as}\quad n\to\infty,\]
but $|\bra G'(\bu_n),\bu_n\ket|> \rho$ by \eqref{eq:gamma}, then, $\l_n\to 0$ as $n\to\infty$. 

Now, we will show that $\|G'(\bu_n)\|$ is bounded. From the Riesz theorem it follows that
\begin{equation}\label{eq:normsup}
\|G'(\bu_n)\|=\|dG(\bu_n)\|=\sup\limits_{\substack{\bh\in \E\\ \|\bh\|=1}}\big|dG(\bu_n)[\bh]\big|,
\end{equation}
where, applying triangular inequality in \eqref{eq:DG}, we deduce 
\begin{align*}
\big|dG(\bu_n)[\bh]\big|\leq & 2|\bra\bu_n,\bh\ket|+4\|u^3_nh_1\|_{L^1}+\frac 32\|v_n^2h_2\|_{L^1}+3\b \| u_nv_nh_1\|_{L^1}+\frac 32\b \|u_n^2h_2\|_{L^1},
\end{align*}
thus, by H\"older inequality
\begin{align*}
\big|dG(\bu_n)[\bh]\big|\leq & 2\|\bu_n\|\|\bh\|+4\|u_n^3\|_{L^{4/3}}\|h_1\|_{L^4}+\frac 32\|v_n^2\|_{L^2}\|h_2\|_{L^2}\\&+3\b \|u_n\|_{L^3}\|v_n\|_{L^3}\|h_1\|_{L^3}+\frac 32\b\|u_n^2\|_{L^2}\|h_2\|_{L^2}\\
\leq & 2\|\bu_n\|\|\bh\|+4\|u_n\|^{3}_{L^{4}}\|h_1\|_{L^4}+\frac 32\|v_n\|^2_{L^4}\|h_2\|_{L^2}\\&+3\b \|u_n\|_{L^3}\|v_n\|_{L^3}\|h_1\|_{L^3}+\frac 32\b\|u_n\|^2_{L^4}\|h_2\|_{L^2}.
\end{align*}
Notice that all the above norms are well defined since $E\subset L^q(\R)$ for all $q\in [2,\infty]$ by the Theorem \ref{cont emb}. Moreover, using the continuous embedding,  we obtain  constants $C_1,C_2,C_3,C_4$ such that
\begin{align*}
\big|dG(\bu_n)[\bh]\big|\leq & 2\|\bu_n\|\|\bh\|+C_1\|u_n\|^3_1\|h_1\|_1+C_2\|v_n\|_2^2\|h_2\|_2\\&+C_3 \|u_n\|_1\|v_n\|_2\|h_1\|_1+C_4\|u_n\|_1^2\|h_2\|_2,
\end{align*}
where, knowing that 
$$
\|\bh\|^2=\|h_1\|_1^2+\|h_2\|_2^2=1\quad\text{and}\quad \|\bu_n\|^2=\|u_n\|_1^2+\|v_n\|_2^2,
$$
we arrive to  
\begin{align*}
\big|dG(\bu_n)[\bh]\big|\leq  2\|\bu_n\|+C_1\|\bu_n\|^3+\left(C_2+C_3 +C_4\right)\|\bu_n\|^2\quad\text{with}\quad\|\bh\|=1.
\end{align*}
The right part in the above inequality is polynomially dependent of $\|\bu_n\|$ and it is clearly bounded since $\bu_n$ is bounded. From \eqref{eq:normsup} we obtain that $\|G'(\bu_n)\|\leq C<+\infty$, hence,  taking into account that $\nabla_\cN J(\bu_n)$ and $G'(u)$ are orthogonal and the fact $\l_n\to 0$, we deduce from \eqref{four} that
$$
\|J'(\bu_n)\|=\|\nabla_\cN J(\bu_n)\|+|\l_n| \|G'(\bu_n)\|\to 0,\qquad\text{as}\quad n\to\infty.
$$
Therefore $\bu_n$ is also a Palais-Smale sequence of  $J$ in $\E$. 

Let us define $\mu_n(x)=u_n^2(x)+v_n^2(x)$, where $\bu_n=(u_n,v_n)$. We {\it claim} that there is no evanescence, i.e., exist $R, C>0$ so that
\be\label{eq:vanishing}
\sup_{z\in\R}\int_{|z-x|<R}\mu_n (x)dx\ge C>0,\quad\forall n\in\mathbb{N}.
\ee
On the contrary, if we suppose
$$
\sup_{z\in\R}\int_{|z-x|<R}\mu_k(x)dx\to 0,
$$
by Lemma \ref{lem:measure}, applied in a similar way as in \cite{c-fract},  we find that $\bu_n\to \bo$ strongly in
$\mathbb{L}^{q}(\R)$ for any $2<q<\infty$, and as a consequence the weak limit $\bu^*\equiv \bo$. This is a contradiction since $\bu_n \in{\cM}$, and by \eqref{eq:restriction0}, \eqref{eq:restriction}, \eqref{eq:PS1} there holds
$$
0< \frac 17\rho <c+o_n(1)=J(\bu_n)=\overline{ F}(\bu_n),\quad\mbox{with } o_n(1)\to 0 \quad\mbox{as }n\to\infty,
$$
hence \eqref{eq:vanishing} is true and the {\it claim} is proved.

We observe that   we can find a sequence of points
$\{z_n\}\subset\R$ so that by \eqref{eq:vanishing}, the translated sequence $\overline{\mu}_n(x)= \mu_n(x+z_n)$ satisfies
$$
\liminf_{n\to\infty}\int_{B_R(0)}\overline{\mu}_n\ge C >0.
$$
Taking into account that $\overline{\mu}_n\to \overline{\mu}$
strongly in $L_{loc}^1(\R)$, we obtain that
$\overline{\mu}\not\equiv 0$. We can also prove that 
$\overline{\bu}_n(x)=\bu_n(x+z_n)\in \cM$ since the invariance of $G$ under translations and $\overline{\bu}_n$ is a Palais-Smale sequence of  $J$ in $\E$. In fact, by the form of the functional $J$ is clear that
\[
J(\overline{\bu}_n)=J(\bu_n)\to c,\qquad\text{as}\quad n\to\infty,
\]
and, if for all direction $\bh\in \E$ we define $\overline{\bh}(x)=\bh(x-z_n)$, we have that 
\[
dJ(\overline{\bu}_n)[\bh]=dJ(\bu_n)[\overline{\bh}],
\]
hence
\begin{align*}
\|J'(\overline{\bu}_n)\|&=\|dJ(\overline{\bu}_n)\|=\sup\limits_{\substack{\bh\in \E\setminus\{0\}}}\big|dJ(\overline{\bu}_n)[\bh]\big|=\sup\limits_{\substack{\overline{\bh}\in \E\setminus\{0\}}}\big|dJ(\bu_n)[\overline{\bh}]\big|=\|J'(\bu_n)\|\to 0.
\end{align*}
In particular, the weak limit of $\overline{\bu}_k$, denoted by $\overline{\bu}$, satisfies the following conditions thanks to the  weakly lower semi-continuity of the functional $\overline{F}$ defined in \eqref{Ftecho},
$$
J (\overline{\bu})  =  \dyle F(\overline{\bu})
\le  \dyle\liminf_{k\to\infty} F(\overline{\bu}_k)
=  \dyle\liminf_{k\to\infty}J(\overline{\bu}_k)=\dyle\liminf_{k\to\infty}J(\bu_k)= c,
$$
Then, using the Propositions \ref{punto estacionario constrained}, we have that $\overline{\bu}$ is a constrained critical point of $J$ in $\cM$. Furthermore, by \eqref{eq:infimo B>L1} we know that necessarily
\begin{equation}\label{fact1}
 J (\overline{\bu})\leq c\leq \inf_{\cN}J<J(\bv_2).
\end{equation}
 Taking into account the maximum principle in the second equation of \eqref{NLS-KdV2} it follows that $\overline{v}>0$, thus, if we take $\wt{\bu}=|\overline{\bu}|=(|\overline{u}|,|\overline{v}|)=(|\overline{u}|,\overline{v})$, we can check easily that $G(\wt{\bu})= G(\overline{\bu})=0$, so $\wt{\bu}\in \cM$ and 
 \be\label{eq:min}
J(\wt{\bu})=J(\overline{\bu})=\min\{J(\bu)\, :\: \bu\in\cM\},
\ee
so we have  $\wt{\bu}\ge \bo$ is a critical point of $J$. Finally, by the maximum principle applied to the first equation and the fact \eqref{fact1}, we get $\wt{\bu}> \bo$.

\

{\it Step 2.} Let us denote by $\wt{\bu}^\star=(\wt{u}^\star,\wt{v}^\star)$ the Schwarz symmetrization function associated to each component of $\wt{\bu}=(\wt{u},\wt{v})$. Note that it is possible since $\wt{\bu}=(\wt{u},\wt{v})>0$ and both components vanish at infinity. Using the classical properties of the Schwarz symmetrization (see Section \ref{Schwarz}),  it follows from Theorem \ref{th:deriv} and Lemma \ref{th:sim1} that
\[
\|\wt{u}^\star\|_1^2=\intR \left({|\wt{u}'^\star|}^2+\l_1 \wt{u}^{\star2}\right)\leq\intR \left({|\wt{u}'|}^2+\l_1 \wt{u}^2\right)=\|\wt{u}\|_1^2,
\]
and analogously $\|\wt{v}\|_2^2\geq\|\wt{v}^\star\|_2^2$, thus
\begin{equation}\label{eq:primera}
\|\wt{\bu}^{\star}\|^2\le \|\wt{\bu}\|^2.
\end{equation}
Now, using the Lemma \ref{lema desig} and Theorem \ref{lemma mult} we obtain
\be\label{12}
Q_\b(\wt{\bu}^{\star})\ge Q_\b(\wt{\bu})\qquad\text{and}\qquad{G}(\wt{\bu}^\star)\le {G}(\wt{\bu}).
\ee
Since $\wt{\bu}^\star$ is a radially symmetric function we know that there exists a unique $t_0>0$ so that $t_0\,\wt{\bu}^{\star}\in {\mathcal{N}}$. In fact, $t_0$ comes from  $G (t_0\wt{\bu}^\star)=0$, i.e.,
\be\label{eq:t}
\| \wt{\bu}^{\star}\|^2= t_0^2\intR (\wt{u}^\star)^4dx+t_0\left( \frac 12\intR (\wt{v}^\star)^3dx+\frac 32\b \intR (\wt{u}^\star)^2\wt{v}^\star \,dx\right),
\ee
and, using that $G(\wt{\bu})=0$, we have
\be\label{eq:t1}
\| \wt{\bu}\|^2= \intR (\wt{u})^4dx+ \frac 12\intR (\wt{v})^3dx+\frac 32\b \intR (\wt{u})^2\wt{v} \,dx.
\ee
Then, from \eqref{eq:primera},\eqref{eq:t},\eqref{eq:t1} and the fact that $\wt{\bu}>\bo$ and $t_0>0$ we find

\begin{align*}
 t_0^2\intR (\wt{u}^\star)^4dx+t_0\left( \frac 12\intR (\wt{v}^\star)^3dx+\frac 32\b \intR (\wt{u}^\star)^2\wt{v}^\star \,dx\right)\leq\\
\intR (\wt{u})^4dx+ \frac 12\intR (\wt{v})^3dx+\frac 32\b \intR (\wt{u})^2\wt{v} \,dx.
\end{align*}
Thus, clearly $t_0\le 1$ due to the inequalities of the Schwarz symmetrization, and consequently,
\be\label{eq:previa}
J(t_0\,\wt{\bu}^\star)= \frac 16 t_0^2\|\wt{\bu}^\star\|^2+\frac{1}{12}t_0^4\intR (\wt{u}^\star)^4\,dx\leq \frac 16 \|\wt{\bu}\|^2+\frac{1}{12}\intR \wt{u}^4\,dx=J(\wt{\bu}).
\ee
From inequalities \eqref{eq:previa}, \eqref{eq:primera} and the one obtained by the Schwarz 
symmetrization, we get
$$
J(t_0\,\wt{\bu}^{\star})\leq J(\wt{\bu})= \min \{J(\bu)\, :\:\bu\in \cM\},
$$
thus, the above inequality is indeed an equality and the infimum of $J$ on the full Nehari manifold is attained at an even function. Therefore, $t_0\,\wt{\bu}$ is a constrained critical point of $J$ on $\cN$, hence, it is a positive even ground state of $J$.
\end{pf}

The last result in this subsection deals with the existence of
positive  ground states of  \eqref{NLS-KdV2} not only for $\b>\L$,
but also for $0<\b\le\L$, at least for  $\l_2$ large enough.
\begin{Theorem}\label{th:ground2}
There exists $\L_2>0$ such that if $\l_2>\L_2$, System \eqref{NLS-KdV2} has an even ground state
$\wt\bu>\bo$ for every $\b>0$.
\end{Theorem}
\begin{pf} Arguing in the same way as in the proof of Theorem \ref{th:1},
we initially have that  there exists an even ground state
$\wt{\bu}\ge \bo$. Moreover, in Theorem \ref{th:1} for $\b>\L$ we
proved that $\wt{\bu}>\bo$. Now we need to show that for $\b\le \L$
indeed $\wt{\bu}>\bo$ which follows by the maximum principle
provided $\wt{\bu}\neq \bv_2$. Taking into account Proposition
\eqref{lem:gl3}-$(i)$, $\bv_2$ is a strict local minimum, but this
does not allow us to prove that $\wt{\bu}\neq \bv_2$. The   idea
here consists on proving the existence of  a function
$\bu_1=(u_1,v_1)\in\cN$ with $J(\bu_1)<J(\bv_2)$. To do so,
since $\bv_2=(0,V_2)$ is a local minimum of $J $ on $\cN$
provided $0<\b<\L$, we cannot find $\bu_1$ in a neighborhood of
$\bv_2$ on $\cN$. Thus, we define $\bu_1=t(V_2,V_2)$ where  $t>0$ is
the unique value so that $\bu_1\in \cN$.

Notice that  $t>0$  is given by $G(\bu_1)=0$, i.e.,
\be\label{eq:v22} \| (V_2,V_2)\|^2=t^2\intR V_2^4\, dx+\frac 12
t(1+3\b)\intR V_2^3\,dx.
\ee
Moreover, we can write 
\be
\|(V_2,V_2)\|^2=2\|V_2\|_2^2+(\l_1-\l_2)\intR V_2^2\,dx,
\ee
and taking into account that $V_2\in\cN_2$, then $P_2(V_2)=0$, thus 
\be\label{eq:v2}
\|(V_2,V_2)\|^2=\intR V_2^3\,dx+(\l_1-\l_2)\intR V_2^2\,dx,
\ee
hence,
substituting \eqref{eq:v2} into \eqref{eq:v22} we get
\begin{equation}\label{sus}
t^2\intR V_2^4\, dx+\frac 12 t(1+3\b)\intR V_2^3\,dx=\intR V_2^3\,dx+(\l_1-\l_2)\intR V_2^2\,dx.
\end{equation}
Now, using
$$
\intR \cosh^{-8}(x)\,dx=\frac{32}{35},\qquad \intR \cosh^{-6}(x)\,dx=\frac{16}{15},\qquad \intR \cosh^{-4}(x)\,dx=\frac{4}{3},
$$
we obtain from \eqref{eq:segunda} that
$$
\intR V^4_2\,dx=3^4\l_2^4\frac{32}{35}\frac{2}{\sqrt{\l_2}},\qquad \intR \cosh^{-6}(x)\,dx=3^3\l_2^3\frac{16}{15}\frac{2}{\sqrt{\l_2}},$$
$$ \intR \cosh^{-4}(x)\,dx=3^2\l_2^2\frac{4}{3}\frac{2}{\sqrt{\l_2}},
$$
thus, substituting the above expressions in \eqref{sus} and dividing the $L^1$ norm of $V_2$ we find
\be\label{eq:segundo1}
\frac{18}{7}\l_2 t^2+\frac 12 t(1+3\b)-\left(1+5\frac{\l_1-\l_2}{12\l_2}\right)=0.
\ee
The energies  of $\bu_1$, $\bv_2$ are given by
$$
J(t(V_2,V_2))=\frac 16 t^2\left(\intR V_2^3\,dx+ (\l_1-\l_2)\intR V_2^2\,dx\right)+\frac{1}{12}t^4\intR V_2^4\,dx,
$$
$$
J (\bv_2)=\frac{1}{12}\intR V_2^3\,dx.
$$
Thus, we want to prove that for the unique $t>0$ given by \eqref{eq:segundo1} we have
$$
\frac 16 t^2\left(\intR V_2^3\,dx+ (\l_1-\l_2)\intR V_2^2\,dx\right)+\frac{1}{12}t^4\intR V_2^4\,dx<\frac{1}{12}\intR V_2^3\,dx,
$$
then arguing as for \eqref{eq:segundo1}, it is sufficient to prove that the following inequality holds
\be\label{eq:segundo2}
\frac{18}{7}\l_2 t^4+t^2\left(2+5\frac{\l_1-\l_2}{6\l_2}\right)-1<0.
\ee
Using \eqref{eq:segundo1} and the fact that  $2+5\frac{\l_1-\l_2}{6\l_2}>0$ for every $\l_1,\,\l_2>0$, fixed $\b>0$ we have that \eqref{eq:segundo2} is satisfied provided $\l_2$ is sufficiently large, namely $\l_2>\L_2>0$, proving that  $J(\bu_1)<J(\bv_2)$ which
 concludes the result.
 \end{pf}
\subsection{Existence of bound states}\label{sec:bound}

In this subsection we establish  existence of bound states to
\eqref{NLS-KdV2}. The first theorem deals with  a perturbation
technique, in which we suppose that $\b=\e\wt{\b}$, with $\wt{\b}$
fixed and independent of $\e$. Note that $\wt{\b}$ can be negative, and $0<\e\ll1$. Then we rewrite the energy functional
$J$ as $J_\e$ to emphasize its dependence  on $\e$,
$$
J_\e(\bu)=J_0(\bu)-\frac 12\e \wt{\b} \intR u^2v\,dx,
$$
where $J_0=I_1+I_2$.

Let us set $\bu_0=(U_1,V_2)$, where $V_2$ is given by
\eqref{eq:segunda} and $U_1$ is the unique positive solution of
$-u''+\l_1u=u^3$ in $H$; see \cite{cof,kw}. This function $U_1$ has
the following explicit expression, \be\label{eq:segunda2}
U_1(x)=\frac{\sqrt{2\l_1}}{\cosh (\sqrt{\l_1}x)}. \ee Note also that
$U_1$ satisfies the following identity \be\label{eq:gr2}
 \|U_1\|_1=\inf_{u\in H\setminus\{0\}} \frac{\|u\|_1^2}{\left( \intR u^4dx\right)^{1/2}}.
\ee
\begin{Theorem}\label{th:2}
There exists $\e_0>0$ so that for any $0<\e <\e_0$ and
$\b=\e\wt{\b}$, system \eqref{NLS-KdV2} has an even bound state
$\bu_{\e}$ with  $\bu_{\e}\to \bu_0$ as $\e\to 0$. Moreover, if $\b>0$ then $\bu_\e>\bo.$
\end{Theorem}
In order to prove this result, we follow some ideas of \cite[Theorem
4.2]{c} with appropriate modifications.

\

\noindent {\it Proof of Theorem \ref{th:2}.}
It is well known that $U_1$ and $V_2$ are non-degenerate critical points of $I_1$ and $I_2$ on $H$ respectively; \cite{kw}. Plainly, $\bu_0$ is a non-degenerate critical point of $J_0$ acting  on $\h$.
Then, by the Local Inversion Theorem, there exists a critical point $\bu_\e$ of $J_\e$ for any $0<\e<\e_0$ with $\e_0$ sufficiently small; see  \cite{a-m} for more details. Moreover, $\bu_\e\to \bu_0$ on $\mathbb{H}$ as $\e\to 0$. To complete the proof it remains to show that if $\b>0$, then $\bu_\e> \bo$.

Let us  denote the positive part  $\bu_\e^+=(u_{\e}^+,v_{\e}^+)$ and  the negative part $\bu_\e^-=(u_{\e}^-,v_{\e}^-)$. By  \eqref{eq:gr2} we have
\begin{equation}\label{eq:pm1}
\|u_{\e}^\pm\|_1^2 \geq  \|U_1\|_1\left(\intR (u_{\e}^\pm)^4dx\right)^{1/2}.
\end{equation}
Multiplying  the second equation of \eqref{NLS-KdV2} by $v_{\e}^-$ and integrating on $\R$ one obtains
\be
\|v_{\e}^-\|_2^2 =\intR (v_{\e}^-)^3dx + \e \wt{\b}\intR (u_{\e})^2v_{\e}^- dx\le 0,
\ee
thus $\|v_\e^-\|_2=0$ which implies $v_\e=v_\e^+\ge 0$. Furthermore,  $\bu_\e\to\bu_0$ implies
$v_\e\to V_2$, which jointly with the maximum principle gives $v_\e>0$ provided $\e$ is sufficiently small.

Multiplying now  the first equation of \eqref{NLS-KdV2} by $u_{\e}^\pm$ and integrating on $\R$ one obtains
\begin{eqnarray*}
\|u_{\e}^\pm\|_1^2 &=& \intR (u_{\e}^\pm)^4dx + \e \wt{\b}\intR (u_{\e}^\pm)^2v_{\e} \,dx\\
&\leq&\intR (u_{\e}^\pm)^4dx +\e \wt{\b}\left( \intR (u_{\e}^\pm)^4dx\right)^{1/2}\left(   \intR v_{\e}^2 \,dx \right)^{1/2}.
\end{eqnarray*}
This, jointly with \eqref{eq:pm1}, yields
\be\label{eq:primera1}
\|u_{\e}^\pm\|_1^2 \le \frac{\|u_{\e}^\pm\|_1^4}{ \|U_1\|_1^2} + \e\, \theta_\e\;\frac{\|u_{\e}^\pm\|_1^2}{ \|U_1\|_1},
\ee
where
$$
\theta_\e = \wt{\b}\left( \intR v_{\e}^2\right)^{1/2}.
$$
Hence, if  $\|u_{\e}^\pm\|>0$, one infers
\begin{equation}\label{eq:pmm1}
\|u_{\e}^\pm\|_1^2 \ge  \|U_1\|_1^2 +o(1),
\end{equation}
where $o(1)=o_\e (1)\to 0$ as $\e\to 0$.
Using again  $\bu_\e \to \bu_0$, then $u_{\e}\to U_1>0 $, as a consequence, for $\e$ small enough, $\|u_{\e}^+\|>0$. Thus
\eqref{eq:pmm1} gives
\begin{equation}\label{eq:pm11}
\|\bu_{\e}^+\|^2 =\|u_{\e}^+\|_1^2+\|v_{\e}^+\|_2^2 \geq   \|U_1\|_1^2 +o(1).
\end{equation}
Now, suppose for a contradiction,  that $\|u_{\e}^-\|_1>0$. Then as for \eqref{eq:pm11}, one obtains
\begin{equation}\label{eq:pm12}
\|\bu_\e^-\|^2 =\|u_{\e}^-\|_1^2+\|v_{\e}^-\|_2^2 \geq  \|U_1\|_1^2+o(1).
\end{equation}
On one hand, using \eqref{eq:pm11}-\eqref{eq:pm12}, we find
\begin{align}\label{eq:eval-funct}
\begin{split}
J(\bu_\e) & =  \dyle \frac 16 \|\bu_\e\|^2+\frac{1}{12}\intR u_\e^4\,dx \\
& =   \dyle\frac 16 \left[  \|\bu_\e^+\|^2 + \|\bu_\e^-\|^2\right]+\frac{1}{12}\intR [(u_\e^+)^4+(u_\e^-)^4]\,dx \\
 &  \ge\dyle \frac 16  \|\bu_0\|^2 +\frac 16 \| U_1\|_1^2+ \frac{1}{12}\intR U_1^4\,dx  +o(1).
\end{split}
\end{align}
On the other hand,  since $\bu_\e \to \bu_0$ we have
\be\label{eq:final}
J(\bu_\e)= \frac 16 \|\bu_\e\|^2+\frac{1}{12}\intR u_\e^4\,dx\to \frac 16\|\bu_0\|^2+ \frac{1}{12}\intR U_1^4\,dx,
\ee
which is in contradiction with \eqref{eq:eval-funct}, proving that $u_{\e}\geq 0$.

In conclusion, we have proved that $v_\e>0$ and $u_\e\ge 0$. To
prove the positivity of $u_\e$, using once more that $\bu_\e\to
\bu_0$,  and $\b=\e\wt{\b}\ge 0$ we can apply the maximum principle
to the first equation of \eqref{NLS-KdV2}, which implies that
$u_\e>0$, and finally, $\bu_\e>\bo$. \rule{2mm}{2mm}\medskip

From  the existence of a positive ground state established in
Theorem \ref{th:1} for $\b>\L$, and more precisely in Theorem
\ref{th:ground2} for $\b>0$, provided  $\l_2$ is sufficiently large,
we can show the existence of a different positive bound state of
\eqref{NLS-KdV2} in the following.
\begin{Theorem}\label{th:bound2}
 In the hypotheses of Theorem \ref{th:ground2} and $0<\b<\L$, there
exists an even bound state $\bu^*>\bo$ with
$J(\bu^*)>J(\bv_2)$.
\end{Theorem}
\begin{pf}
The positive ground state $\wt\bu$ founded in Theorem
\ref{th:ground2} satisfies $J(\wt{\bu})<J(\bv_2)$ and even
more, if $\b<\L$ by Proposition \ref{lem:gl3}, $\bv_2$ is a strict
local minimum of $J$ constrained on $\cN$. As a consequence, we
have the Mountain Pass geometry between $\wt{\bu}$ and
$\bv_2$ on $\cN$. We define the set of all continuous paths joining
$\wt{\bu}$ and $\bv_2$  on the Nehari manifold by
$$
\G=\{ \g\in \mathcal{C}([0,1],\cN)\ | \ \g(0)=\wt{\bu},\: \g(1)=\bv_2\}.
$$
Thanks to the  Mountain Pass Theorem, there exists a Palais-Smale sequence $\bu_k\subset\cN$, such that
$$
J (\bu_k)\to c,\qquad \n_{\cN}J(\bu_k)\to 0,
$$
where
\be\label{eq:MP-level}
c=\inf_{\g\in\G}\max_{0\le t\le 1}J(\g(t)).
\ee
Plainly, by \eqref{eq:restriction0} the sequence $\{\bu_k\}$ is bounded  on $\h$, and we obtain a weakly convergent subsequence  $\bu_k\rightharpoonup\bu^*\in \cN$.

The  difficulty of the lack of compactness, due to work in the one dimensional case (see Remark \ref{rem:obs1}-$(iii)$), can be circumvent  in a similar way as in the proof of Theorem \ref{th:1}, so we omit the full detail for short. Thus, we find that the weak limit $\bu^*=(u^*,v^*)$ is an even bound state of \eqref{NLS-KdV2}, and clearly, $J (\bu^*)>J(\bv_2)$.

It remains to prove that  $\bu^*>\bo$. To do so, let us introduce the following problem
\begin{equation}\label{NLS-KdV+}
\left\{\begin{array}{rcl}
-u'' +\l_1 u & = & (u^+)^3+\beta u^+v \\
-v'' +\l_2 v & = & \frac 12 v^2+\frac 12\beta (u^+)^2.
\end{array}\right.
\end{equation}
By the maximum principle every nontrivial solution $\bu=(u,v)$ of
\eqref{NLS-KdV+}  has the second component  $v>0$ and the first one
$u\ge 0$.  Let us define its energy functional
$$
J^+ (\bu)=\frac 12\|\bu \|^2 -G_\b (u^+,v),
$$
and consider the corresponding Nehari manifold
$$
\cN^+=\{\bu\in\h\setminus\{\bo\}\, :\: (\n J^+ (\bu)|\bu)=0\}.
$$
Also, we denote
$$
I_1^+(u)=\frac12 \| u\|_1^2-\frac 14\intR (u^+)^4\,dx.
$$
It is not very difficult to show that the properties proved for
$J$ and $\cN$ still hold for $J^+$ and $\cN^+$. Unfortunately,
$J^+$ is not $\mathcal{C}^2$, thus Proposition
\ref{lem:gl3}-$(i)$ does not hold directly for $J^+$. To solve
this difficulty, we are going to prove that $\bv_2$ is a strict
local minimum of $J^+$ constrained on $\cN^+$ without using the
second derivative of the functional. Note that in a similar way as
in \eqref{eq:tang1}, there holds \be \bh=(h_1,h_2)\in
T_{\bv_2}\cN^+\Longleftrightarrow h_2\in T_{V_2}\cN_2. \ee Taking
$\bh\in T_{\bv_2}\cN^+$ with $\|\bh\|=1$, we consider $\bv_\e=(\e
h_1,V_2+\e h_2)$. Plainly, there exists a unique $t_{\e}>0$ so that
$t_\e\bv_\e\in \cN^+$. Thus, we want to prove there exists $\e_1>0$
so that
$$
J^+(t_\e\bv_\e)>J^+ (\bv_2),\qquad \forall\: 0<\e<\e_1.
$$
It is convenient to distinguish if $h_1=0$ or not. In the former case, $h_1=0$, $\bv_\e=(0, V_2+\e h_2)$. Hence $t_\e\bv_\e\in \cN^+\Leftrightarrow t_\e(V_2+\e h_2)\in\cN_2$. Furthermore,
\be\label{eq:anterior}
J^+(t_\e\bv_\e)=I_2(t_\e(V_2+\e h_2))>I_2(V_2)=J(\bv_2)=J^+(\bv_2),
\ee
where the previous inequality holds because $V_2$ is a strict local minimum of $I_2$  on $\cN_2$.

Let us now consider the case $h_1\neq 0$. There holds
\be\label{eq:posterior}
J^+(t_\e \bv_\e)=I_2(t_\e(V_2+\e h_2))+I_1^+(t_\e\e h_1)-\frac 12\b \e^2t_\e^2\intR (h_1^+)^2(V_2+\e h_2)\,dx.
\ee
By \eqref{eq:anterior} and \eqref{eq:posterior} it follows,
\be\label{eq:posterior2}
J^+(t_\e \bv_\e)>J^+(\bv_2)+I_1^+(t_\e\e h_1)-\frac 12\b \e^2t_\e^2\intR (h_1^+)^2(V_2+\e h_2)\,dx.
\ee
To finish, it is sufficient to show that
$$
\mathcal{J} (t_\e\bv_\e):=I_1^+(t_\e\e h_1)-\frac 12\b \e^2t_\e^2\intR ( h_1^+)^2(V_2+\e h_2)\,dx>0\qquad \forall\: 0<\e<\e_1.
$$
Let $\a<1$ be such that $\a>\frac{\b}{\L}$. By \eqref{eq:Lambda} and $\b<\L$ there holds
$$
\b\intR V_2( h_1^+)^2\,dx <\a\| h_1\|_1^2,
$$
then for $\e_1$ smaller than before (if necessary) we have
\be\label{eq:posterior3}
\b\intR (V_2+\e h_2)(h_1^+)^2\,dx <\a\| h_1\|_1^2 \qquad \forall\: 0<\e<\e_1.
\ee
Using  \eqref{eq:posterior3} and the Sobolev inequality, we obtain
$$
\mathcal{J} (t_\e\bv_\e)>\frac 12t_\e^2\e^2\| h_1\|_1^2(1-\a-ct_\e^2\e^2),\quad\mbox{ for a constant } c>0.
$$
Now, taking into account that $t_\e\to 1$ as $\e\searrow 0$, we infer there exists a  constant $c_0>0$ so that
\be\label{eq:final2}
\mathcal{J} (t_\e\bv_\e)>\e^2c_0\|h_1\|_1^2.
\ee
Finally, by \eqref{eq:posterior2}, \eqref{eq:final2} it follows that
$$
J^+(t_\e\bv_\e)>\e^2c_0\| h_1\|_1^2+J^+(\bv_2) > J^+(\bv_2),
$$
which proves that $\bv_2$ is a strict local minimum for $J^+$ on $\cN^+$.

\noindent From the preceding arguments, it follows that $J^+$ has a MP critical point $\bu^*\in \cN^+$, which gives rise to a solution of \eqref{NLS-KdV+}.
In particular, one finds that $u, v\ge 0$.
   In addition, since $\bu^*$ is a MP critical point, one
has that $J (\bu^*)=J^+(\bu^*)>J^+(\bv_2)=J(\bv_2)>0$,
which implies  $u^*\ge 0$ with $u^*\not\equiv 0$, and by the maximum
principle applied to each single equation we get $u^*,\, v^*>0$,
hence $\bu^*>\bo$. \end{pf}

In view of Theorems \ref{th:ground2}, \ref{th:bound2}, some remarks
are in order.

\begin{remarks}\label{rem:11}  In the hypotheses of Theorems \ref{th:ground2}, \ref{th:bound2} we have found the coexistence of two positive solutions, the ground state $\wt{\bu}$ in Theorem \ref{th:ground2} and the bound state $\bu^*$ in Theorem \ref{th:bound2}, proving a non-uniqueness result of positive solutions to \eqref{NLS-KdV2}. This is a great difference with the more studied system of coupled nonlinear Schr\"odinger equations
$$
\left \{
\begin{array}{ll}
- \D u_1+ \l_1 u_1 &=  \mu_1u_1^3+\b u_2^2 u_1\\
- \D u_2 + \l_2 u_2 &= \mu_2  u_2^3+\b u_1^2u_2,
\end{array} \right.
$$
(see for instance
\cite{a,ac1,ac2,bt,chen-zou,c,fl,iko,itanaka,linwei,liu-wang,mmp,sirakov,wy}
and the references therein) for which it is known that there is
uniqueness of positive solutions, under appropriate conditions on
the parameters including the case $\b>0$ small; see more
specifically  \cite{iko,wy}. Indeed, for $\b>0$ small, the ground
state is not positive, and it is given by one of the two
semi-trivial solutions  $(U^{(1)},0)$ or $(0,U^{(2)})$ depending on
if $J (U^{(1)},0)$ is lower or grater than $J(0,U^{(2)})$
which plainly corresponds to $\l_1^{2-\frac N2}\mu_2<\l_2^{2-\frac
N2}\mu_1$ or $\l_1^{2-\frac N2}\mu_2>\l_2^{2-\frac N2}\mu_1$
respectively. Here $U^{(j)}$ is the unique\footnote{See
\cite{cof,kw} for this uniqueness result.} positive radial solution
of $-\D u_j+\l_ju_j=\m_j u_j^3$ in $W^{1,2}(\mathbb{R}^N)$, for
$N=1,\,2,\, 3$ and $j=1,\, 2$.

\end{remarks}

\section{Extended system}\label{sec:extended}\

Note that  System \eqref{NLS-KdV}  has no sense in the dimensional case $N=2,\, 3$, however, \eqref{NLS-KdV2}  makes sense to be extended to more dimensions. Moreover, previous results can be established in the dimensional case $N=2,3$ with minor changes for system
\begin{equation}\label{NLS-KdV2-n}
\left\{\begin{array}{rcl}
-\D u +\l_1 u & = & u^3+\beta uv \\
-\D v +\l_2 v & = & \frac 12 v^2+\frac 12\beta u^2,
\end{array}\right.
\end{equation}
working on the corresponding Sobolev Spaces $E=W^{1,2}(\R^N)$, $N=2,3$ and its radial subspace $H=E_r$. In particular, Theorems \ref{th:1}, \ref{th:ground2}, \ref{th:2} and \ref{th:bound2} can be obtained for $N=2,3$ in a less complicated way due to the compact embedding of $H$ given by Theorem \ref{radial compact embe}. Thus, we obtain the corresponding positive radially symmetric bound and ground state solutions.

\begin{remarks}\label{rem:compacidad}\
\begin{itemize}

\item[(i)] Following some ideas by Ambrosetti and Colorado in \cite{ac2}, as Liu and Zheng cited in \cite{lz}, they proved a partial result on existence of solutions
to the corresponding system \eqref{NLS-KdV2} in the dimensional case $N=2,\, 3$. Precisely, in \cite{lz} the authors showed that the infimum of the energy functional on the corresponding Nehari manifold (defined on the radial Sobolev space) is achieved by a non-negative bound state, although it was not shown that the infimum on the Nehari Manifold is a ground state, i.e., the least energy solution
 of the functional that we have proved here for $N=1,\,2,\,3$. Also, in \cite{lz} was not investigated
 the existence of other bound states, as he have done in this manuscript, not only in the non-critical dimensions $N=2,\,3$ but also in the one dimensional case, $N=1$, which is the relevant case as the application in physics dealing with the interaction between  the short and long capillary - gravity water waves.
\item[(ii)] System \eqref{NLS-KdV2-n} can be seen as the stationary system of two coupled
nonlinear Schr\"odinger  equations when one looks for solitary wave solutions, and  $(u,v)$ are the corresponding standing wave solutions. It is well known that time-dependent systems of
nonlinear Schr\"odinger  equations have applications in some aspects of Optics, Hartree-Fock theory for Bose-Einstein
 condensates, among other physical phenomena; see for instance the earlier mathematical works
 \cite{Ack,a,ac1,ac2,acr,bt,fo,linwei,mmp,sirakov}, the more recent list (far from complete)
 \cite{chen-zou,itanaka,liu-wang} and references therein. See also
 \cite{bor,rep} for some recent results on nonlinear Schr\"odinger equations, and also \cite{gmp} for other results including higher-order nonlinear Schr\"odinger equations.

 \end{itemize}
\end{remarks}


\chapter{A higher order system of nonlinear Schr\"odinger--Korteweg-de Vries equations}\label{chap 4to}
\markboth{HIGHER ORDER SYSTEM OF NLS--KdV EQUATIONS}{HIGHER ORDER SYSTEM OF NLS--KdV EQUATIONS}
\textbf{Publication.} The results presented in this chapter correspond to the content of the submitted paper \cite{acf}. 

\

In this chapter we will analyze the existence of solutions of a higher order system coming from \eqref{NLS-KdV}. More
precisely, we consider the following system
\begin{equation}\label{NLS-KdV edu}\tag{S2}
\left\{\begin{array}{rcl}
if_t -  f_{xxxx} + |f|^2f+ \b fg & = & 0\\
g_t-g_{xxxxx}+|g|g_x+\frac12\beta(|f|^2)_x& = & 0
\end{array}\right.
\end{equation}
where $f=f(x,t)\in \mathbb{C}$ while $g=g(x,t)\in \mathbb{R}$, and $\b\in \mathbb{R}$ is the  coupling coefficient.
We look for ``standing-traveling" wave
solutions of the form
$$
(f(x,t),g(x,t))=\left(e^{i\l_1 t} u(x),v(x-\l_2t)\right),
$$
where $u,v$ are real functions and $\l_1,\l_2$ real positive  parameters.  Performing the change of variable we have
\begin{equation}
\begin{array}{rcl}
if_t(x,t)&=&-\l_1e^{i\l_1 t} u(x),\\&\\
f_{xxxx}(x,t)&=&e^{i\l_1 t} u^{(iv}(x),\\&\\
|f(x,t)|^2f(x,t)&=&e^{i\l_1 t} (u(x))^3,\\&\\
\beta f(x,t)g(x,t)&=&\beta e^{i\l_1 t} u(x)v(x-\l_2t),\\&\\
g_{xxxx}(x,t)&=&v^{(iv}(x-\l_2t),\\&\\
|g(x,t)|g(x,t)&=&|v(x-\l_2t)|v(x-\l_2t),\\&\\
|f(x,t)|^2&=&(u(x))^2,
\end{array} 
\end{equation} 
where $w^{(iv}$ denotes the fourth derivative of $w$. Then, the first equation of \eqref{NLS-KdV edu} takes the form 
\begin{equation}\label{convert la prmera} 
u^{(iv} +\l_1u  = u^3+\beta uv.
\end{equation}
On the other hand, the second equation of \eqref{NLS-KdV edu} can be written as
\[g_{xxxxx}+\l_2g_x=\frac{1}{2}(|g|g)_x+\frac12\beta(|f|^2)_x, \]
where integrating we obtain
\[g_{xxxx}+\l_2g=\frac{1}{2}|g|g+\frac12\beta|f|^2 ,\]
which is equivalent to
\begin{equation}\label{convert la segunda}
v^{(iv}+\l_2v  =\frac{1}{2}|v|v+\frac{1}{2}\beta u^2.
\end{equation}
We arrive at the fourth-order stationary system
 \be \left\lbrace
\begin{array}{ccl}\label{eq:NLS-KdV222}
u^{(iv} +\l_1u & =& u^3+\beta uv\\
v^{(iv}+\l_2v & =& \frac{1}{2}|v|v+\frac{1}{2}\beta u^2. 
\end{array}
 \right.
\ee 
 Although
system \eqref{NLS-KdV edu} only makes physical sense in dimension $N=1$, passing
to the stationary system \eqref{eq:NLS-KdV222}, it makes sense to
consider it in higher dimensional cases, as the following,
\be
\left\lbrace
\begin{array}{ccl}\label{eq:NLS-KdV2}
\Delta^2u +\l_1u & =& u^3+\beta uv\\
\Delta^2v+\l_2v & =& \frac{1}{2}|v|v+\frac{1}{2}\beta u^2,
\end{array}
 \right.
\ee
 where $u,v\in W^{2,2}(\R^N)$, $1\le N\le 7$, $\l_j>0$ with
$j=1,2$ and $\beta>0$ is the coupling parameter.

As we shall see, system \eqref{eq:NLS-KdV2} has a non-negative semi-trivial solution $\bv_2=(0,V_2)$ where $V_2$ is a radially symmetric ground state of the equation $\Delta^2v+\l_2v = \frac{1}{2}|v|v$. Then, in order
to find non-negative bound or ground state solutions, we need to
check that they are different from $\bv_2$.


\section{Functional setting and notation}\label{sec:2}\

Let us redefine $E$ as the Sobolev space $W^{2,2}(\R^N)$ then, we define the following equivalent norms and inner products in $E$ as follows

\[\bra u,v\ket_j=\int_{\R^N}\Delta u\cdot\Delta v\ dx+\l_j\int_{\R^N}uv\,dx,\qquad\|u\|_j^2=\bra u,u\ket_j,\qquad j=1,2.\]
Let us define the product Sobolev space $\E:=E\times E$ and we will take the following inner product in $\E$,
\begin{equation}\label{inner product}
\bra\bu_1,\bu_2\ket=\bra u_1,u_2\ket_1+\bra v_1,v_2 \ket_2,
\end{equation}
which induces the following norm
$$
\|\bu\|=\sqrt{\|u\|_1^2+\|v\|_2^2}.
$$
We denote by $H$ the space of radially symmetric functions in $E$, and $\h=H\times H$. The functional associated to both equations in \eqref{eq:NLS-KdV2}, without the coupling term, take the forms 
$$
I_1(u)=\frac 12 \|u\|_1^2 -\frac 14\, \int_{\R^N} u^4dx,\qquad I_2(v)=\frac 12 \|v\|_2^2 -\frac 16\, \int_{\R^N} |v|^3dx,\qquad u,\, v\in E,
$$
respectively and, hence, the complete energy functional associated to system \eqref{eq:NLS-KdV2} is
\begin{equation}\label{Phi2-2}
J (\bu)= I_1(u)+I_2(v)- \frac 12\b \int_{\R^N} u^2v\,dx,\qquad \bu\in \E.
\end{equation}
Notice that $I_1,I_2$ and $J$ are differentiable on $\E$ and their differentials at $\bu=(u,v)\in \E$ are given by
\begin{align}
\displaystyle dI_1(u)[h_1]&=\int_{\R^N} ({\Delta u\cdot \Delta h_1}+\l_1 uh_1)\,dx-\int_{\R^N}u^3h_1dx\label{di1-2},\\
\displaystyle  dI_2(v)[h_2]&=\displaystyle  \int_{\R^N} (\Delta v\cdot \Delta h_2+\l_2 vh_2)\,dx-\frac 12\int_{\R^N} |v|vh_2dx,\label{di2-2}
\end{align}
and
\begin{align}\label{Dphi-2}
\begin{split}
dJ(\bu)[\bh]
=&dI_1(u)[h_1]-\b \int_{\R^N} uvh_1\,dx
+dI_2(v)[h_2]-\frac 12\b \int_{\R^N} u^2h_2\,dx.
\end{split}
\end{align}
We set 
\begin{equation}\label{J1,J2-2}
P_1(u)=dI_1(u)[u],\qquad P_2(v)=dI_2(v)[v],
\end{equation}
and
\begin{align}\label{Psi-2}
\begin{split}
G(\bu)=&dJ(\bu)[\bu]=P_1(u)+P_2(v)-\frac 32\b \int_{\R^N} u^2v\,dx\\
=&\|\bu\|^2-\int_{\R^N} u^4dx-\frac 12\int_{\R^N} |v|^3dx-\frac 32\b \int_{\R^N} u^2v\,dx.
\end{split}
\end{align}

\section{Natural constraints and key results}\label{sec:3}\

We can easily see that the functional $J$ is not bounded below
on $\mathbb{E}$. Thus, we are going to work on the so called Nehari
manifold, which we will prove that it is a natural constraint for the functional $J$,
and even more the functional constrained to the Nehari manifold is
bounded below. Let us define the following radial Nehari manifolds 
\begin{equation}
\cN =\{ \bu\in \h\setminus\{\bo\}: G(\bu)=0\}, 
\end{equation}
and the full Nehari manifold 
\begin{equation}\label{full nehari}
\cM =\{ \bu\in \E\setminus\{\bo\}: G(\bu)=0\}.
\end{equation}

\begin{remark}
All the properties we are going to prove in this section are satisfied for both $\cM$ and $\cN$, but the Palais-Smale  condition in Lemma \ref{Lemma PS},
 is only satisfied for $J$ on $\cN$, see Theorem \ref{radial compact embe}.
 To be short, we are going to demonstrate the following properties only for $\cN$.
\end{remark}
\begin{Proposition}
The Nehari manifold $\cN$ is a natural constraint for the functional $J$.
\end{Proposition}
\begin{pfn}
 For all $\bu,\bh\in \E$, we have 
\begin{align}\label{eq:DG-2}
\begin{split}
dG(\bu)[\bh]=& 2\bra\bu,\bh\ket-4\int_{R^N} u^3h_1dx-\frac 32\int_{R^N} |v|vh_2dx\\&-3\b \int_{R^N} uvh_1\,dx-\frac 32\b \int_{R^N} u^2h_2\,dx.
\end{split}
\end{align}
In particular, if $\bh=\bu\in\cN$, we can combine the above expression with the fact $G(\bu)=0$ and we obtain 
\begin{equation}\label{eq:delta psi restringida a N}
dG(\bu)[\bu]=dG(\bu)[\bu]-3G(\bu)= - \|\bu \|^2-\int_{R^N} u^4\,dx<0,\quad\forall\, \bu\in \cN.
\end{equation}
Then, $\cN$ is a locally smooth manifold near any point $\bu\neq  0$ with $G(\bu)=0$.  Now we are going to prove that $\bo$ is away from the Nehari manifold using the second derivative of the functional $J$. We can see that 
$$dJ(\bo)[\bh]=0.$$
At this point we would like to indicate that $\bo$ is a critical point of $J$.
Moreover,
\begin{align}
\displaystyle d^2I_1(u)[h_1][k_1]&=\int_{R^N} ({\Delta h_1\cdot \Delta k_1}+\l_j h_1k_1)\,dx,-3\int_{R^N} u^2h_1k_1\,dx\label{ddi1-2}\\
\displaystyle  
d^2I_2(v)[h_2][k_2]&=\displaystyle  \int_{R^N} ({\Delta h_2\cdot \Delta k_2}+\l_j h_2k_2)\,dx-\int_{R^N} |v|h_2k_2\,dx,\label{ddi2-2}
\end{align}
and
\begin{align}\label{DDphi-2}
\begin{split}
d^2J(\bu)[\bh][\bk]=&d^2I_1(u)[h_1][k_1]+d^2I_2(v)[h_2][k_2]\\
&-\b \int_{R^N} vh_1k_1\,dx-\b \int_{R^N} uh_2k_1\,dx-\b \int_{R^N} uh_1k_2\,dx.
\end{split}
\end{align}
Then, 
$$d^2J(\bo)[\bh]^2=\|\bh\|^2,$$
is positive definite, so that we infer that $\bo$ is a strict minimum for $J$. Consequently, $\bo$ is an isolated point of the set of all critical points of $J$, thus $\bo$ is away from $\cN$.

Therefore, we have that $\cN$ is a smooth complete manifold of codimension one, and there exists a constant $\rho>0$ such that
\be\label{eq:away from zero}
\|\bu\|^2>\rho,\qquad\forall\bu\in\cN.
\ee
Furthermore, \eqref{eq:delta psi restringida a N} and \eqref{eq:away from zero} imply that $\cN$ is a Natural constraint of $J$ by the Theorem \ref{TH:natural constrain}, i.e., $\bu\in \h\setminus\{\bo\}$ is a critical point of $J$ if and only if $\bu$ is a critical point of $J$ constrained on $\cN$.
\end{pfn}

\begin{remarks}\label{re:functionl constrain y PS condition}\ 

\begin{enumerate}
\item[(i)] The functional constrained on $\cN$ takes the form
\begin{equation}\label{eq:Functional constrain N2}
J|_{\cN}(\bu)= \frac 16\|\bu\|^2+\frac{1}{12}\int_{\R^N} u^4\,dx.
\end{equation}
Even more, using \eqref{eq:away from zero} and \eqref{eq:Functional constrain N2},
\begin{equation}\label{eq:Functional constrain N}
J(\bu)>\frac{1}{6}\rho,\qquad \forall\,\bu\in\cN.
\end{equation}
Therefore, $J$ is bounded from below on $\cN$, so we can try to
minimize it on the Nehari manifold.
\item[(ii)] Theorem \ref{cont emb} give us in particular the continuous embedding
\[
E\hookrightarrow L^q(\R^N),\qquad \text{with } 2\le q \leq 4<2^*,
\] 
for $N\leq 7$, since the critical exponent in this case is given by
\[
2^*=\left\lbrace
\begin{array}{ccl}
\frac{2N}{N-4} & \text{if} & N\ge 5 \\ 
\infty & \text{if} & N=1,2,3,4.
\end{array} 
 \right. 
\]
Therefore, $J$ is well defined for $N\leq 7$.
\item[(iii)]
Concerning the Palais-Smale condition for $N\geq 2$ it follows due to the compact embedding. By Theorem \ref{radial compact embe} we obtain:
\[
H\hookrightarrow\hookrightarrow L^q(\R^N),\qquad \text{with }2<q< 2^*.
\] 
In the one-dimensional case we have no compact embedding, but we can avoid this problem proceeding in the same way as in the previous chapter.
\end{enumerate}
\end{remarks}

System  \eqref{eq:NLS-KdV2} only admits one kind of semi-trivial
solutions of the form $(0, v)$. Indeed, if we suppose  $v=0$, the
second equation in \eqref{eq:NLS-KdV2} gives us that $u=0$ as well. Thus, let
us take $\bv_2=(0, V_2)$, where $V_2$ is a 
radially symmetric  ground state solution of the equation
$\Delta^2v+\lambda_2v = \frac{1}{2}|v|v$. In particular, we can
assume that $V_2$ is positive because in other case, taking $|V_2|$,
it has the same energy. Moreover, if we denote by $V$ a positive radially
symmetric ground state solution of the equation $\Delta^2v+v =
\frac{1}{2}|v|v$, then, after some rescaling $V_2$ can be defined by 
\be\label{elemento reescalado} V_2(x)=\lambda_2V(\sqrt[4]{\lambda_2}x). 
\ee
 As a
consequence, $\bv_2=(0,V_2)$ is a non-negative semi-trivial solution
of \eqref{eq:NLS-KdV2}, independently of the value of $\beta$.

Furthermore, since the only semi-trivial solution comes from the second equation, as seen just above, we define the Nehari manifold corresponding to the second equation
\[
\cN_2=\left\lbrace 
v\in H\setminus\{0\}:P_2(v)=0
\right\rbrace. \]
Moreover, if we consider the tangent spaces 
\[
T_{\bv_2}\cN=\left\lbrace 
\bh\in \h:dG(\bv_2)[\bh]=0
\right\rbrace\quad\text{and}\quad T_{V_2}\cN_2=\left\lbrace 
h\in H:dP_2(V_2)[h]=0
\right\rbrace,
\]
we have
\begin{equation}\label{eq:tang2}
\bh=(h_1,h_2)\in T_{\bv_2} \cN  \Longleftrightarrow h_2\in T_{V_2} \cN_2.
\end{equation}
The proof of the above equivalence is identical to the proof of Lemma \ref{lemma tangent}. 

In the following result we establish the
character of $\bv_2$ in terms of the size of the coupling parameter.

\begin{Proposition}\label{Prop:fund}
There exists $\L'>0$ such that:
\begin{itemize}
\item[(i)] if $\beta<\L'$, then $\bv_2$ is a strict local minimum of $ J$ constrained on $\cN$,
\item[(ii)] if $\beta>\L'$, then $\bv_2$ is a saddle point of $ J$ constrained on $\cN$. Moreover,
\be\label{eq:infimo B>L}
\inf\limits_\cN J<J(\bv_2).
\ee
\end{itemize}
\end{Proposition}

\begin{pfn}
\begin{itemize}
\item[(i)] We define
\be\label{Lambda}
\L'=\inf\limits_{\varphi\in H\setminus\{0\}}\frac{\|\varphi\|^2_1}{\int_{\R^N}V_2\varphi^2}.
\ee
For $\bh\in  T_{\bv_2}\cN$ one has that
\be\label{eq:seg deriv}
d^2_\cN J(\bv_2)[\bh]^2=\|h_1\|^2_1+d^2_{\cN_2}I_2(V_2)[h_2]^2-\beta\int_{\R^N}V_2h_1^2.
\ee
Since $\bh=(h_1,h_2)\in T_{\bv_2}\cN$, then, $h_2\in T_{V_2} \cN_2$ thanks to \eqref{eq:tang2}. Thus, since $V_2$ is a minimum of $I_2$ on $\cN_2$, there exists a constant $c>0$ so that
\be\label{eq:minimo-pos}
 d^2_{\cN_2}I_2 (V_2)[h_2]^2\ge c\|h_2\|_2^2.
\ee
From \eqref{Lambda} we obtain that
$$\int_{\R^N}V_2h_1^2\leq \|h_1\|^2_1/\L',\qquad\forall\,h_1\in H.$$
Thus, substituting both previous inequalities in \eqref{eq:seg deriv} we arrive at
\be\label{eq:desig seg der}
d^2_\cN J(\bv_2)[\bh]^2\geq\left(1-\frac{\beta}{\L'} \right) \|h_1\|_1^2+c\|h_2\|_2^2.
\ee
Moreover, since $\beta<\L'$ we have that $ d^2_\cN J(\bv_2)[\bh]^2$ is positive definite. Therefore, $\bv_2$ is a strict local minimum of $J$ on $\cN$.

\item[(ii)]
Since $\beta >\L'$, there exists $\wt{h}\in H$ such that
$$
\L'< \frac{\|\wt{h}\|_1^2}{\int_{\R^N} V_2\wt{h}^2dx}<\b,
$$  
and, using the equivalence \eqref{eq:tang2}, we obtain $\bh_1=(\wt{h},0)\in T_{\bv_2}\cN$ and
$$
d^2_\cN J(\bv_2)[\bh_1]^2 =\|\wt{h}\|_1^2 -\b\int_{\R^N} V_2 \wt{h}^2dx<0.$$
On the other hand, taking $h_2\in T_{V_2}\cN_2$ not equal to zero, then $\bh_2=(0,h_2)\in T_{\bv_2}\cN$ and
$$d^2_\cN J(\bv_2)[\bh_2]^2=d^2_{\cN_2}I_2 (V_2)[h_2]^2\ge c_2\|h_2\|_2^2>0.
$$
Consequently, this is sufficient to conclude that $\bv_2$ is a saddle point of $J$ on $\cN$ and obviously inequality \eqref{eq:infimo B>L} holds.
\end{itemize}
\end{pfn}

To conclude this section we also proof that the functional $J$ satisfies the Palais-Smale condition in $\cN$ on the appropriate dimensions.

\begin{Lemma}\label{Lemma PS}
Assume that $2\le N\leq 7$, then $J$ satisfies the Palais-Smale  condition on $\cN$.
\end{Lemma}
\begin{pfn}
Let $\bu_n=(u_n,v_n)\in \cN$ be a Palais-Smale sequence such that 
$$
J(\bu_n)\to c>0 \quad \text{and}\quad \nabla_\cN J(\bu_n)\to 0,\quad \text{as}\quad n\to\infty.
$$
 From \eqref{eq:Functional constrain N2} it follows that $\bu_n$ is bounded and, due to the reflexivity of $W^{2,2}(\R^N)$, we have a convergent subsequence $\bu_n\rightharpoonup \bu_0$  (relabelling). Since $H$ is compactly embedded into $L^q(\R^N)$ with $2<q<4+\frac{2}{3}$ while $2\le N\leq 7$ (see Remark \ref{re:functionl constrain y PS condition}-(iii)), we infer that
\[
\int_{\R^N}u_n^4\to\int_{\R^N}u_0^4,
\qquad \int_{\R^N}|v_n|^3\to\int_{\R^N}|v_0|^3,
\qquad \int_{\R^N}u_n^2v_n\to\int_{\R^N}u_0^2v_0.
\]
Moreover, using the fact that $\bu_n\in\cN$ and \eqref{eq:away from zero}, we have that
\begin{align*}
\|\bu_n\|^2=\int_{\R^N}& u_n^4\,dx +\frac{1}{2} \int_{\R^N} |v_n|^3 \,dx+\frac{3}{2}\beta\int_{\R^N} u_n^2v_n\,dx\to\\&\int_{\R^N} u_0^4\,dx +\frac{1}{2} \int_{\R^N} |v_0|^3\,dx +\frac{3}{2}\beta\int_{\R^N} u^2_0v_0\,dx\geq\rho,
\end{align*}
which implies that $\bu_0\neq 0$. We know that
\begin{equation}\label{four2}
\nabla_\cN J(\bu_n)=J'(\bu_n)-\l_n G'(\bu_n)\to 0,
\end{equation}
then, evaluating the above functional in the direction $\bu_n$, we have
\[\big|\bra\nabla_\cN J(\bu_n),\bu_n\ket\big|\leq \|\nabla_\cN J(\bu_n)\|\|\bu_n\|\to 0\qquad\text{as}\quad n\to\infty.\] On the other hand $\bra J'(\bu_n),\bu_n\ket=G(\bu_n)=0$ since $\bu_n\in \cN$ and, using \eqref{eq:delta psi restringida a N} jointly with \eqref{eq:away from zero}, we obtain
\[
\rho \big|\l_n\big|\leq\big|\l_n\bra G'(\bu_n),\bu_n\ket\big|\to 0,
\]
then, $\l_n\to 0$ as $n\to\infty$.

Now, we will show that $\|G'(\bu_n)\|$ is bounded in a similar way as we saw in the previous chapter. We can write the norm of a linear functional as
\begin{equation}\label{eq:normsup}
\|G'(\bu_n)\|=\|dG(\bu_n)\|=\sup\limits_{\substack{\bh\in \E\\ \|\bh\|=1}}\big|dG(\bu_n)[\bh]\big|.
\end{equation}
Using the triangular inequality in \eqref{eq:DG-2} we obtain
\begin{align*}
\big|dG(\bu_n)[\bh]\big|\leq & 2|\bra\bu_n,\bh\ket|+4\|u^3_nh_1\|_{L^1}+\frac 32\|v_n^2h_2\|_{L^1}+3\b \| u_nv_nh_1\|_{L^1}+\frac 32\b \|u_n^2h_2\|_{L^1},
\end{align*}
thus, applying the H\"older inequality
\begin{align*}
\big|dG(\bu_n)[\bh]\big|\leq & 2\|\bu_n\|\|\bh\|+4\|u_n^3\|_{L^{4/3}}\|h_1\|_{L^4}+\frac 32\|v_n^2\|_{L^2}\|h_2\|_{L^2}\\&+3\b \|u_n\|_{L^3}\|v_n\|_{L^3}\|h_1\|_{L^3}+\frac 32\b\|u_n^2\|_{L^2}\|h_2\|_{L^2}\\
\leq & 2\|\bu_n\|\|\bh\|+4\|u_n\|^{3}_{L^{4}}\|h_1\|_{L^4}+\frac 32\|v_n\|^2_{L^4}\|h_2\|_{L^2}\\&+3\b \|u_n\|_{L^3}\|v_n\|_{L^3}\|h_1\|_{L^3}+\frac 32\b\|u_n\|^2_{L^4}\|h_2\|_{L^2}.
\end{align*}
Note that all the above norms are well defined due to the continuous embedding mentioned in Remark \ref{re:functionl constrain y PS condition}-(ii). Moreover, by the same reason, there exist constant $C_1,C_2,C_3,C_4$ such that 
\begin{align*}
\big|dG(\bu_n)[\bh]\big|\leq & 2\|\bu_n\|\|\bh\|+C_1\|u_n\|^3_1\|h_1\|_1+C_2\|v_n\|_2^2\|h_2\|_2\\&+C_3 \|u_n\|_1\|v_n\|_2\|h_1\|_1+C_4\|u_n\|_1^2\|h_2\|_2,
\end{align*}
where, knowing that 
$$
\|\bh\|^2=\|h_1\|_1^2+\|h_2\|_2^2=1\quad\text{and}\quad \|\bu_n\|^2=\|u_n\|_1^2+\|v_n\|_2^2,
$$
we arrive at 
\begin{align*}
\big|dG(\bu_n)[\bh]\big|\leq  2\|\bu_n\|+C_1\|\bu_n\|^3+\left(C_2+C_3 +C_4\right)\|\bu_n\|^2\quad\text{with}\quad\|\bh\|=1.
\end{align*}
The right part in the above inequality is polynomially dependent of $\|\bu_n\|$ and it is clearly bounded since $\bu_n$ is bounded. From \eqref{eq:normsup} we obtain that there exists a constant $C>0$ such that $\|G'(\bu_n)\|\leq C<+\infty$, for all $n\in\N$. Then,  taking into account that $\nabla_\cN J(\bu_n)$ and $G'(\bu_n)$ are orthogonal and the fact $\l_n\to 0$, we deduce from \eqref{four2} that
$$
\|J'(\bu_n)\|=\|\nabla_\cN J(\bu_n)\|+|\l_n| \|G'(\bu_n)\|\to 0,\qquad\text{as}\quad n\to\infty.
$$
To finish the proof, since $J'(\bu_n)[\bu_0]\to 0$ as $n\to \infty$, one can conclude that $\bu_n\to \bu_0$ strongly in $\h$. Moreover, $\bu_0$ is a critical point of $J$, hence, $\bu_0\in\cN$ since $\cN$ is a natural constraint.
\end{pfn}

\section{Existence results}\label{sec:4}\

This section is divided into two subsections depending on the dimension of the problem \eqref{eq:NLS-KdV2}.

\subsection{High-dimensional case, $2\le N\le 7$}\label{sec:4.1}
In this subsection we will see that the infimum of $J$ constrained on the radial Nehari manifold $\cN$, is attained under appropriate
parameter conditions.
We also prove the existence of a mountain pass critical point.

\begin{Theorem}\label{th:minimo1} Suppose $\beta>\L'$ and $2\leq N\leq 7$. The infimum of $J$ on $\cN$ is attained at some point $\wt{\bu}\geq \bo$
with $J(\wt{\bu})<J(\bv_2)$ and both components $\wt{u},\wt{v}\not \equiv 0$.
\end{Theorem}

\begin{pfn} By the Ekeland's variational principle there exists a minimizing Palais-Smale sequence $\bu_n$ in $\cN$, i.e.,
\[
J(\bu_n)\to m=\inf\limits_\cN J\quad\text{and}\quad
\nabla_\cN J(\bu_n)\to 0.
\]
Due to the Lemma \ref{Lemma PS}, there exists $\wt{\bu}\in\cN$ such that 
$$
\bu_n\to \wt{\bu}\quad \text{strongly as}\quad n\to\infty,
$$
 hence, $\wt{\bu}$ is a minimum point of $J$ on $\cN$. Moreover, taking into account the Proposition \ref{Prop:fund}-(ii), we have
\[
J(\wt{\bu})=m<J(\bv_2).
\]
Note that the second component $\wt{v}$ can not be  zero, because if that occurs then $\wt{\bu}\equiv 0$ due to the form of the second equation of \eqref{eq:NLS-KdV2}, and zero is not in $\cN$. On the other hand, if we suppose that the first component $\wt{u}\equiv 0$, then 
$$
I_2(\wt{v})=J(\wt{\bu})<J(\bv_2)=I_2(V_2),
$$
and this is a contradiction with the fact that $V_2$ is a ground state of the equation $\Delta^2v+\l_1v = \frac{1}{2}|v|v$. 

In general we can not ensure that both components of $\wt{\bu}$ are non-negative, thus, in order to obtain this fact we take $t|\wt{\bu}|\in \cN$, and we will show that
 $$
J(t|\wt{\bu}|)\leq J(\wt{\bu}). 
 $$
 Note that by \eqref{eq:Functional constrain N2} we have that
\begin{equation}\label{eq:comparar t}
J(t|\wt{\bu}|)= \frac{t^2}{ 6}\|\wt{\bu}\|^2+\frac{t^4}{12}\int_{\R^N} \wt{u}^4\,dx,
\qquad J(\wt{\bu})= \frac 16\|\wt{\bu}\|^2+\frac{1}{12}\int_{\R^N} \wt{u}^4\,dx.
\end{equation}
Hence, to prove $J(t|\wt{\bu}|)\leq J(\wt{\bu}) $ is equivalent to show that $t\leq 1$. Taking into account that $G(t|\wt{\bu}|)=0$, we find
\begin{align*}
0=G(t|\wt{\bu}|)&=t^2\|\wt{\bu}\|^2-t^4\int_{\R^N} \wt{u}^4\,dx-t^3\frac{1}{2} \int_{\R^N} |\wt{v}|^3 \,dx-t^3\frac{3}{2}\beta\int_{\R^N} \wt{u}^2|\wt{v}|\,dx,
\end{align*}
which is equivalent to,
\begin{equation}\label{eq:condicion con t}
0=\|\wt{\bu}\|^2-t^2\int_{\R^N} \wt{u}^4\,dx -t\frac{1}{2} \int_{\R^N} |\wt{v}|^3 \,dx-t\frac{3}{2}\beta\int_{\R^N} \wt{u}^2|\wt{v}|\,dx.
\end{equation}
Furthermore, since $\wt{\bu}\in \cN$ we also have,
\begin{equation}\label{eq:condicion sin t}
0=G(\wt{\bu})=\|\wt{\bu}\|^2-\int_{\R^N} \wt{u}^4\,dx -\frac{1}{2} \int_{\R^N} |\wt{v}|^3 \,dx-\frac{3}{2}\beta\int_{\R^N} \wt{u}^2\wt{v} \,dx.
\end{equation}
Now, if we suppose that $t>1$ it follows that
\begin{align*}
t^2\int_{\R^N} \wt{u}^4 \,dx+t\frac{1}{2} \int_{\R^N} |\wt{v}|^3 \,dx&+t\frac{3}{2}\beta\int_{\R^N} \wt{u}^2|\wt{v}|\,dx>\\&\int_{\R^N} \wt{u}^4 \,dx+\frac{1}{2} \int_{\R^N} |\wt{v}|^3\,dx +\frac{3}{2}\beta\int_{\R^N} \wt{u}^2|\wt{v}|\,dx.
\end{align*}
Then, thanks to \eqref{eq:condicion con t} we obtain
\begin{align}\label{desig}
0<\|\wt{\bu}\|^2-\int_{\R^N} \wt{u}^4 \,dx-\frac{1}{2} \int_{\R^N} |\wt{v}|^3\,dx -\frac{3}{2}\beta\int_{\R^N} \wt{u}^2|\wt{v}|\,dx.
\end{align}
Combining \eqref{eq:condicion sin t} with \eqref{desig} we arrive at
\[0<\frac{3}{2}\beta\int_{\R^N} \wt{u}^2\left(\wt{v}-|\wt{v}|\right)\,dx,\]
which is a contradiction. Consequently, $t\leq 1$ and therefore $J(t|\wt{\bu}|)\leq J(\wt{\bu})$. On the other hand we know that $J$ attains its minimum at $\wt{\bu}$ on $\cN$, and therefore the last inequality can not be strict. Moreover, due to  \eqref{eq:comparar t} it can not happen that $t<1$, hence, $ t=1$ and
$$J(|\wt{\bu}|)=J(\wt{\bu}).$$
Redefining $\wt{\bu}$ as $|\wt{\bu}|$ we finally have that the minimum on the Nehari manifold is attained at $\wt{\bu}\geq 0$ with non-trivial components. 
\end{pfn}

\begin{Theorem}\label{Th:lambda>L_2}
Assume  $2\leq N\leq 7$, $\beta>0$. There exists a positive constant
$\L_2'$ such that, if $\l_2>\L_2'$, the functional $J$ attains its
infimum on $\cN$ at some $\widehat{\bu}\geq\bo$ with
$J(\widehat{\bu})<J(\bv_2)$ and both
$\widehat{u},\widehat{v}\not \equiv 0$.
\end{Theorem}

\begin{pfn}
Using the same argument as above in the previous  theorem, we can prove that the minimum is attained at some point $\widehat{\bu}\in \cN$, but to show that $\widehat{u},\widehat{v}\not \equiv 0$ we need to ensure that $J(\widehat{\bu})<J(\bv_2)$. In Theorem \ref{th:minimo1} this fact was proved for the case $\beta>\L'$ and here we need to prove it for $0<\beta\leq\L'$. 
In this case the point $\bv_2$ is a strict local minima and this does not guarantee that $\widehat{\bu}\not \equiv\bv_2$. 

Then, to see $J(\widehat{\bu})<J(\bv_2)$ we will use a similar procedure to the one applied in \cite[Theorem 4.3]{c3} showing 
that there exists an element of the form 
$$ 
\bw=t(V_2,V_2)\in\cN\quad \hbox{such that}\quad J(\bw)<J(\bv_2),
$$
for $\lambda_2$ big enough. Notice that, thanks to the equation $G(\bw)=0$ we have that $t>0$ satisfies the following condition
\be\label{eq:condicion1}
t^2\|(V_2,V_2)\|^2-t^4\intN V_2^4\,dx -\frac{1}{2}t^3(1+3\beta)\intN V_2^3\,dx=0,
\ee
and by definition we also have 
\be\label{eq:norma doble}
\|(V_2,V_2)\|^2=2\|V_2\|_2^2+(\lambda_1-\lambda_2)\intN V_2^2\,dx.
\ee
Moreover, we know that $V_2\in\cN_2$, hence satisfies the equation $J_2(V_2)=0$, i.e.,
\be\label{eq:norma simple}
\|V_2\|_2^2-\frac{1}{2}\intN V_2^3\,dx=0.
\ee
Observe that the absolute value does not appear because $V_2$ is positive. Substituting \eqref{eq:norma doble}  and \eqref{eq:norma simple} in \eqref{eq:condicion1} it follows
\be\label{eq:condicion2}
t^2\left(\intN V_2^3\,dx+(\lambda_1-\lambda_2) \intN V_2^2\right)\,dx -t^4\intN V_2^4 \,dx-\frac{1}{2}t^3(1+3\beta)\intN V_2^3\,dx=0.
\ee
Hence, applying the rescaling \eqref{elemento reescalado}  yields
\be\label{cambio p}
\intN V_2^p\,dx=\lambda_2^{p-\frac{N}{4}}\intN V^p\,dx.
\ee
Subsequently, substituting it into \eqref{eq:condicion2}, for $p=2,3,4$, and dividing by $t^2\lambda_2^{3-\frac{N}{4}}$ we have that
\be\label{eq:condicion}
\intN V^3\,dx+\dfrac{\lambda_1-\lambda_2}{\lambda_2}\intN V^2\,dx-t^2\lambda_2\intN V^4\,dx-\frac{1}{2}t(1+3\beta)\intN V^3\,dx=0.
\ee 
Moreover, due to \eqref{eq:Functional constrain N}, \eqref{eq:norma doble} and \eqref{eq:norma simple} we find respectively the expressions
\be\label{eq: forma de Phi(w)}
J(\bw)=\frac{1}{6}t^2\left( \intN V_2^3\,dx+(\lambda_1-\lambda_2) \intN V_2^2\right)\,dx +\frac{1}{12}t^4\intN V_2^4\,dx,
\ee
\be\label{eq: forma de Phi(bv_2)}
J(\bv_2)=I_2(V_2)=\frac{1}{2}\|V_2\|_2^2\,dx-\frac{1}{6}\intN V_2^3\,dx=\frac{1}{12}\intN V_2^3\,dx.
\ee
Furthermore, we are looking for the inequality $J(\bw)< J(\bv_2)$, or equivalently,
\be
\frac{1}{6}t^2\left( \intN V_2^3+(\lambda_1-\lambda_2) \intN V_2^2\right) +\frac{1}{12}t^4\intN V_2^4- \frac{1}{12}\intN V_2^3
<0,
\ee
and, then, applying again \eqref{cambio p} and multiplying by $6\lambda^{\frac{N}{4}-3}$, we actually have 
\be\label{desigualdad de w}
t^2\left( \intN V^3\,dx+\dfrac{\lambda_1-\lambda_2}{\lambda_2}\intN V^2\,dx\right) +\frac{1}{2}t^4\lambda_2\intN V^4\,dx-\frac{1}{2}\intN V^3\,dx<0.
\ee
On the other hand, fixing $\beta$ in condition \eqref{eq:condicion} when $\lambda_2$ is sufficiently large
then, $t$ will have to be sufficiently small in order for the element $\bw$ to stay on $\cN$, i.e., $t\to 0^+$ when $\lambda_2\to \infty$. 
Thus, applying it to \eqref{desigualdad de w} we arrive at the convergence  
\[
t^2\left( \intN V^3\,dx+\dfrac{\lambda_1-\lambda_2}{\lambda_2}\intN V^2\,dx\right)\to 0,\quad \hbox{as}\quad \lambda_2\to\infty.
\]
Moreover, since $t^2\lambda_2$ is bounded we also have
\[
\frac{1}{2}t^4\lambda_2\intN V^4\,dx\to0.
\]
Therefore, there exists a positive constant $\L'_2$ such that for $\l_2>\L'_2$ inequality \eqref{desigualdad de w} holds and, hence, 
$$
J(\widehat{\bu})\leq J(\bw)< J(\bv_2).
$$ 
Finally, to show that $\widehat{\bu}\geq\bo$ and $\widehat{u},\widehat{v}\not \equiv 0$ we can use the same argument as in  Theorem \ref{th:minimo1}.
\end{pfn}

In the following theorem we will prove the existence of a Mountain Pass critical point of $J$ on $\cN$.

\begin{Theorem}\label{Montain Pass}
Assume  $2\leq N\leq 7$ and $\beta<\L'$. There exists a constant
$\L'_2$ such that, if $\l_2>\L'_2$, then $J$ constrained on $\cN$
has a Mountain Pass critical point $\bu^*$ with
$J(\bu^*)>J(\bv_2)$.
\end{Theorem}
\begin{pfn}
Due to Proposition \ref{Prop:fund}-(i), $\bv_2$ is a strict local minima of $J$ on $\cN$, and taking into account Theorem \ref{Th:lambda>L_2} we 
obtain $\L'_2$ such that, for $\lambda_2>\L'$, we have $J(\widehat{\bu})<J(\bv_2)$. Under those conditions 
we are able to apply the Mountain Pass Theorem to $J$ on $\cN$, that provides us with a Palais-Smale sequence $\bv_n\in\cN$ such that
\[
J(\bv_n)\to c=\inf\limits_{\gamma\in\Gamma}\max\limits_{0\le t\le 1}J(\gamma(t)),
\]
where 
\[
\Gamma=\left\{\gamma\in \mathcal{C}([0,1],\cN)\  |\ \gamma(0)=\bv_2,\ \gamma(1)=\widehat{\bu} \right\}.
\]
Furthermore, applying the Lemma \ref{Lemma PS}, 
we are able to find a subsequence of $\bv_n$ such that (relabelling) $\bv_n\to\bu^*$ strongly in $\h$. Thus, $\bu^*$ is a critical point of $J$, and also by the Mountain Pass Theorem we have that
\[
J(\bu^*)>J(\bv_2),
\]
which concludes the proof.
\end{pfn}

\subsection{One-dimensional case, $N=1$}\label{sec:4.2}
Here we must  point out that we do not have the compact embedding even for $\h$. However, we  will show that for a Palais-Smale sequence we are able to find a subsequence
for which its weak limit is a solution of \eqref{eq:NLS-KdV2} belonging to $\E$.
 In order to avoid the lack of compactness for $N=1$ we will use the same idea used in \cite{c3} working on the full Nehari manifold $\cM$ defined in \eqref{full nehari}.

The next theorem is the analogous of Theorem \ref{th:minimo1} in  dimension one.

\begin{Theorem}\label{th:minimo N=1}
Suppose $N=1$ and $\beta>\L'$. The infimum of $J$ on $\mathcal{M}$
is attained at some $\wt{\bu}\geq \bo$ with both components
$\wt{u},\wt{v}\not \equiv 0$. Moreover,
$J(\wt{\bu})<J(\bv_2)$.
\end{Theorem}
\begin{pfn}    
 Again, by the Ekeland's variational principle there exists a minimizing Palais-Smale sequence $\bu_n$ in $\cM$, i.e.,
\[
J(\bu_n)\to m=\inf\limits_{\cM}J\quad\text{and}\quad
\nabla_{\cM}J(\bu_n)\to 0,
\]
thus, $\bu_n$ is bounded due to \eqref{eq:Functional constrain N}. We can assume (relabelling) that $\bu_n\rightharpoonup \bu$ weakly in $\E$, $\bu_n\to \bu$ strongly in $\mathbb{L}^q_{loc}(\R)=L^q_{loc}(\R)\times  L^q_{loc}(\R)$ for every $1\le q<\infty$ and $\bu_k\to \bu$ a.e. in $\R$. Moreover, arguing as the same way that in Lemma \ref{Lemma PS} we can obtain $J'(\bu_n)\to 0$ as $n\to\infty$.

Using the same idea that in \cite{c3} we will prove that there is no evanescence for $\mu_n(x)=u_n^2(x)+v_n^2(x)$, where $\bu_n=(u_n,v_n)$, i.e, exist $R, C>0$ so that 
\be\label{eq:vanishing}
\sup_{z\in\R}\int_{|z-x|<R}\mu_n (x)dx\ge C>0,\quad\forall n\in\mathbb{N}.
\ee
On the contrary, if we suppose
$$
\sup_{z\in\R}\int_{|z-x|<R}\mu_k(x)dx\to 0,
$$
Thanks to Lemma \ref{lem:measure}, applied in a similar way as in \cite{c-fract},  we find that $\bu_k\to \bo$ strongly in
$\mathbb{L}^{q}(\R)$ for any $2<q<\infty$, This is a contradiction since $\bu_n \in\cN$, and by \eqref{eq:Functional constrain N} jointly with the fact $J(\bu_n)\to c$ we have
$$
0< \frac 17\rho <m+o_n(1)=J(\bu_n)=F(\bu_n),\quad\mbox{with } o_n(1)\to 0 \quad\mbox{as }n\to\infty,
$$
hence \eqref{eq:vanishing} is true and there is no evanescence.

We observe that   we can find a sequence of points
$\{z_n\}\subset\R$ so that by \eqref{eq:vanishing}, the translated sequence $\overline{\mu}_n(x)= \mu_n(x+z_n)$ satisfies
$$
\liminf_{n\to\infty}\int_{B_R(0)}\overline{\mu}_n\ge C >0.
$$
Taking into account that $\overline{\mu}_n\to \overline{\mu}$
strongly in $L_{loc}^1(\R)$, we obtain that
$\overline{\mu}\not\equiv 0$, thus, the weak limit of 
$\overline{\bu}_n(x)=\bu_n(x+z_n)$ which we denote by $\overline{\bu}$ is non-trivial. Notice that $\overline{\bu}_n,\overline{\bu}\in\cM$ and $\overline{\bu}_n$ is a Palais-Smale sequence of  $J$ on $\cM$ due to the invariance of $J$ under translations. Moreover, if we set $\overline{F}=J|_{\cM}$, we obtain the following from the lower semi-continuity of $\overline{F}$,
$$
J (\overline{\bu})  =  \dyle \overline{F}(\overline{\bu})
 \le  \dyle\liminf_{n\to\infty} \overline{F}(\overline{\bu}_n)
 =  \dyle\liminf_{n\to\infty}J(\overline{\bu}_n)=
   \dyle\liminf_{n\to\infty}J(\bu_n)= m.
$$
Therefore, $\overline{\bu}$ is a non-trivial critical point of $J$ constrained on $\cM$. Furthermore, it is not a semi-trivial solutions since we know that $J(\overline{\bu})<J(\bv_2)$ from Proposition \ref{Prop:fund}-$(ii)$. Finally, to show that $\overline{\bu}\geq\bo$ we apply the same argument used in Theorem \ref{th:minimo1}.
\end{pfn}

Theorem \ref{Th:lambda>L_2} can be extended to the one-dimensional
case directly using the same idea that was utilized  in the last
proof, obtaining the following.
\begin{Corollary}\label{Cor:lambda>L_2}
Assume  $N=1$, $\beta>0$. There exists a positive constant $\L'_2$
such that, if $\l_2>\L'_2$, the functional $J$ attains its infimum
on $\cN$ at some $\widehat{\bu}\geq\bo$ with
$J(\widehat{\bu})<J(\bv_2)$ and both
$\widehat{u},\widehat{v}\not \equiv 0$.
\end{Corollary}
To finish, for $N=1$, Theorem \ref{Montain Pass}  can be obtained in
a similar manner, obtaining the following.
\begin{Corollary}\label{Cor:Montain Pass}
Assume  $N=1$ and $\beta<\L'$. There exists a constant $\L'_2$ such
that, if $\l_2>\L'_2$, then $J$  constrained on $\cN$ has a
Mountain Pass critical point $\bu^*$ with $J(\bu^*)>J(\bv_2)$.
\end{Corollary}

\chapter*{Conclusions}
\addcontentsline{toc}{chapter}{Conclusions}
\markboth{CONCLUSIONS}{CONCLUSIONS}

In this work, we have studied the existence of bound and ground states for  a coupled stationary system of nonlinear Schr\"odinger--Korteweg-de Vries equations. This fact has been proved for dimensions $1\leq N\leq 3$  where, for $N=1$,  we have used a  measure lemma in order to circumvent the lack of compactness and, for $N=2,3$, the compact embedding of the radial Sobolev space. 

Also, we have considered for the first time a higher order system of nonlinear Schr\"odinger--Korteweg-de Vries equations as a natural extension of the above mentioned system. From this, we have obtained a bi-harmonic stationary system looking for ``standing-traveling'' waves solutions and, using similar variational techniques, we have  proved the existence and multiplicity of solutions under appropriate conditions on the parameters and with dimension $1\leq N\leq 7$. To the best of our knowledge this is a new result.


\providecommand{\bysame}{\leavevmode\hbox to3em{\hrulefill}\thinspace}
\providecommand{\MR}{\relax\ifhmode\unskip\space\fi MR }
\providecommand{\MRhref}[2]{%
  \href{http://www.ams.org/mathscinet-getitem?mr=#1}{#2}
}
\providecommand{\href}[2]{#2}



\begin{thebibliography}{10}

\bibitem{Adams}
R.~A. Adams and J.~J.~F. Fournier, \emph{Sobolev spaces}, second ed., Pure and
  Applied Mathematics (Amsterdam), vol. 140, Elsevier/Academic Press,
  Amsterdam, 2003. \MR{2424078}

\bibitem{Ack}
N.~N. Akhmediev and A.~Ankiewicz, \emph{Solitons: nonlinear pulses and beams},
  Chapman \& Hall, 1997.

\bibitem{aa}
J.~Albert and J.~Angulo~Pava, \emph{Existence and stability of ground-state
  solutions of a {S}chr\"odinger-{K}d{V} system}, Proc. Roy. Soc. Edinburgh
  Sect. A \textbf{133} (2003), no.~5, 987--1029. \MR{2018323}

\bibitem{ab}
J.~Albert and S.~Bhattarai, \emph{Existence and stability of a two-parameter
  family of solitary waves for an {NLS}-{K}d{V} system}, Adv. Differential
  Equations \textbf{18} (2013), no.~11-12, 1129--1164. \MR{3129020}

\bibitem{acf}
P.~{\'A}lvarez-Caudevilla, E.~Colorado, and R.~Fabelo, \emph{A higher order
  system of some coupled nonlinear {S}chr\"odinger and {K}orteweg-de {V}ries
  equations}, Preprint, arXiv:1607.00482, 2016.

\bibitem{pev}
P.~{\'A}lvarez-Caudevilla, E.~Colorado, and V.~A. Galaktionov, \emph{Existence
  of solutions for a system of coupled nonlinear stationary bi-harmonic
  {S}chr\"odinger equations}, Nonlinear Anal. Real World Appl. \textbf{23}
  (2015), 78--93. \MR{3316625}

\bibitem{a}
A.~Ambrosetti, \emph{A note on nonlinear {S}chr\"odinger systems: existence of
  a-symmetric solutions}, Adv. Nonlinear Stud. \textbf{6} (2006), no.~2,
  149--155. \MR{2172811}

\bibitem{ac1}
A.~Ambrosetti and E.~Colorado, \emph{Bound and ground states of coupled
  nonlinear {S}chr\"odinger equations}, C. R. Math. Acad. Sci. Paris
  \textbf{342} (2006), no.~7, 453--458. \MR{2214594}

\bibitem{ac2}
\bysame, \emph{Standing waves of some coupled nonlinear {S}chr\"odinger
  equations}, J. Lond. Math. Soc. (2) \textbf{75} (2007), no.~1, 67--82.
  \MR{2302730}

\bibitem{acr}
A.~Ambrosetti, E.~Colorado, and D.~Ruiz, \emph{Multi-bump solitons to linearly
  coupled systems of nonlinear {S}chr\"odinger equations}, Calc. Var. Partial
  Differential Equations \textbf{30} (2007), no.~1, 85--112. \MR{2333097}

\bibitem{a-m}
A.~Ambrosetti and A.~Malchiodi, \emph{Perturbation methods and semilinear
  elliptic problems on {${\bf R}^n$}}, Progress in Mathematics, vol. 240,
  Birkh\"auser Verlag, Basel, 2006. \MR{2186962}

\bibitem{AM}
\bysame, \emph{Nonlinear analysis and semilinear elliptic problems}, Cambridge
  Studies in Advanced Mathematics, vol. 104, Cambridge University Press,
  Cambridge, 2007. \MR{2292344}

\bibitem{a-prodi}
A.~Ambrosetti and G.~Prodi, \emph{A primer of nonlinear analysis}, Cambridge
  Studies in Advanced Mathematics, vol.~34, Cambridge University Press,
  Cambridge, 1993. \MR{1225101}

\bibitem{ar}
A.~Ambrosetti and P.~H. Rabinowitz, \emph{Dual variational methods in critical
  point theory and applications}, J. Functional Analysis \textbf{14} (1973),
  349--381. \MR{0370183}

\bibitem{bor}
A.~Bahrouni, H.~Ounaies, and V.~D. R{\u{a}}dulescu, \emph{Infinitely many
  solutions for a class of sublinear {S}chr\"odinger equations with indefinite
  potentials}, Proc. Roy. Soc. Edinburgh Sect. A \textbf{145} (2015), no.~3,
  445--465. \MR{3371562}

\bibitem{Balakrishnan}
R~Balakrishnan, \emph{Soliton propagation in nonuniform media}, Phys. Rev. A
  \textbf{32} (1985), 1144--1149.

\bibitem{Ballentine}
L.~E. Ballentine, \emph{Quantum mechanics: A modern development}, World
  Scientific Publishing Co. Pte. Ltd, River Edge, NJ, 1998.

\bibitem{Bandle}
C.~Bandle, \emph{Isoperimetric inequalities and applications}, Monographs and
  Studies in Mathematics, vol.~7, Pitman (Advanced Publishing Program), Boston,
  Mass.-London, 1980. \MR{572958}

\bibitem{bt}
T.~Bartsch and Z.-Q. Wang, \emph{Note on ground states of nonlinear
  {S}chr\"odinger systems}, J. Partial Differential Equations \textbf{19}
  (2006), no.~3, 200--207. \MR{2252973}

\bibitem{lions}
H.~Berestycki and P.-L. Lions, \emph{Nonlinear scalar field equations. {I}.
  {E}xistence of a ground state}, Arch. Rational Mech. Anal. \textbf{82}
  (1983), no.~4, 313--345. \MR{695535}

\bibitem{Boussinesq}
J.~Boussinesq, \emph{Th\'eorie des ondes et des remous qui se propagent le long
  d'un canal rectangulaire horizontal, en communiquant au liquide contenu dans
  ce canal des vitesses sensiblement pareilles de la surface au fond}, J. Math.
  Pures Appl. (2) \textbf{17} (1872), 55--108. \MR{3363411}

\bibitem{Brezis}
H.~Brezis, \emph{Functional analysis, {S}obolev spaces and partial differential
  equations}, Universitext, Springer, New York, 2011. \MR{2759829}

\bibitem{brezis-lieb}
H.~Brezis and E.~Lieb, \emph{A relation between pointwise convergence of
  functions and convergence of functionals}, Proc. Amer. Math. Soc. \textbf{88}
  (1983), no.~3, 486--490. \MR{699419}

\bibitem{chen-zou}
Z.~Chen and W.~Zou, \emph{An optimal constant for the existence of least energy
  solutions of a coupled {S}chr\"odinger system}, Calc. Var. Partial
  Differential Equations \textbf{48} (2013), no.~3-4, 695--711. \MR{3116028}

\bibitem{cof}
C.~V. Coffman, \emph{Uniqueness of the ground state solution for {$\Delta
  u-u+u^{3}=0$}\ and a variational characterization of other solutions}, Arch.
  Rational Mech. Anal. \textbf{46} (1972), 81--95. \MR{0333489}

\bibitem{c-fract}
E.~Colorado, \emph{Existence results for some systems of coupled fractional
  nonlinear {S}chr\"odinger equations}, Recent trends in nonlinear partial
  differential equations. {II}. {S}tationary problems, Contemp. Math., vol.
  595, Amer. Math. Soc., Providence, RI, 2013, pp.~135--150. \MR{3155969}

\bibitem{c}
\bysame, \emph{Positive solutions to some systems of coupled nonlinear
  {S}chr\"odinger equations}, Nonlinear Anal. \textbf{110} (2014), 104--112.
  \MR{3259736}

\bibitem{c2}
\bysame, \emph{Existence of bound and ground states for a system of coupled
  nonlinear {S}chr\"odinger-{K}d{V} equations}, C. R. Math. Acad. Sci. Paris
  \textbf{353} (2015), no.~6, 511--516. \MR{3348984}

\bibitem{c3}
\bysame, \emph{On the existence of bound and ground states for some coupled
  nonlinear {S}chr\"odinger-{K}orteweg-de {V}ries equations}, Adv. Nonlinear
  Anal. \textbf{0} (2016), no.~0, DOI: 10.1515/anona--2015--0181.

\bibitem{cl}
A.~J. Corcho and F.~Linares, \emph{Well-posedness for the
  {S}chr\"odinger-{K}orteweg-de {V}ries system}, Trans. Amer. Math. Soc.
  \textbf{359} (2007), no.~9, 4089--4106. \MR{2309177}

\bibitem{Crighton}
D.~G. Crighton, \emph{Applications of {K}d{V}}, Acta Appl. Math. \textbf{39}
  (1995), no.~1-3, 39--67, KdV '95 (Amsterdam, 1995). \MR{1329553}

\bibitem{fl}
D.~G. de~Figueiredo and O.~Lopes, \emph{Solitary waves for some nonlinear
  {S}chr\"odinger systems}, Ann. Inst. H. Poincar\'e Anal. Non Lin\'eaire
  \textbf{25} (2008), no.~1, 149--161. \MR{2383083}

\bibitem{dfo}
J.-P. Dias, M.~Figueira, and F.~Oliveira, \emph{Existence of bound states for
  the coupled {S}chr\"odinger-{K}d{V} system with cubic nonlinearity}, C. R.
  Math. Acad. Sci. Paris \textbf{348} (2010), no.~19-20, 1079--1082.
  \MR{2735011}

\bibitem{dfo2}
\bysame, \emph{Well-posedness and existence of bound states for a coupled
  {S}chr\"odinger-g{K}d{V} system}, Nonlinear Anal. \textbf{73} (2010), no.~8,
  2686--2698. \MR{2674102}

\bibitem{Dodd}
R.~K. Dodd, J.~C. Eilbeck, J.~D. Gibbon, and H.~C. Morris, \emph{Solitons and
  nonlinear wave equations}, Academic Press, Inc. [Harcourt Brace Jovanovich,
  Publishers], London-New York, 1982. \MR{696935}

\bibitem{Drazin}
P.~G. Drazin and R.~S. Johnson, \emph{Solitons: an introduction}, Cambridge
  Texts in Applied Mathematics, Cambridge University Press, Cambridge, 1989.
  \MR{985322}

\bibitem{eke}
I.~Ekeland, \emph{On the variational principle}, J. Math. Anal. Appl.
  \textbf{47} (1974), 324--353. \MR{0346619}

\bibitem{Wakil}
S.~A. El-Wakil, E.~M. Abulwafa, E.~K. El-shewy, and A.~A. Mahmoud,
  \emph{Time-fractional kdv equation for electron-acoustic waves in plasma of
  cold electron and two different temperature isothermal ions}, Astrophysics
  and Space Science \textbf{333} (2011), no.~1, 269--276.

\bibitem{Evans}
L.~C. Evans, \emph{Partial differential equations}, Graduate Studies in
  Mathematics, vol.~19, American Mathematical Society, Providence, RI, 1998.
  \MR{1625845}

\bibitem{Falkovich}
G.~Falkovich, \emph{Fluid mechanics}, Cambridge University Press, Cambridge,
  2011, A short course for physicists. \MR{2663889}

\bibitem{federer}
H.~Federer, \emph{Curvature measures}, Trans. Amer. Math. Soc. \textbf{93}
  (1959), 418--491. \MR{0110078}

\bibitem{federerlibro}
\bysame, \emph{Geometric measure theory}, Die Grundlehren der mathematischen
  Wissenschaften, Band 153, Springer-Verlag New York Inc., New York, 1969.
  \MR{0257325}

\bibitem{fo}
M.~Funakoshi and M.~Oikawa, \emph{The resonant interaction between a long
  internal gravity wave and a surface gravity wave packet}, J. Phys. Soc. Japan
  \textbf{52} (1983), no.~6, 1982--1995. \MR{710730}

\bibitem{gmp}
V.~A. Galaktionov, E.~L. Mitidieri, and S.~I. Pohozaev, \emph{Blow-up for
  higher-order parabolic, hyperbolic, dispersion and {S}chr\"odinger
  equations}, Monographs and Research Notes in Mathematics, CRC Press, Boca
  Raton, FL, 2015. \MR{3362689}

\bibitem{Gardner}
C.~S. Gardner, J.~M. Greene, M.~D. Kruskal, and R.~M. Miura, \emph{Method for
  solving the korteweg-devries equation}, Phys. Rev. Lett. \textbf{19} (1967),
  1095--1097.

\bibitem{Griffiths}
D.~J. Griffiths, \emph{Introduction to quantum mechanics}, Pearson Education
  India, 2005.

\bibitem{Xiaoyi}
X.~Guo and M.~Xu, \emph{Some physical applications of fractional
  {S}chr\"odinger equation}, J. Math. Phys. \textbf{47} (2006), no.~8, 082104,
  9. \MR{2258580}

\bibitem{Gurevich}
A~Gurevich, \emph{Nonlinear phenomena in the ionosphere}, vol.~10, Springer
  Science \& Business Media, 2012.

\bibitem{Hereman}
W.~Hereman, \emph{Shallow water waves and solitary waves}, Mathematics of
  complexity and dynamical systems. {V}ols. 1--3, Springer, New York, 2012,
  pp.~1520--1532. \MR{3220768}

\bibitem{iko}
N.~Ikoma, \emph{Uniqueness of positive solutions for a nonlinear elliptic
  system}, NoDEA Nonlinear Differential Equations Appl. \textbf{16} (2009),
  no.~5, 555--567. \MR{2551904}

\bibitem{itanaka}
N.~Ikoma and K.~Tanaka, \emph{A local mountain pass type result for a system of
  nonlinear {S}chr\"odinger equations}, Calc. Var. Partial Differential
  Equations \textbf{40} (2011), no.~3-4, 449--480. \MR{2764914}

\bibitem{Khismatullin}
D.~B. {Khismatullin} and I.~S. {Akhatov}, \emph{{Sound-ultrasound interaction
  in bubbly fluids: Theory and possible applications}}, Physics of Fluids
  \textbf{13} (2001), 3582--3598.

\bibitem{Korteweg}
D.~J. Korteweg and G.~de~Vries, \emph{On the change of form of long waves
  advancing in a rectangular canal, and on a new type of long stationary
  waves}, Philos. Mag. (5) \textbf{39} (1895), no.~240, 422--443. \MR{3363408}

\bibitem{kw}
M.~K. Kwong, \emph{Uniqueness of positive solutions of {$\Delta u-u+u^p=0$} in
  {${\bf R}^n$}}, Arch. Rational Mech. Anal. \textbf{105} (1989), no.~3,
  243--266. \MR{969899}

\bibitem{Las2000}
N.~Laskin, \emph{Fractional quantum mechanics and {L}\'evy path integrals},
  Phys. Lett. A \textbf{268} (2000), no.~4-6, 298--305. \MR{1755089}

\bibitem{Las2002}
\bysame, \emph{Fractional {S}chr\"odinger equation}, Phys. Rev. E (3)
  \textbf{66} (2002), no.~5, 056108, 7. \MR{1948569}

\bibitem{Lax}
P.~D. Lax, \emph{Integrals of nonlinear equations of evolution and solitary
  waves}, Comm. Pure Appl. Math. \textbf{21} (1968), 467--490. \MR{0235310}

\bibitem{linwei}
T.-C. Lin and J.~Wei, \emph{{G}round state of {$N$} coupled nonlinear
  {S}chr\"odinger equations in {${\bf R}^n$}, {$n\leq3$}}, Comm. Math. Phys.
  \textbf{277} (2008), no.~2, 573--576. \MR{2358296}

\bibitem{Lions-JFA82}
P.-L. Lions, \emph{Sym\'etrie et compacit\'e dans les espaces de {S}obolev}, J.
  Funct. Anal. \textbf{49} (1982), no.~3, 315--334. \MR{683027}

\bibitem{lions2}
\bysame, \emph{The concentration-compactness principle in the calculus of
  variations. {T}he locally compact case. {II}}, Ann. Inst. H. Poincar\'e Anal.
  Non Lin\'eaire \textbf{1} (1984), no.~4, 223--283. \MR{778974}

\bibitem{lz}
C.~Liu and Y.~Zheng, \emph{On soliton solutions to a class of
  {S}chr\"odinger-{K}d{V} systems}, Proc. Amer. Math. Soc. \textbf{141} (2013),
  no.~10, 3477--3484. \MR{3080170}

\bibitem{liu-wang}
Z.~Liu and Z.-Q. Wang, \emph{Ground states and bound states of a nonlinear
  {S}chr\"odinger system}, Adv. Nonlinear Stud. \textbf{10} (2010), no.~1,
  175--193. \MR{2574384}

\bibitem{mmp}
L.~A. Maia, E.~Montefusco, and B.~Pellacci, \emph{Positive solutions for a
  weakly coupled nonlinear {S}chr\"odinger system}, J. Differential Equations
  \textbf{229} (2006), no.~2, 743--767. \MR{2263573}

\bibitem{meyers}
N.~G. Meyers and J.~Serrin, \emph{{$H=W$}}, Proc. Nat. Acad. Sci. U.S.A.
  \textbf{51} (1964), 1055--1056. \MR{0164252}

\bibitem{Nakamura}
Y.~Nakamura, H.~Bailung, and P.~K. Shukla, \emph{Observation of ion-acoustic
  shocks in a dusty plasma}, Phys. Rev. Lett. \textbf{83} (1999), 1602--1605.

\bibitem{Nehari}
Z.~Nehari, \emph{On a class of nonlinear second-order differential equations},
  Trans. Amer. Math. Soc. \textbf{95} (1960), 101--123. \MR{0111898}

\bibitem{Nehari2}
\bysame, \emph{Characteristic values associated with a class of non-linear
  second-order differential equations}, Acta Math. \textbf{105} (1961),
  141--175. \MR{0123775}

\bibitem{Newell}
A.~C. Newell, \emph{Solitons in mathematics and physics}, CBMS-NSF Regional
  Conference Series in Applied Mathematics, vol.~48, Society for Industrial and
  Applied Mathematics (SIAM), Philadelphia, PA, 1985. \MR{847245}

\bibitem{Pitaevskii}
L.~Pitaevskii and S.~Stringari, \emph{Bose-{E}instein condensation},
  International Series of Monographs on Physics, vol. 116, The Clarendon Press,
  Oxford University Press, Oxford, 2003. \MR{2012737}

\bibitem{polya}
G.~P{\'o}lya and G.~Szeg{\"o}, \emph{Inequalities for the capacity of a
  condenser}, Amer. J. Math. \textbf{67} (1945), 1--32. \MR{0011871}

\bibitem{rep}
D.~Repov{\v{s}}, \emph{Stationary waves of {S}chr\"odinger-type equations with
  variable exponent}, Anal. Appl. (Singap.) \textbf{13} (2015), no.~6,
  645--661. \MR{3376930}

\bibitem{Riesz}
M.~Riesz, \emph{L'int\'egrale de {R}iemann-{L}iouville et le probl\`eme de
  {C}auchy}, Acta Math. \textbf{81} (1949), 1--223. \MR{0030102}

\bibitem{Russell}
J.~S. Russell, \emph{Report on waves}, in: Report of the $14^{th}$ meeting,
  Brit. Assoc. Adv. Marray, Lomdon (1844), 311--390.

\bibitem{Sard}
A.~Sard, \emph{The measure of the critical values of differentiable maps},
  Bull. Amer. Math. Soc. \textbf{48} (1942), 883--890. \MR{0007523}

\bibitem{S}
E.~Schr\"odinger, \emph{An undulatory theory of the mechanics of atoms and
  molecules}, Phys. Rev. \textbf{28} (1926), 1049--1070.

\bibitem{Malomed}
A.~Scott et~al., \emph{Encyclopedia of nonlinear science}, Routledge, Taylor \&
  Francis Group, New York, 2006.

\bibitem{sirakov}
B.~Sirakov, \emph{Least energy solitary waves for a system of nonlinear
  {S}chr\"odinger equations in {$\Bbb R^n$}}, Comm. Math. Phys. \textbf{271}
  (2007), no.~1, 199--221. \MR{2283958}

\bibitem{talenti}
G~Talenti, \emph{Best constant in {S}obolev inequality}, Ann. Mat. Pura Appl.
  (4) \textbf{110} (1976), 353--372. \MR{0463908}

\bibitem{wy}
J.~Wei and W.~Yao, \emph{Uniqueness of positive solutions to some coupled
  nonlinear {S}chr\"odinger equations}, Commun. Pure Appl. Anal. \textbf{11}
  (2012), no.~3, 1003--1011. \MR{2968605}

\bibitem{Zabusky}
N.~J. Zabusky and M.~D. Kruskal, \emph{Interaction of ``solitons'' in a
  collisionless plasma and the recurrence of initial states}, Phys. Rev. Lett.
  \textbf{15} (1965), 240--243.

\end{thebibliography}

\printindex
\end{document}